\documentclass[onecolumn]{autart}    

\PassOptionsToPackage{svgnames}{xcolor}

\usepackage{titlesec}

\usepackage{graphicx}
\graphicspath{ {Images/} }

\usepackage{hyperref}
\hypersetup{colorlinks=true,
            linkcolor=blue,
            citecolor=[rgb]{1,0.25,0},
            urlcolor=blue,
            }

\usepackage{float}           
\usepackage{caption}
\usepackage{subcaption}
\usepackage{setspace}
\usepackage{tocloft}

\usepackage{bigints}
\usepackage{amsmath}
\usepackage{amsfonts}
\usepackage{amssymb}
\usepackage{amsbsy}
\usepackage{mathtools, nccmath, textcomp}
\usepackage{multirow}

\usepackage{enumitem}
\usepackage{booktabs}

\usepackage{tabularx}
\usepackage{xcolor}

\usepackage{tikz}
\usetikzlibrary{positioning}
\usetikzlibrary{shapes.geometric,arrows.meta,chains,positioning,shapes.symbols}
\usetikzlibrary{shadings,shadows}

\usepackage{fixmath}
\usepackage{bm}
\usepackage{mathrsfs}
\usepackage{longtable}
\usepackage{soul}
\usepackage{lmodern}
\usepackage[T1]{fontenc}
\usepackage{multicol}

\usepackage{tcolorbox}
\tcbuselibrary{skins,breakable}

\usepackage{bbm}

\font\f=cmbsy7 at2.5pt
\def\ointp{\mathop
   {\oint \kern-5.55pt\raise1.8pt\hbox{%
    \rlap{\f\char"5E}\kern.1pt\rlap{\f\char"5E}} \hspace{1mm}  }\limits
    }
    
\font\f=cmbsy7 at2.5pt
\def\ointn{\mathop
   {\oint \kern-5.45pt\raise1.8pt\hbox{%
    \rlap{\f\char"5F}\kern.1pt\rlap{\f\char"5F}} \hspace{1mm}  }\limits
    }

    {\endtcolorbox}

    {\endtcolorbox}

    {\endtcolorbox}

\makeatletter
\newlength\xvec@height%
\newlength\xvec@depth%
\newlength\xvec@width%
\newcommand{\xvec}[2][]{%
  \ifmmode%
    \settoheight{\xvec@height}{$#2$}%
    \settodepth{\xvec@depth}{$#2$}%
    \settowidth{\xvec@width}{$#2$}%
  \else%
    \settoheight{\xvec@height}{#2}%
    \settodepth{\xvec@depth}{#2}%
    \settowidth{\xvec@width}{#2}%
  \fi%
  \def\xvec@arg{#1}%
  \def\xvec@dd{:}%
  \def\xvec@d{.}%
  \raisebox{.2ex}{\raisebox{\xvec@height}{\rlap{%
    \kern.05em
    \begin{tikzpicture}[scale=1]
    \pgfsetroundcap
    \draw (.05em,0)--(\xvec@width-.05em,0);
    \draw (\xvec@width-.05em,0)--(\xvec@width-.15em, .075em);
    \draw (\xvec@width-.05em,0)--(\xvec@width-.15em,-.075em);
    \ifx\xvec@arg\xvec@d%
      \fill(\xvec@width*.45,.5ex) circle (.5pt);%
    \else\ifx\xvec@arg\xvec@dd%
      \fill(\xvec@width*.30,.5ex) circle (.5pt);%
      \fill(\xvec@width*.65,.5ex) circle (.5pt);%
    \fi\fi%
    \end{tikzpicture}%
  }}}%
  #2%
}
\makeatother

\def\onedot{$\mathsurround0pt\ldotp$}
\def\cdddot#1{
  \mathbin{\vcenter{\baselineskip.67ex
    \hbox{\onedot}\hbox{\onedot}\hbox{\onedot}%
  }}%
}

\makeatletter
\def\underbracex#1#2{\mathop{\vtop{\m@th\ialign{##\crcr
   $\hfil\displaystyle{#2}\hfil$\crcr
   \noalign{\kern3\p@\nointerlineskip}%
   #1\crcr\noalign{\kern3\p@}}}}\limits}

\def\upbracefilla{$\m@th \setbox\z@\hbox{$\braceld$}%
  \bracelu\leaders\vrule \@height\ht\z@ \@depth\z@\hfill 
\kern\p@\vrule \@width\p@\kern\p@\vrule \@width\p@\kern\p@\vrule \@width\p@
$}

\def\upbracefillb{$\m@th \setbox\z@\hbox{$\braceld$}%
\vrule \@width\p@\kern\p@\vrule \@width\p@\kern\p@\vrule \@width\p@\kern\p@
 \leaders\vrule \@height\ht\z@ \@depth\z@\hfill\bracerd
  \braceld\leaders\vrule \@height\ht\z@ \@depth\z@\hfill
\kern\p@\vrule \@width\p@\kern\p@\vrule \@width\p@\kern\p@\vrule \@width\p@
$}

\def\upbracefillc{$\m@th \setbox\z@\hbox{$\braceld$}%
\vrule \@width\p@\kern\p@\vrule \@width\p@\kern\p@\vrule \@width\p@\kern\p@
\leaders\vrule \@height\ht\z@ \@depth\z@\hfill
\kern\p@\vrule \@width\p@\kern\p@\vrule \@width\p@\kern\p@\vrule \@width\p@
$}

\def\upbracefilld{$\m@th \setbox\z@\hbox{$\braceld$}%
\vrule \@width\p@\kern\p@\vrule \@width\p@\kern\p@\vrule \@width\p@\kern\p@
 \leaders\vrule \@height\ht\z@ \@depth\z@\hfill\braceru$}

\def\upbracefillbd{$\m@th \setbox\z@\hbox{$\braceld$}%
\vrule \@width\p@\kern\p@\vrule \@width\p@\kern\p@\vrule \@width\p@\kern\p@
\bracerd\braceld
 \leaders\vrule \@height\ht\z@ \@depth\z@\hfill\braceru$}

\makeatother

\renewcommand{\vec}[1]{\xvec[]{#1}}

\newcommand*{\B}[1]{\ifmmode\bm{#1}\else\textbf{#1}\fi}
\newcommand{\HS}[1]{\hspace{#1 mm}}

\newcommand{\ub}[2]{\underbrace{#1}_{#2}}

\newcommand{\p}{\partial}
\newcommand{\var}{\delta}
\newcommand{\dint}{\displaystyle\int}

\newcommand{\mtx}[1]{\mbox{\footnotesize$#1$}}

\newcommand{\R}[1]{\mathbb{R}^{#1}}

\newcommand{\F}{\mathcal{F}}

\newcommand{\x}[1]{\zeta_{#1}}
\newcommand{\X}{\mathtt{x}}
\newcommand{\CX}{\mathtt{X}}
\newcommand{\sx}{\mathtt{s}}
\newcommand{\Csx}{\mathtt{S}}

\newcommand{\Grad}{\mbox{Grad}}
\newcommand{\Div}{\mbox{Div}}

\newcommand{\q}{\epsilon}

\newcommand{\bsm}[1]{ \begin{bsmallmatrix} #1 \end{bsmallmatrix} }

\newcommand{\blanco}[1]{\textcolor{white}{#1}}

\def\doubleunderline#1{\underline{\underline{#1}}}
\def\rddots#1{\cdot^{\cdot^{\cdot^{#1}}}}

\theoremstyle{definition} 

\theoremstyle{definition} 
\newtheorem{assumption}{\normalfont\textbf{Assumption}}
\theoremstyle{definition} 
\newtheorem{definition}{\normalfont\textbf{Definition}}
\theoremstyle{definition} 
\newtheorem{theorem}{\normalfont\textbf{Theorem}}
\theoremstyle{definition} 
\newtheorem{lemma}{\normalfont\textbf{Lemma}}
\theoremstyle{definition} 
\newtheorem{proposition}{\normalfont\textbf{Proposition}}
\theoremstyle{definition} 
\newtheorem{corollary}{\normalfont\textbf{Corollary}}
\theoremstyle{definition} 
\newtheorem{remark}{\normalfont\textbf{Remark}}

\newenvironment{proof}[1][Proof]{\textbf{#1.} }{\ \rule{0.5em}{0.5em}}

\begin{document}

\begin{frontmatter}

\title{Port-Hamiltonian modeling of multidimensional flexible mechanical structures defined by linear elastic relations\thanksref{footnoteinfo}} 

\thanks[footnoteinfo]{Corresponding author C.Ponce. Tel. +33 751145392.}

\author[First,Second]{Cristobal Ponce}\ead{cristobal.ponces@usm.cl,cristobal.ponce@femto-st.fr},  
\author[First]{Yongxin Wu}\ead{yongxin.wu@femto-st.fr}, 
\author[First]{Yann Le Gorrec}\ead{legorrec@femto-st.fr},
\author[Second]{Hector Ramirez}\ead{hector.ramireze@usm.cl}

\address[First]{Supmicrotech, FEMTO-ST AS2M department, 24 rue Alain Savary, F-25000 Besancon, France. }                                                
\address[Second]{Department of Electronic Engineering, Universidad Tecnica Federico Santa Maria, Av. Espana 1680, Valparaiso, Chile. }

\begin{abstract}                          
This article presents a systematic methodology for modeling a class of flexible multidimensional mechanical structures defined by linear elastic relations that directly allows to obtain their infinite-dimensional port-Hamiltonian representation. The approach is restricted to systems based on a certain class of kinematic assumptions. However this class encompasses a wide range of models currently available in the literature, such as $\ell$-dimensional elasticity models (with $\ell$ = 1,2,3), vibrating strings, torsion in circular bars, classical beam and plate models, among others. The methodology is based on Hamilton's principle for a continuum medium which allows defining the energy variables of the port-Hamiltonian system, and also on a generalization of the integration by parts theorem, which allows characterizing the skew-adjoint differential operator and boundary inputs and boundary outputs variables. To illustrate this method, the plate modeling process based on Reddy's third-order shear deformation theory is presented as an example. To the best of our knowledge, this is the first time that an infinite-dimensional port-Hamiltonian representation of this system is presented in the literature.
\end{abstract}


\end{frontmatter}

\section{Introduction}\label{sec:Introduction}

\noindent The modeling of mechanical systems is traditionally derived through two fundamental approaches: Newton's method and d'Alembert's principle. From the Newtonian perspective, the equations of motion are obtained from the forces acting on the system, while in d'Alembert's principle, also known as the principle of virtual work, the equations of motion are obtained through an energetic approach, where the concept of virtual work defined as the dot product between an applied force and a small virtual displacement that does not violate the constraints imposed on the system is used. Applying d'Alembert's principle to a mechanical system leads to the well-known Euler-Lagrange (E-L) equations, where the Lagrangian functional that depends on the energy of the system is the key quantity used. This approach offers a more compact description of the system, which allows the use of variational techniques to obtain approximate and analytical solutions \cite{goldstein2002classical,lanczos2012variational}. Therefore, the E-L equations can also be obtained from the most fundamental principles of physics, such as that of action, which establish that the real trajectory followed by a mechanical system is the one for which the integral of the action, defined as the integral of the Lagrangian along the trajectory, is minimal \cite{landau2013course}. The connection between the E-L equations and Hamiltonian mechanics is evidenced through the Legendre transform and the Poisson structure \cite{arnol2013mathematical,hassani2013mathematical}. In the classical Hamiltonian approach, the equations describe the temporal evolution of the system through an antisymmetric matrix called symplectic matrix, which is nothing more than a representation of the set of transformations that leaves the Poisson brackets invariant \cite{hassani2013mathematical}. In addition, the Poisson structure defines the algebra of the brackets and allows the study of the dynamic properties of the system \cite{de2011methods}. In the present work we focus on port-Hamiltonian systems (PHS), which are a class of geometrically defined open physical systems with ports that came from the network modeling of physical systems through bond-graphs, which are associated with a Dirac structure that generalizes the Poisson structure of the classical Hamiltonian approach \cite{maschke1993port,van2002hamiltonian}. 
\noindent Since finite-dimensional PHS was introduced in \cite{maschke1993port} and later extended to infinite dimensions in \cite{van2002hamiltonian}, the usual way to find the port-Hamiltonian (PH) representation of existing infinite-dimensional models is through the adequate definition of the energy variables, in such a way that it allows writing the dynamics of the system with respect to a skew-adjoint differential operator and the variational derivatives of the Hamiltonian functional. Despite the fact that this procedure requires at least in some cases intuition to choose the correct set of state variables in such a way that the skew-adjoint differential operator arises,  this procedure has been effective for rewriting classical elasticity models within the PH framework, for example the vibrating string model \cite{talasila2002wave}, Euler-Bernoulli beam \cite{nishida2004higher},  Timoshenko beam \cite{macchelli2004modeling}, classical plate models \cite{macchelli2005port,brugnoli2019port,brugnoli2019port2}, among others. In  \cite{mattioni2021modelling} it is shown that the E-L equations from which dynamic models of flexible mechanisms arise can be obtained by applying the principle of least action to a system with a properly defined Lagrangian functional. Then, through a suitable change of variables it can be rewritten as an infinite-dimensional PHS. One of the works that treats linear elastodynamics in a more general way is \cite{brugnoli2020port}, where it is shown that defining energy variables as vector or tensor fields leads to PH models associated with coordinate-free skew-adjoint differential operators, such as those written in terms of gradient and divergence operators. It should be noted that the starting point of all the previously cited works is to know some other model representation, e.g. the Lagrangian, but one must not lose sight of the fact that these models respond to a series of kinematic assumptions and constitutive laws, and that they may arise from E-L equations applying variational principles such as the d'Alembert's principle, or the least action from Hamilton's principle. In this line, \cite{nishida2005formal} shows a one-to-one correspondence between E-L equations and field port-Lagrangian systems on variational complexes of jet bundles. When the Lagrangian functional is known, a systematic procedure for the transformation to a port-Lagrangian model is presented. Finally, other works such as \cite{yoshimura2006dirac,nishida2010variational,nishida2012implicit,schoberl2014jet}   explore the relationships between field equations from variational structures and PH representations, the connection between Poisson and Dirac structures, and also propose adapted variational methods such as the Hamilton–Pontryagin principle, that internally includes the Legendre transform and allows directly obtaining the port-Lagrangian representations.\\[-3mm]

\noindent This work proposes a systematic modeling methodology to directly obtain the PH representation of flexible mechanical systems characterized by the Stokes-Dirac structure. We restrict ourselves to systems derived by linear elasticity whose kinematic assumptions fit with a certain class of displacement field. The main contributions of this work are twofold: a characterization of the integration by parts theorem for a class of higher order multidimensional differential operators, and the modeling methodology itself. The main differences with respect to other works are: Firstly, the definition of the energy and co-energy variables are explicitly provided. Secondly, these variables allows us to construct a priori the structure of the skew-adjoint differential operator. Thirdly, given that the methodology is based on kinematic assumptions, it only can be employed to create models from scratch, without need of previous models. Lastly, the boundary variables are explicitly defined with respect to the structure of the differential operator and the co-energy variables. This paper is structured as follows. Section 2 briefly presents the framework of infinite-dimensional PHS and a motivation example for the modeling of flexible structures. Section 3 presents the results that support the proposed methodology. To illustrate the potential of the results, Section 4 presents a “non-classical” plate model based on the Reddy's third-order shear deformation theory, which to the best of our knowledge, has never been formulated as an infinite-dimensional PHS. Finally, in Section 5 the conclusions and perspectives for future work are given.


\section{Background}

\noindent This section presents the theory that supports the rest of the article. Subsection \ref{sec:intro_lin_elast} begins by presenting the Hamilton's principle for a continuum elastic body. Then,  Subsection \ref{ssec:PHS} defines infinite-dimensional PHS and shows how they are associated with Stokes-Dirac structures. Lastly, in subsection \ref{ssec:motivation_example} a motivation example is used to define notations and highlight the main difficulties in deriving an appropriate PH model for a flexible structure.

\subsection{Hamilton's principle for a continuum elastic body} \label{sec:intro_lin_elast}

\noindent The branch of physics that studies the relationship between external forces, internal stresses, and deformation of an elastic body is known as elasticity. For the mathematical description of elastic problems, particularly those based on linear elasticity, it is necessary to introduce the concepts of the displacement field, stress and strain tensors. For a more in-depth review of elasticity and all related concepts, the reader can refer to \cite{reddy2013introduction,reddy2017energy}. 
Let $\Omega_T \subset \R{3}$ be the total volume of an elastic body in space, the displacement field $\textbf{u}=\textbf{u}(\CX,t) \in \R{3}$ assigns to each point $\CX_p \in \Omega_T$ of the body a displacement vector u$(\CX_p,t)$, that specifies its current position in the deformed configuration regarding an initial undeformed configuration (see Figure \ref{fig:u_reference}).

\begin{figure}[h]
\begin{center}
	\tikzset{every picture/.style={line width=0.75pt}} 

\scalebox{0.8}{
\begin{tikzpicture}[x=0.75pt,y=0.75pt,yscale=-0.8,xscale=0.95]


\draw    (182,220) -- (182,185) ;
\draw [shift={(182,183)}, rotate = 90] [fill={rgb, 255:red, 0; green, 0; blue, 0 }  ][line width=0.08]  [draw opacity=0] (12,-3) -- (0,0) -- (12,3) -- cycle    ;
\draw    (182,220) -- (168.32,234.3) -- (161.38,241.55) ;
\draw [shift={(160,243)}, rotate = 313.73] [fill={rgb, 255:red, 0; green, 0; blue, 0 }  ][line width=0.08]  [draw opacity=0] (12,-3) -- (0,0) -- (12,3) -- cycle    ;
\draw    (182,220) -- (216,220) ;
\draw [shift={(218,220)}, rotate = 180] [fill={rgb, 255:red, 0; green, 0; blue, 0 }  ][line width=0.08]  [draw opacity=0] (12,-3) -- (0,0) -- (12,3) -- cycle    ;
\draw   (217.88,197.06) .. controls (211.92,180.68) and (224.98,157.69) .. (247.05,145.7) .. controls (269.11,133.72) and (291.83,137.28) .. (297.78,153.66) .. controls (303.74,170.03) and (290.68,193.03) .. (268.61,205.01) .. controls (246.55,217) and (223.83,213.44) .. (217.88,197.06) -- cycle ;
\draw   (417.77,141.06) .. controls (436.86,142.33) and (463.67,160.28) .. (477.65,181.14) .. controls (491.63,202) and (487.48,217.87) .. (468.39,216.6) .. controls (449.31,215.32) and (422.5,197.38) .. (408.52,176.52) .. controls (394.54,155.66) and (398.68,139.78) .. (417.77,141.06) -- cycle ;
\draw  [fill={rgb, 255:red, 7; green, 0; blue, 0 }  ,fill opacity=1 ] (287.97,170.97) .. controls (287.95,169.86) and (287.04,168.98) .. (285.93,169) .. controls (284.83,169.02) and (283.95,169.93) .. (283.97,171.03) .. controls (283.99,172.14) and (284.9,173.02) .. (286,173) .. controls (287.1,172.98) and (287.98,172.07) .. (287.97,170.97) -- cycle ;
\draw  [fill={rgb, 255:red, 7; green, 0; blue, 0 }  ,fill opacity=1 ] (462.97,191.97) .. controls (462.95,190.86) and (462.04,189.98) .. (460.93,190) .. controls (459.83,190.02) and (458.95,190.93) .. (458.97,192.03) .. controls (458.99,193.14) and (459.9,194.02) .. (461,194) .. controls (462.1,193.98) and (462.98,193.07) .. (462.97,191.97) -- cycle ;
\draw [color={rgb, 255:red, 209; green, 34; blue, 9 }  ,draw opacity=1 ]   (287.97,170.97) -- (456.98,191.79) ;
\draw [shift={(458.97,192.03)}, rotate = 187.02] [fill={rgb, 255:red, 209; green, 34; blue, 9 }  ,fill opacity=1 ][line width=0.08]  [draw opacity=0] (12,-3) -- (0,0) -- (12,3) -- cycle    ;

\draw (390,112) node [anchor=north west][inner sep=0.75pt]  [font=\small]  {$\text{configuration}$};
\draw (225,111) node [anchor=north west][inner sep=0.75pt]  [font=\small]  {$\text{configuration}$};
\draw (284,208) node [anchor=north west][inner sep=0.75pt]  [font=\small]  {$\text{(Displacement\ vector)}$};
\draw (398,96) node [anchor=north west][inner sep=0.75pt]  [font=\small]  {$\ \!\!\text{Deformed}$};
\draw (222,95) node [anchor=north west][inner sep=0.75pt]  [font=\small]  {$\ \text{Undeformed}$};
\draw (388,69) node [anchor=north west][inner sep=0.75pt]  [font=\small,color={rgb, 255:red, 21; green, 95; blue, 194 }  ,opacity=1 ]  {$\mathnormal{Time:} \ t >t_{0}$};
\draw (233,69) node [anchor=north west][inner sep=0.75pt]  [font=\small,color={rgb, 255:red, 21; green, 95; blue, 194 }  ,opacity=1 ]  {$\mathnormal{Time:} \ t_{0}$};
\draw (322,183) node [anchor=north west][inner sep=0.75pt]    {$\mathnormal{u(}\mathtt{X_{p}}\mathnormal{,t)}$};
\draw (450,166) node [anchor=north west][inner sep=0.75pt]    {$\mathtt{X'_{P}}$};
\draw (260,153) node [anchor=north west][inner sep=0.75pt]    {$\;\,\mathtt{X_{P}}$};
\draw (226,187) node [anchor=north west][inner sep=0.75pt]    {$\Omega_T $};
\draw (165,157) node [anchor=north west][inner sep=0.75pt]   [align=left] {$ \;\; \displaystyle \zeta _{3}$};
\draw (141,239) node [anchor=north west][inner sep=0.75pt]   [align=left] {$\;\displaystyle \zeta _{2}$};
\draw (215,215) node [anchor=north west][inner sep=0.75pt]   [align=left] {$\;\displaystyle \zeta _{1}$};

\end{tikzpicture}
}
\end{center}
\vspace{-10mm}\caption{Displacement field and configurations.}
\label{fig:u_reference}
\end{figure}
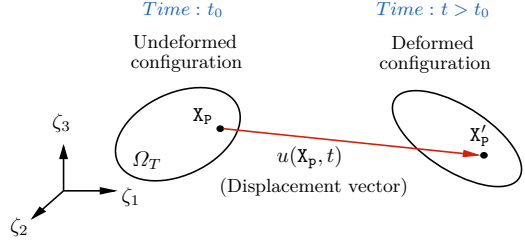

\noindent In case of small deformation ($\|\nabla \textbf{u}(\CX,t)\|<<1$), the strain tensor $\underline{\varepsilon}(\CX,t) \in \R{3\times 3}$, which is a measure of deformation excluding rigid body displacements, is defined as
\begin{equation}
\underline{\varepsilon}(\CX,t) = \Grad(\textbf{u}(\CX,t)),
\label{eq:strain_tensor_lineal}
\end{equation}
where $\Grad(\cdot) = \frac{1}{2}[(\nabla (\cdot)) + (\nabla(\cdot))^\top]$ is the symmetric part of the tensor gradient operator. The stress tensor $\underline{\sigma}(\CX,t) \in \R{3\times 3}$ is obtained by means of the generalized Hooke's law given by
$$
\underline{\sigma}(\CX,t) = \doubleunderline{C}: \underline{\varepsilon}(\CX,t) 
$$
with $\doubleunderline{C} \in \R{3\times 3\times 3\times 3}$ a constant symmetric fourth-order constitutive tensor. 
\noindent We consider the linear elastic body with volume $\Omega_T \subset \R{3}$. Its boundary is defined by $\p\Omega_T = \p\Omega_{T_u} \cup \p\Omega_{T_\sigma}$ where $\p\Omega_{T_u}$ and $\p\Omega_{T_\sigma}$ denote the portion on which displacements and stresses are specified, respectively (see Figure \ref{fig:papa_3D}). The kinetic energy $T$, the elastic potential energy $U$, and the total external work $W_E$, which are defined as \cite{reddy2013introduction}:
\begin{align}
T= & \;\mfrac{1}{2}\int_{\Omega_T} \rho(\CX) \dot{\textbf{u}}(\CX,t)\cdot \dot{\textbf{u}}(\CX,t) \,d\CX  \label{eq:T_1}   \\
U = &\; \mfrac{1}{2}\int_{\Omega_T} \underline{\sigma}(\CX,t): \underline{\varepsilon}(\CX,t) \,d\CX  \label{eq:U_1} \\
W_E =  & -\int_{\Omega_T} \textbf{F}(\CX,t) \cdot \textbf{u}(\CX,t) \,d\CX - \int_{\p \Omega_{T_\sigma}} \!\!\! \textbf{t}(\Csx_\sigma,t)\cdot \textbf{u}(\Csx_\sigma,t) \,d\Csx_\sigma   \label{eq:We_1} 
\end{align}
where $\Csx_\sigma \in \p\Omega_{T_\sigma}$ and $\Csx_u \in \p\Omega_{T_u}$ are curvilinear coordinates along the boundary,  $\p_t \textbf{u}(\CX,t) = \dot{\textbf{u}}(\CX,t) \in \R{3}$ is the velocity vector, $\rho(\CX) \in \R{}$ is the density of the body, $\textbf{F}(\CX,t)\in \R{3}$ is the sum of all external body forces, and $\textbf{t}(\Csx_\sigma,t)\in \R{3}$ is the sum of all external surface forces (also called tractions). Hamilton's principle states that the true evolution $\textbf{u}=\textbf{u}(\CX,t)$ of a system, between two specific states $\textbf{u}_1 = \textbf{u}(\CX,t_1)$ and $\textbf{u}_2 = \textbf{u}(\CX,t_2)$ at two specific times $t_1$ and $t_2$ is a stationary point (a point where the variation is zero) of the action functional $S$, that is
\begin{equation}
\var S = \int_{t_1}^{t_2}(\var U + \var W_E - \var T)\, dt = 0.
\label{eq:Hamilton_principle_1}
\end{equation}
\noindent Defining as generalized coordinates the variables that define the configuration of the system and defining as configuration space the space generated by these coordinates, Hamilton's principle states that, as the system evolves, a path is traced through the configuration space, where the real path $\textbf{u}(\CX,t)$ taken by the system has a stationary action $(\var S = 0)$ under admissible small variations $(\var \textbf{u})$ in the configuration of the system. Thus, an admissible variation $\var \textbf{u}$, also called a virtual displacement, must be consistent with the essential boundary conditions (in $\p\Omega_{T_u}$) and satisfies
\begin{align}
\var\textbf{u}(\Csx_u,t)= 0 \mbox{ on } \p\Omega_{T_u} \mbox{ for all }t \label{eq:admisible_1}\\[1mm]
\var\textbf{u}(\CX,t_1)=\var\textbf{u}(\CX,t_2)=0 \mbox{ for all }\CX . \label{eq:admisible_2}
\end{align}
So with the definitions of kinetic energy $T$, elastic potential energy $U$ and external work $W_E$, in \eqref{eq:T_1}, \eqref{eq:U_1} and \eqref{eq:We_1}, respectively, equations \eqref{eq:Hamilton_principle_1} to  \eqref{eq:admisible_2} can be used to find the dynamic equations of the system.

\begin{figure}[t]
   \begin{minipage}[b]{.32\linewidth}
    \begin{center}
		\tikzset{every picture/.style={line width=0.75pt}} 

\scalebox{0.68}{
\begin{tikzpicture}[x=0.75pt,y=0.75pt,yscale=-0.5,xscale=0.5]

\draw  [dash pattern={on 4.5pt off 4.5pt}]  (388.87,309.8) .. controls (410.8,303.1) and (426.78,280.56) .. (428.05,247.9) .. controls (429.32,215.25) and (406.4,179.98) .. (389.84,168.23) .. controls (373.28,156.47) and (359.22,154.41) .. (325.41,155.19) ;
\draw  [fill={rgb, 255:red, 184; green, 233; blue, 134 }  ,fill opacity=0.34 ] (528,135.88) .. controls (560.92,145.63) and (586.61,143) .. (601.97,235.63) .. controls (617.33,328.25) and (479.06,267.31) .. (400.04,306.31) .. controls (321.02,345.31) and (195.91,420.87) .. (207.98,328.25) .. controls (220.06,235.63) and (254.09,191.99) .. (325.41,155.19) .. controls (342.32,146.47) and (359.72,139.83) .. (376.87,134.9) .. controls (432.07,119.04) and (502.88,128.44) .. (528,135.88) -- cycle ;
\draw  [fill={rgb, 255:red, 119; green, 155; blue, 255 }  ,fill opacity=0.34 ] (528,135.88) .. controls (560.92,145.63) and (586.61,143) .. (601.97,235.63) .. controls (617.33,328.25) and (479.06,267.31) .. (400.04,306.31) .. controls (366.52,322.86) and (322.33,274.03) .. (309.21,247.77) .. controls (305.77,237.45) and (297.55,215.24) .. (300.29,194.22) .. controls (305.39,168.1) and (311.4,170.86) .. (325.41,155.19) .. controls (342.32,146.47) and (359.72,139.83) .. (376.87,134.9) .. controls (432.07,119.04) and (502.88,128.44) .. (528,135.88) -- cycle ;
\draw [color={rgb, 255:red, 255; green, 0; blue, 0 }  ,draw opacity=1 ]   (301,190.25) -- (237.24,169.96) ;
\draw [shift={(235.33,169.35)}, rotate = 17.65] [fill={rgb, 255:red, 255; green, 0; blue, 0 }  ,fill opacity=1 ][line width=0.08]  [draw opacity=0] (12,-3) -- (0,0) -- (12,3) -- cycle    ;
\draw    (141,252.25) -- (141,217.25) ;
\draw [shift={(141,215.25)}, rotate = 90] [fill={rgb, 255:red, 0; green, 0; blue, 0 }  ][line width=0.08]  [draw opacity=0] (12,-3) -- (0,0) -- (12,3) -- cycle    ;
\draw    (141,252.25) -- (127.32,266.55) -- (120.38,273.8) ;
\draw [shift={(119,275.25)}, rotate = 313.73] [fill={rgb, 255:red, 0; green, 0; blue, 0 }  ][line width=0.08]  [draw opacity=0] (12,-3) -- (0,0) -- (12,3) -- cycle    ;
\draw    (141,252.25) -- (175,252.25) ;
\draw [shift={(177,252.25)}, rotate = 180] [fill={rgb, 255:red, 0; green, 0; blue, 0 }  ][line width=0.08]  [draw opacity=0] (12,-3) -- (0,0) -- (12,3) -- cycle    ;
\draw [color={rgb, 255:red, 71; green, 4; blue, 132 }  ,draw opacity=1 ]   (301,190.25) -- (301.64,147.8) ;
\draw [shift={(301.67,145.8)}, rotate = 90.86] [fill={rgb, 255:red, 71; green, 4; blue, 132 }  ,fill opacity=1 ][line width=0.08]  [draw opacity=0] (12,-3) -- (0,0) -- (12,3) -- cycle    ;
\draw [color={rgb, 255:red, 71; green, 4; blue, 132 }  ,draw opacity=1 ]   (301,190.25) -- (287.32,204.55) -- (280.38,211.8) ;
\draw [shift={(279,213.25)}, rotate = 313.73] [fill={rgb, 255:red, 71; green, 4; blue, 132 }  ,fill opacity=1 ][line width=0.08]  [draw opacity=0] (12,-3) -- (0,0) -- (12,3) -- cycle    ;
\draw [color={rgb, 255:red, 71; green, 4; blue, 132 }  ,draw opacity=1 ]   (258.67,190.75) -- (277.67,190.25) -- (301,190.25) ;
\draw [shift={(256.67,190.8)}, rotate = 358.5] [fill={rgb, 255:red, 71; green, 4; blue, 132 }  ,fill opacity=1 ][line width=0.08]  [draw opacity=0] (12,-3) -- (0,0) -- (12,3) -- cycle    ;
\draw [color={rgb, 255:red, 0; green, 0; blue, 0 }  ,draw opacity=0.3 ] [dash pattern={on 0.84pt off 2.51pt}]  (233.5,213.8) -- (233.5,169.35) ;
\draw [color={rgb, 255:red, 0; green, 0; blue, 0 }  ,draw opacity=0.3 ] [dash pattern={on 0.84pt off 2.51pt}]  (234.67,213.8) -- (279,213.25) ;
\draw [color={rgb, 255:red, 0; green, 0; blue, 0 }  ,draw opacity=0.3 ] [dash pattern={on 0.84pt off 2.51pt}]  (256.67,190.8) -- (242.98,205.1) -- (234.67,213.8) ;

\draw (218.2,321.42) node [anchor=north west][inner sep=0.75pt]  [font=\large,color={rgb, 255:red, 62; green, 108; blue, 5 }  ,opacity=1 ]  {${\displaystyle \partial \Omega _{T_{u}}}$};
\draw (470.39,260.36) node [anchor=north west][inner sep=0.75pt]    {$\Omega _{T} =\Omega $};
\draw (495.71,144.57) node [anchor=north west][inner sep=0.75pt]  [font=\large,color={rgb, 255:red, 68; green, 124; blue, 219 }  ,opacity=1 ]  {$\partial \Omega _{T_{\sigma }}$};
\draw (177,243.25) node [anchor=north west][inner sep=0.75pt]   [align=left] {$\displaystyle \zeta _{1}$};
\draw (102,275.25) node [anchor=north west][inner sep=0.75pt]   [align=left] {$\displaystyle \zeta _{2}$};
\draw (131,180.25) node [anchor=north west][inner sep=0.75pt]   [align=left] {$\displaystyle \zeta _{3}$};
\draw (223.33,140.8) node [anchor=north west][inner sep=0.75pt]  [color={rgb, 255:red, 255; green, 0; blue, 0 }  ,opacity=1 ]  {$\hat{n}$};
\draw (290.33,120.8) node [anchor=north west][inner sep=0.75pt]  [font=\small,color={rgb, 255:red, 71; green, 4; blue, 132 }  ,opacity=1 ]  {$\hat{n}_{3}$};
\draw (231,180) node [anchor=north west][inner sep=0.75pt]  [font=\small,color={rgb, 255:red, 71; green, 4; blue, 132 }  ,opacity=1 ]  {$\hat{n}_{1}$};
\draw (264.33,212.35) node [anchor=north west][inner sep=0.75pt]  [font=\small,color={rgb, 255:red, 71; green, 4; blue, 132 }  ,opacity=1 ]  {$\hat{n}_{2}$};

\end{tikzpicture}
}
	\end{center}
	\vspace{-14mm}\subcaption{3D domain}
	 \label{fig:papa_3D}
  \end{minipage}%
   \begin{minipage}[b]{.32\linewidth}
         \begin{center}
		\tikzset{every picture/.style={line width=0.75pt}} 

\scalebox{0.68}{
\begin{tikzpicture}[x=0.75pt,y=0.75pt,yscale=-0.75,xscale=0.75]

\draw  [fill={rgb, 255:red, 155; green, 155; blue, 155 }  ,fill opacity=0.3 ] (433.25,184.5) .. controls (402.25,198.5) and (394.8,207.16) .. (321.25,190.5) .. controls (279.03,180.94) and (222.88,202.79) .. (192.25,195.5) .. controls (145.25,177.5) and (201.25,106.5) .. (241.25,102.5) .. controls (245.22,101.83) and (249.1,101.21) .. (252.89,100.62) .. controls (345.53,86.39) and (385.55,97.41) .. (425.25,117.5) .. controls (471.2,140.75) and (452.25,171.5) .. (433.25,184.5) -- cycle ;
\draw [color={rgb, 255:red, 255; green, 0; blue, 0 }  ,draw opacity=1 ]   (433.25,184.5) -- (457.15,211.89) ;
\draw [shift={(458.47,213.4)}, rotate = 228.89] [fill={rgb, 255:red, 255; green, 0; blue, 0 }  ,fill opacity=1 ][line width=0.08]  [draw opacity=0] (12,-3) -- (0,0) -- (12,3) -- cycle    ;
\draw [color={rgb, 255:red, 71; green, 4; blue, 132 }  ,draw opacity=1 ]   (433.25,184.5) -- (408.82,210.93) ;
\draw [shift={(407.47,212.4)}, rotate = 312.74] [fill={rgb, 255:red, 71; green, 4; blue, 132 }  ,fill opacity=1 ][line width=0.08]  [draw opacity=0] (12,-3) -- (0,0) -- (12,3) -- cycle    ;
\draw [color={rgb, 255:red, 71; green, 4; blue, 132 }  ,draw opacity=1 ]   (482.25,185.46) -- (433.25,184.5) ;
\draw [shift={(484.25,185.5)}, rotate = 181.12] [fill={rgb, 255:red, 71; green, 4; blue, 132 }  ,fill opacity=1 ][line width=0.08]  [draw opacity=0] (12,-3) -- (0,0) -- (12,3) -- cycle    ;
\draw [color={rgb, 255:red, 215; green, 212; blue, 212 }  ,draw opacity=1 ]   (230,132.25) -- (230.4,79.25) ;
\draw [shift={(230.42,77.25)}, rotate = 90.43] [fill={rgb, 255:red, 215; green, 212; blue, 212 }  ,fill opacity=1 ][line width=0.08]  [draw opacity=0] (12,-3) -- (0,0) -- (12,3) -- cycle    ;
\draw    (230,132.25) -- (159.65,204.07) ;
\draw [shift={(158.25,205.5)}, rotate = 314.41] [fill={rgb, 255:red, 0; green, 0; blue, 0 }  ][line width=0.08]  [draw opacity=0] (12,-3) -- (0,0) -- (12,3) -- cycle    ;
\draw    (230,132.25) -- (471.25,133.49) ;
\draw [shift={(473.25,133.5)}, rotate = 180.29] [fill={rgb, 255:red, 0; green, 0; blue, 0 }  ][line width=0.08]  [draw opacity=0] (12,-3) -- (0,0) -- (12,3) -- cycle    ;
\draw [color={rgb, 255:red, 184; green, 233; blue, 134 }  ,draw opacity=0.34 ][line width=3.75]    (173.42,176.25) .. controls (173.51,168.19) and (178.18,150.38) .. (187.42,137.25) .. controls (196.65,124.12) and (205.37,121.72) .. (211.42,115.25) .. controls (217.46,108.78) and (222.73,106.44) .. (241.25,102.5) .. controls (259.77,98.56) and (315.08,92.8) .. (333.42,95.25) ;
\draw [color={rgb, 255:red, 119; green, 155; blue, 255 }  ,draw opacity=0.34 ][line width=3.75]    (333.42,95.25) .. controls (345.42,92.08) and (449.42,107.08) .. (452.42,149.08) .. controls (455.42,191.08) and (392.42,200.08) .. (382.42,200.08) .. controls (372.42,200.08) and (304.42,187.08) .. (295.42,187.08) .. controls (286.42,187.08) and (208.42,199.08) .. (200.42,198.08) .. controls (192.42,197.08) and (176.42,188.08) .. (173.42,176.25) ;
\draw  [color={rgb, 255:red, 0; green, 0; blue, 0 }  ,draw opacity=0.25 ][fill={rgb, 255:red, 155; green, 155; blue, 155 }  ,fill opacity=0 ] (434.25,173.5) .. controls (403.25,187.5) and (395.8,196.16) .. (322.25,179.5) .. controls (280.03,169.94) and (223.88,191.79) .. (193.25,184.5) .. controls (146.25,166.5) and (202.25,95.5) .. (242.25,91.5) .. controls (246.22,90.83) and (250.1,90.21) .. (253.89,89.62) .. controls (346.53,75.39) and (386.55,86.41) .. (426.25,106.5) .. controls (472.2,129.75) and (453.25,160.5) .. (434.25,173.5) -- cycle ;
\draw [color={rgb, 255:red, 0; green, 0; blue, 0 }  ,draw opacity=0.25 ]   (173.42,176.25) .. controls (173.01,217.18) and (212.01,208.18) .. (256.01,202.18) .. controls (300.01,196.18) and (316.01,198.18) .. (353.01,206.18) .. controls (390.01,214.18) and (415.01,205.18) .. (431.84,194.6) .. controls (448.67,184.02) and (451.01,173.18) .. (452.28,164.05) ;
\draw [color={rgb, 255:red, 0; green, 0; blue, 0 }  ,draw opacity=0.25 ] [dash pattern={on 0.84pt off 2.51pt}]  (173.42,176.25) .. controls (174.01,162.85) and (194.01,125.85) .. (233.01,113.85) .. controls (272.01,101.85) and (323.01,103.85) .. (357.01,106.85) .. controls (391.01,109.85) and (400.01,115.85) .. (423.84,127.6) .. controls (447.67,139.35) and (451.01,146.85) .. (452.28,164.05) ;
\draw    (203.03,171.6) -- (203,184.5) ;
\draw [shift={(203,184.5)}, rotate = 270.15] [color={rgb, 255:red, 0; green, 0; blue, 0 }  ][line width=0.75]    (0,3.35) -- (0,-3.35)(6.56,-1.97) .. controls (4.17,-0.84) and (1.99,-0.18) .. (0,0) .. controls (1.99,0.18) and (4.17,0.84) .. (6.56,1.97)   ;
\draw    (203.03,221.6) -- (203.03,207.6) ;
\draw [shift={(203.03,207.6)}, rotate = 90] [color={rgb, 255:red, 0; green, 0; blue, 0 }  ][line width=0.75]    (0,3.35) -- (0,-3.35)(6.56,-1.97) .. controls (4.17,-0.84) and (1.99,-0.18) .. (0,0) .. controls (1.99,0.18) and (4.17,0.84) .. (6.56,1.97)   ;
\draw [color={rgb, 255:red, 0; green, 0; blue, 0 }  ,draw opacity=0.5 ]   (203,184.5) -- (203.03,207.6) ;
\draw [color={rgb, 255:red, 4; green, 3; blue, 4 }  ,draw opacity=0.5 ] [dash pattern={on 0.84pt off 2.51pt}]  (484.25,185.5) -- (458.47,213.4) ;
\draw [color={rgb, 255:red, 0; green, 0; blue, 0 }  ,draw opacity=0.5 ] [dash pattern={on 0.84pt off 2.51pt}]  (458.47,213.4) -- (407.47,212.4) ;

\draw (172.95,86.8) node [anchor=north west][inner sep=0.75pt]  [font=\large,color={rgb, 255:red, 62; green, 108; blue, 5 }  ,opacity=1 ,rotate=-359]  {${\displaystyle \partial \Omega _{u}}$};
\draw (369.39,161.36) node [anchor=north west][inner sep=0.75pt]  [color={rgb, 255:red, 0; green, 0; blue, 0 }  ,opacity=0.75 ]  {$\Omega $};
\draw (395.71,81.57) node [anchor=north west][inner sep=0.75pt]  [font=\large,color={rgb, 255:red, 68; green, 124; blue, 219 }  ,opacity=1 ]  {$\partial \Omega _{\sigma }$};
\draw (473,121.25) node [anchor=north west][inner sep=0.75pt]   [align=left] {$\displaystyle \zeta _{1}$};
\draw (142,200.25) node [anchor=north west][inner sep=0.75pt]   [align=left] {$\displaystyle \zeta _{2}$};
\draw (219,53.25) node [anchor=north west][inner sep=0.75pt]  [color={rgb, 255:red, 0; green, 0; blue, 0 }  ,opacity=0.25 ] [align=left] {$\displaystyle \zeta _{3}$};
\draw (458.33,208.8) node [anchor=north west][inner sep=0.75pt]  [color={rgb, 255:red, 255; green, 0; blue, 0 }  ,opacity=1 ]  {$\hat{n}$};
\draw (486.33,175.35) node [anchor=north west][inner sep=0.75pt]  [font=\small,color={rgb, 255:red, 71; green, 4; blue, 132 }  ,opacity=1 ]  {$\hat{n}_{1}$};
\draw (394.33,210.35) node [anchor=north west][inner sep=0.75pt]  [font=\small,color={rgb, 255:red, 71; green, 4; blue, 132 }  ,opacity=1 ]  {$\hat{n}_{2}$};
\draw (174,175) node [anchor=north west][inner sep=0.75pt]  [font=\footnotesize,color={rgb, 255:red, 0; green, 0; blue, 0 }  ,opacity=1 ,rotate=-26] [align=left] {$\times$};
\draw (439,126) node [anchor=north west][inner sep=0.75pt]  [font=\footnotesize,color={rgb, 255:red, 0; green, 0; blue, 0 }  ,opacity=1 ,rotate=-5] [align=left] {$\times$};
\draw (204,162.5) node [anchor=north west][inner sep=0.75pt]    {$h$};
\draw (214.39,212.36) node [anchor=north west][inner sep=0.75pt]  [color={rgb, 255:red, 0; green, 0; blue, 0 }  ,opacity=0.75 ]  {$\Omega _{T} =\Omega \times $};
\draw (290.39,208.36) node [anchor=north west][inner sep=0.75pt]  [font=\scriptsize,color={rgb, 255:red, 0; green, 0; blue, 0 }  ,opacity=0.75 ]  {$\left( -\frac{h}{2} ,\frac{h}{2}\right) \ $};

\end{tikzpicture}
}
	\end{center}
	\vspace{-10mm}\subcaption{2D domain}
	\label{fig:papa_2D}
  \end{minipage}%
     \begin{minipage}[b]{.32\linewidth}
         \begin{center}
		\tikzset{every picture/.style={line width=0.75pt}} 

\scalebox{0.68}{
\begin{tikzpicture}[x=0.75pt,y=0.75pt,yscale=-0.8,xscale=0.8]

\draw [color={rgb, 255:red, 215; green, 212; blue, 212 }  ,draw opacity=1 ]   (230,132.25) -- (230.4,79.25) ;
\draw [shift={(230.42,77.25)}, rotate = 90.43] [fill={rgb, 255:red, 215; green, 212; blue, 212 }  ,fill opacity=1 ][line width=0.08]  [draw opacity=0] (12,-3) -- (0,0) -- (12,3) -- cycle    ;
\draw  [color={rgb, 255:red, 0; green, 0; blue, 0 }  ,draw opacity=0.15 ] (428.73,225.24) .. controls (417.92,222.94) and (412.42,205.75) .. (416.44,186.85) .. controls (420.45,167.94) and (432.47,154.47) .. (443.27,156.76) .. controls (454.08,159.06) and (459.58,176.25) .. (455.56,195.15) .. controls (451.55,214.06) and (439.53,227.53) .. (428.73,225.24) -- cycle ;
\draw [color={rgb, 255:red, 0; green, 0; blue, 0 }  ,draw opacity=0.15 ]   (292.6,113.27) -- (443.27,156.76) ;
\draw [color={rgb, 255:red, 0; green, 0; blue, 0 }  ,draw opacity=0.15 ]   (278.06,181.74) -- (428.73,225.24) ;
\draw [color={rgb, 255:red, 0; green, 0; blue, 0 }  ,draw opacity=0.15 ]   (292.6,113.27) .. controls (265.6,110.27) and (253.6,175.27) .. (278.06,181.74) ;
\draw [color={rgb, 255:red, 0; green, 0; blue, 0 }  ,draw opacity=0.15 ] [dash pattern={on 0.84pt off 2.51pt}]  (292.6,113.27) .. controls (316.6,118.27) and (306.6,184.27) .. (278.06,181.74) ;
\draw  [draw opacity=0][fill={rgb, 255:red, 184; green, 233; blue, 134 }  ,fill opacity=0.34 ] (275,148.13) .. controls (275,141.98) and (279.84,137) .. (285.8,137) .. controls (291.76,137) and (296.6,141.98) .. (296.6,148.13) .. controls (296.6,154.28) and (291.76,159.27) .. (285.8,159.27) .. controls (279.84,159.27) and (275,154.28) .. (275,148.13) -- cycle ;
\draw  [draw opacity=0][fill={rgb, 255:red, 119; green, 155; blue, 255 }  ,fill opacity=0.34 ] (426.2,191) .. controls (426.2,184.85) and (431.04,179.87) .. (437,179.87) .. controls (442.96,179.87) and (447.8,184.85) .. (447.8,191) .. controls (447.8,197.15) and (442.96,202.13) .. (437,202.13) .. controls (431.04,202.13) and (426.2,197.15) .. (426.2,191) -- cycle ;
\draw [color={rgb, 255:red, 215; green, 212; blue, 212 }  ,draw opacity=1 ]   (230,132.25) -- (189.87,172.84) ;
\draw [shift={(188.47,174.27)}, rotate = 314.67] [fill={rgb, 255:red, 215; green, 212; blue, 212 }  ,fill opacity=1 ][line width=0.08]  [draw opacity=0] (12,-3) -- (0,0) -- (12,3) -- cycle    ;
\draw    (230,132.25) -- (285.8,148.13) ;
\draw [color={rgb, 255:red, 155; green, 155; blue, 155 }  ,draw opacity=1 ]   (285.8,148.13) -- (436,191) ;
\draw    (436,191) -- (489.88,206.34) ;
\draw [shift={(491.8,206.88)}, rotate = 195.89] [fill={rgb, 255:red, 0; green, 0; blue, 0 }  ][line width=0.08]  [draw opacity=0] (12,-3) -- (0,0) -- (12,3) -- cycle    ;
\draw [color={rgb, 255:red, 184; green, 233; blue, 134 }  ,draw opacity=0.34 ][line width=1.5]    (285.8,137) .. controls (288.4,125.8) and (301.6,114.6) .. (311.2,112.2) ;
\draw [color={rgb, 255:red, 119; green, 155; blue, 255 }  ,draw opacity=0.34 ][line width=1.5]    (437,179.87) .. controls (437.87,170.7) and (424.67,147.1) .. (416.27,142.3) ;
\draw [color={rgb, 255:red, 255; green, 0; blue, 0 }  ,draw opacity=1 ]   (437,191) -- (461.98,198.38) ;
\draw [shift={(463.9,198.94)}, rotate = 196.45] [fill={rgb, 255:red, 255; green, 0; blue, 0 }  ,fill opacity=1 ][line width=0.08]  [draw opacity=0] (12,-3) -- (0,0) -- (12,3) -- cycle    ;
\draw [color={rgb, 255:red, 255; green, 0; blue, 0 }  ,draw opacity=1 ]   (285.8,148.13) -- (259.82,140.74) ;
\draw [shift={(257.9,140.19)}, rotate = 15.89] [fill={rgb, 255:red, 255; green, 0; blue, 0 }  ,fill opacity=1 ][line width=0.08]  [draw opacity=0] (12,-3) -- (0,0) -- (12,3) -- cycle    ;
\draw  [color={rgb, 255:red, 0; green, 0; blue, 0 }  ,draw opacity=0.15 ] (441.5,169.35) .. controls (441.5,166.78) and (443.58,164.7) .. (446.15,164.7) .. controls (448.72,164.7) and (450.8,166.78) .. (450.8,169.35) .. controls (450.8,171.92) and (448.72,174) .. (446.15,174) .. controls (443.58,174) and (441.5,171.92) .. (441.5,169.35) -- cycle ;
\draw  [color={rgb, 255:red, 0; green, 0; blue, 0 }  ,draw opacity=0.15 ] (445.35,169.35) .. controls (445.35,168.91) and (445.71,168.55) .. (446.15,168.55) .. controls (446.59,168.55) and (446.95,168.91) .. (446.95,169.35) .. controls (446.95,169.79) and (446.59,170.15) .. (446.15,170.15) .. controls (445.71,170.15) and (445.35,169.79) .. (445.35,169.35) -- cycle ;
\draw [color={rgb, 255:red, 0; green, 0; blue, 0 }  ,draw opacity=0.15 ]   (446.95,169.35) .. controls (449.07,141.1) and (457.07,183.1) .. (460.27,143.5) ;

\draw (292.35,91.86) node [anchor=north west][inner sep=0.75pt]  [font=\large,color={rgb, 255:red, 62; green, 108; blue, 5 }  ,opacity=1 ,rotate=-359]  {${\displaystyle \partial \Omega _{u}}$};
\draw (325.14,137.5) node [anchor=north west][inner sep=0.75pt]  [color={rgb, 255:red, 155; green, 155; blue, 155 }  ,opacity=1 ,rotate=-15.81]  {$\Omega =( a,b)$};
\draw (390.84,120.1) node [anchor=north west][inner sep=0.75pt]  [font=\large,color={rgb, 255:red, 68; green, 124; blue, 219 }  ,opacity=1 ]  {$\partial \Omega _{\sigma }$};
\draw (486.6,207.38) node [anchor=north west][inner sep=0.75pt]   [align=left] {$\displaystyle \zeta _{1}$};
\draw (219,56.25) node [anchor=north west][inner sep=0.75pt]  [color={rgb, 255:red, 0; green, 0; blue, 0 }  ,opacity=0.25 ] [align=left] {$\displaystyle \zeta _{3}$};
\draw (460.33,180) node [anchor=north west][inner sep=0.75pt]  [color={rgb, 255:red, 255; green, 0; blue, 0 }  ,opacity=1 ]  {$\hat{n} =\ $};
\draw (490.33,180) node [anchor=north west][inner sep=0.75pt]  [font=\small,color={rgb, 255:red, 71; green, 4; blue, 132 }  ,opacity=1 ]  {$\hat{n}_{1}$};
\draw (278,141) node [anchor=north west][inner sep=0.75pt]    {$\times $};
\draw (429,184) node [anchor=north west][inner sep=0.75pt]    {$\times $};
\draw (174,170.25) node [anchor=north west][inner sep=0.75pt]  [color={rgb, 255:red, 0; green, 0; blue, 0 }  ,opacity=0.25 ] [align=left] {$\displaystyle \zeta _{2}$};
\draw (263.7,157.25) node [anchor=north west][inner sep=0.75pt]  [font=\tiny] [align=left] {$\displaystyle \zeta _{1} =a$};
\draw (415.33,202.05) node [anchor=north west][inner sep=0.75pt]  [font=\tiny] [align=left] {$\displaystyle \zeta _{1} =b$};
\draw (231.93,142) node [anchor=north west][inner sep=0.75pt]  [color={rgb, 255:red, 255; green, 0; blue, 0 }  ,opacity=1 ]  {$=\hat{n}$};
\draw (219.33,197.8) node [anchor=north west][inner sep=0.75pt]  [color={rgb, 255:red, 0; green, 0; blue, 0 }  ,opacity=0.75 ]  {$\Omega _{T} =\Omega \times A$};
\draw (215.73,142) node [anchor=north west][inner sep=0.75pt]  [font=\small,color={rgb, 255:red, 71; green, 4; blue, 132 }  ,opacity=1 ]  {$\hat{n}_{1}$};
\draw (430,105.8) node [anchor=north west][inner sep=0.75pt]  [font=\footnotesize]  {$ \begin{array}{l}
A( \zeta _{1}) :Cross\\
section\ Area
\end{array}$};

\end{tikzpicture}
}
	\end{center}
	\vspace{-10mm}\subcaption{1D domain}
	\label{fig:papa_1D}
  \end{minipage}%
  \caption{Schemes to illustrate notation. }
  \label{fig:papas_elast}
\end{figure}
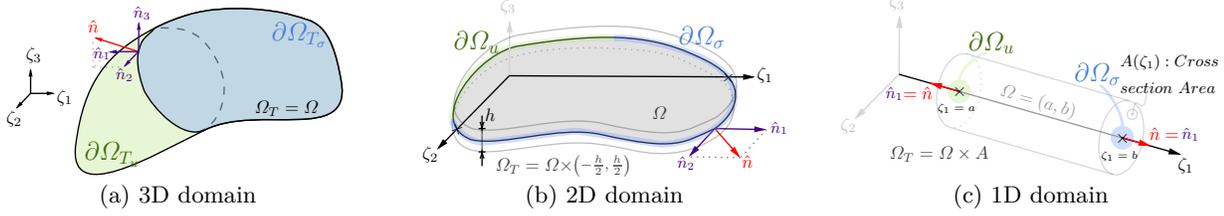

\subsection{Port-Hamiltonian systems} \label{ssec:PHS}

\noindent The theory of PHS provides a framework for the geometric description of interconnected systems. The framework's key features are the precise characterization of power flow between subsystems, the separation of the interconnecting structure from the constitutive relationships of its components, and the exploitation of this structure for analysis and control. Hence, the PH framework is particularly well adapted for the modeling, control, and simulation of complex nonlinear multiphysics systems \cite{L2gain3}. Using the so-called input and output ports, the modeling of complex systems can be approached constructively as the interconnection of elementary PH subsystems that may belong to one or more different physical domains (mechanical, electrical, hydraulic, thermal, etc.), where the interconnection is such that the total energy is preserved and the resulting system is also a PHS \cite{duindam2009modeling}. In simple terms, the conservative infinite-dimensional system can be written as a PHS from \cite{brugnoli2020port} as
\begin{equation}
\begin{array}{rl}
\p_t x = & \mathcal{J}\, \var_x H + \mathcal{G}\,u_d \\
y_d  = & \mathcal{G}^{*} \var_x H \\[2mm]
u_\p = & \mathcal{B}_\p \var_x H \\
y_\p = & \mathcal{C}_\p \var_x H 
\end{array}
\label{eq:dPHS_main}
\end{equation}
where $\p_t=\p/\p t$, $x$ contains the energy variables, $u_d$, $y_d$ are  the distributed input and output ports, respectively. $\mathcal{J}=-\mathcal{J}^*$ is a formally skew-adjoint differential operator (interconnection operator), $ \mathcal{G}$, $\mathcal{G}^{*}$ are the input map operator and its formal adjoint, respectively. $\var_x H$ is the variational derivative of the Hamiltonian functional $H$ with respect to $x$ and defines the co-energy variables. $\mathcal{B}_\p$, $\mathcal{C}_\p$ are boundary operators that provide the boundary input $u_\p$ and boundary output $y_\p$ \cite{le2005dirac,le2004semigroup}. Infinite-dimensional PHS are characterized by Hamiltonian functional and the  Stokes-Dirac structure \cite{van2002hamiltonian}.

\begin{definition} \label{def:Stokes_Dirac}
Let $ \mathscr{B}$, $\mathscr{F}$, and $\mathscr{E}$ be linear spaces, where $ \mathscr{F}$ is the flow space, its dual $\mathscr{E}$ is the effort space, and $\mathscr{B} = \mathscr{F} \times \mathscr{E}$ is called the bond space of power variables.  A Stokes-Dirac structure on $ \mathscr{B} = \mathscr{F} \times \mathscr{E} $ is a subspace $ \mathscr{D}_{s} \subset \mathscr{B} $ such that $ \mathscr{D}_{s} = \mathscr{D}_{s} ^\perp $ with respect to a bilinear form $\langle\langle \cdot,\cdot \rangle\rangle $ given by \cite{le2005dirac} 
\begin{equation}
\langle \langle (\texttt{f}_1,\texttt{f}_{\p_1},\texttt{e}_1,\texttt{e}_{\p_1}),(\texttt{f}_2,\texttt{f}_{\p_2},\texttt{e}_2,\texttt{e}_{\p_2})\rangle \rangle =
\langle \texttt{e}_1|\texttt{f}_2 \rangle_{in}^{\Omega} + \langle \texttt{e}_2|\texttt{f}_1 \rangle_{in}^{\Omega} - \langle \texttt{e}_{\p_1}|\texttt{f}_{\p_2} \rangle_{in}^{\p \Omega} - \langle \texttt{e}_{\p_2}|\texttt{f}_{\p_1} \rangle_{in}^{\p \Omega} 
\notag
\end{equation}
where $\langle \cdot|\cdot \rangle_{in}^{\Omega}$ and $\langle \cdot|\cdot \rangle_{in}^{\p \Omega}$ are inner products defined over the spatial domain $\Omega$, and its boundary $\p\Omega$, respectively. 
\noindent From the previous definition we have that for any $ (\texttt{f},\texttt{f}_{\p}, \texttt{e},\texttt{e}_{\p}) \in \mathscr{D}_{s} $ it is satisfied that
$
\langle \langle (\texttt{f},\texttt{f}_\p,\texttt{e},\texttt{e}_\p),(\texttt{f},\texttt{f}_\p,\texttt{e},\texttt{e}_\p)\rangle \rangle = 0 
$,  
which is verified in a general way using the Stokes' theorem \cite{le2005dirac}. 
\end{definition}

\noindent Now, consider the infinite-dimensional PHS in \eqref{eq:dPHS_main}, if we choose $\texttt{f}=[\texttt{f}_s,\texttt{f}_e,\texttt{f}_\p]^\top$ and $\texttt{e}=[\texttt{e}_s,\texttt{e}_e,\texttt{e}_\p]^\top$, where $\texttt{f}_s=\p_t x$, $\texttt{f}_e=u_d$, $\texttt{f}_\p=u_\p$, $\texttt{e}_s= \var_x H$, $\texttt{e}_e=-y_d$, $\texttt{e}_\p=-y_\p$, then the set
$$
\mathscr{D}_{s}=\lbrace (\texttt{f},\texttt{e})\in \mathscr{B}  \,\, |\,\, \texttt{f}_s = \mathcal{J}\texttt{e}_s + \mathcal{G}\texttt{f}_e \, , \, \texttt{e}_e = -\mathcal{G}^{*} \texttt{e}_s \,,\, \texttt{f}_\p =  \mathcal{B}_\p \texttt{e}_s \, , \, \texttt{e}_\p = -\mathcal{C}_\p \texttt{e}_s  \rbrace
$$ 
is a Stokes-Dirac structure \cite{brugnoli2020port}.  With the above definitions for $ \texttt{f} $ and $ \texttt{e} $, it is easy to show that the system \eqref{eq:dPHS_main} is conservative, that is, $\, \langle\langle (\texttt{f}, \texttt{e}),(\texttt{f}, \texttt{e}) \rangle\rangle = 0 $, and that the energy exchange with the environment is determined by the distributed and boundary ports by the expression
$$
{\p_t} H = \langle y_d |\, u_d \rangle_{in}^{\Omega} + \langle y_{\p} |\, u_{\p} \rangle_{in}^{\p \Omega}.
$$ 

\subsection{Motivation example: Timoshenko beam} \label{ssec:motivation_example}

\noindent\noindent In this part, we use the Timoshenko beam as an example to highlight the main difficulties in obtaining the PH representation for such system and to introduce several notations and key assumptions for the proposed modeling methodology. \noindent First consider the beam scheme in Figure \ref{fig:papa_1D}, with $\Omega = (a,b) \subset \R{}$ the spatial domain, $\p\Omega = \p\Omega_u \cup \,\p\Omega_\sigma$ the boundary domain, where $\p\Omega_u = \lbrace a \rbrace$ is the portion where displacements are imposed, and $\p\Omega_\sigma = \lbrace b \rbrace$ the portion where tractions are imposed. In this example assume no body forces applied to the beam ($ \textbf{F}(\CX,t)= 0$). Then, based on the first-order shear deformation theory, that is, plane sections perpendicular to the neutral axis before deformation remain plane but not necessarily perpendicular to the neutral axis after deformation, the kinematic assumptions of the Timoshenko beam are equivalent to the displacement field $\textbf{u}(\CX,t) \in \R{3}$ given by  
\begin{equation}
\textbf{u}(\CX,t) = 
\begin{bmatrix}
- \x{3}\,\psi(\x{1},t)\\[-2mm]
0\\[-2mm]
w(\x{1},t)
\end{bmatrix} 
\label{eq:u_TIM}
\end{equation}
where $w(\x{1},t)$ is the vertical displacement of a point in the neutral axis, and $\psi(\x{1},t)$ is the total angle rotated by the cross section respecting to the vertical axis. 
\noindent The equations of motion obtained after applying Hamilton's principle are given by (see the details in \cite[Chapter 2.2.3]{reddy2006theory})  
\begin{equation}
\begin{array}{lc l}
\mbox{for all $\x{1} \in \Omega$}: & & 
\rho I(\x{1}) \mfrac{\p^2 \psi}{\p t^2} + \kappa G A(\x{1})\left(\psi-\mfrac{\p w}{\p \x{1}}\right) - \mfrac{\p}{\p \x{1}}\left( EI(\x{1})\mfrac{\p \psi}{\p \x{1}}\right) =  0  \\[2mm]
& & \rho A(\x{1})\mfrac{\partial^2 w}{\partial t^2} + \mfrac{\partial}{\partial \x{1}}\left[\kappa G A(\x{1}) \left( \psi-\mfrac{\partial w}{\partial \x{1}} \right)\right] = 0   \\[5mm]
\mbox{for all $\sx_\sigma \in \p\Omega_\sigma$:}& & 
\hat{M}(\sx_\sigma \in \p\Omega_\sigma,t) - EI(\sx_\sigma \in \p\Omega_\sigma) \mfrac{\p \psi}{\p \x{1}} = 0  \\[2mm]
& & \hat{V}(\sx_\sigma \in \p\Omega_\sigma,t) - \kappa G A(\sx_\sigma \in \p\Omega_\sigma) \left( \mfrac{\p w}{\p \x{1}}  - \psi \right)  =  0 
\end{array}
\label{eq:Timoshenko_model_main}
\end{equation}
where $\sx_\sigma \in \p\Omega_\sigma$ is a coordinate along the boundary $\p\Omega_\sigma$,  $A(\x{1})$ is the cross section area, $I(\x{1})$ the second moment of inertia of the cross section, $G,E$ the material properties,  $\kappa$ is a correction factor, and $\hat{M}$,$\hat{V}$ are the imposed boundary tractions which represent internal bending moment and internal shearing force, respectively. For the PH representation of the Timoshenko beam model, the energy variables of the system can be chosen as \cite{macchelli2004modeling}:
\begin{equation}
\begin{array}{lcl c lcl}
p_1= \mtx{\rho I(\x{1})} \mfrac{\p \psi}{\p t} & , & p_2=\mtx{\rho A(\x{1})} \mfrac{\p w}{\p t}      & , & 
\q_1=\mfrac{\p \psi}{\p \x{1}}    & , &  \q_2=\left( \mfrac{\p w}{\p \x{1}} - \psi \right)
\end{array}
\label{eq:energy_var_TIM}
\end{equation} 
where $p_1$ is the angular momentum, $p_2$ the linear momentum, $ \q_1 $ is the deformation due to bending (also called curvature) and $\q_2$ is the shear deformation. The Hamiltonian function that represents the total stored energy in the system is given by
\begin{equation}
H(p,\q) = \mfrac{1}{2}\dint_{a}^{b} \left(\mfrac{p_1^2}{\rho I(\x{1})} + \mfrac{p_2^2}{\rho A(\x{1})} + \mtx{EI(\x{1})\q_1^2} + \mtx{\kappa G A(\x{1})\q_2^2} \right) d\x{1} 
\end{equation}

where the first and second terms are the rotational and translational kinetic energies, respectively, and the third and fourth terms are the elastic potential energy due to the bending and shearing, respectively. The power variables (or also called co-energy variables) are given by 
\noindent
\begin{equation}
\begin{array}{lcl}
f_{p_1} = \mfrac{\p p_1}{\p t}= \mtx{\rho I(\x{1})}\mfrac{\p^2 \psi}{\p t^2}  
& , &   
e_{p_1} = \mfrac{\var H}{\var p_1} = \mfrac{p_1}{\rho I(\x{1})} = \mfrac{\p \psi}{\p t} \\[2mm]
f_{p_2} = \mfrac{\p p_2}{\p t}= \mtx{\rho A(\x{1})}\mfrac{\p^2 w}{\p t^2}  
& , &   
e_{p_2} = \mfrac{\var H}{\var p_2} = \mfrac{p_2}{\rho A(\x{1})} = \mfrac{\p w}{\p t} \\[2mm]
f_{\q_1} = \mfrac{\p \q_1}{\p t}= \mfrac{\p}{\p t}\left(\mfrac{\p \psi}{\p \x{1}}\right)   
& , &    
e_{\q_1} = \mfrac{\var H}{\var \q_1} = \mtx{EI(\x{1})}\q_1 = \mtx{EI(\x{1})}\mfrac{\p \psi}{\p \x{1}} \\[2mm]
f_{\q_2} = \mfrac{\p \q_2}{\p t}= \mfrac{\p}{\p t}\left(\mfrac{\p w}{\p \x{1}}-\psi\right)   
& , &    
e_{\q_2} = \mfrac{\var H}{\var \q_2} = \mtx{\kappa GA(\x{1})}\q_2 = \mtx{\kappa GA(\x{1})}\left(\mfrac{\p w}{\p \x{1}}-\psi\right)  
\end{array}
\label{eq:co-energy_var_Tim}
\end{equation}
where $f_{p_1}$ is the inertial moment, $f_{p_2}$ the inertial force, $f_{\q_1}$ the bending strain velocity, $f_{\q_2}$ shearing strain velocity, $e_{p_1}$ the angular velocity, $e_{p_2}$ the linear velocity, $e_{\q_1}$ the internal bending moment and $e_{\q_2}$ the internal shearing force. With all the above defined power variables and using the notation $\mtx{\p_{1}} = \mtx{\p /\p \x{1}}$, the PH representation of the Timoshenko beam model is given by 
\begin{equation}
\begin{array}{l}
 \begin{pmatrix}
f_{p_1} \\[0mm] f_{p_2} \\[0mm] f_{\q_1} \\[0mm] f_{\q_2}
\end{pmatrix} = \ub{ 
\begin{bmatrix}
0 & 0 & \p_{1} & 1  \\[0mm] 
0 & 0 & 0 & \p_{1}  \\[0mm]
\p_{1} & 0 & 0 & 0  \\[0mm]
-1 & \p_{1} & 0 & 0 
\end{bmatrix} 
}{\mathcal{J} = - \mathcal{J}^*} 
\begin{pmatrix}
e_{p_1} \\[0mm] e_{p_2} \\[0mm] e_{\q_1} \\[0mm] e_{\q_2}
\end{pmatrix}
\\[13mm]
u_\p = \begin{bmatrix} e_{p_1}(a) & e_{p_2}(a) & e_{\q_1}(b) & e_{\q_2}(b) \end{bmatrix}  \\[1.5mm]
y_\p = \begin{bmatrix} e_{p_1}(b) & e_{p_2}(b) & -e_{\q_1}(a) & -e_{\q_2}(a) \end{bmatrix} \\
\end{array}
\label{eq:TBT-PHS_model}
\end{equation}

\noindent Note that this model has four state variables and therefore they are four equations, where the two first are equivalent to the model presented in \eqref{eq:Timoshenko_model_main}, and the two last are compatibility equations. From the Timoshenko beam example presented above, some of the difficulties in finding the PH representations are:
\begin{itemize}
\item Starting from the displacement field in \eqref{eq:u_TIM} and knowing the constitutive relationships, the equations of motion \eqref{eq:Timoshenko_model_main} are obtained after applying Hamilton's principle, which involves a lot of algebraic work and where integration by parts is required (generally over multidimensional domains). \\[-2mm]
\item The equations of motion in \eqref{eq:Timoshenko_model_main} do not always give clues about the structure of the associated PH model or the energy variables. A wrong choice of variables can lead to a model not associated with a skew-adjoint differential operator, making this an iterative process, especially in non-trivial cases. \\[-2mm]
\item While the boundary conditions of \eqref{eq:Timoshenko_model_main} arise naturally from Hamilton's principle, the boundary conditions of \eqref{eq:TBT-PHS_model} arise from the energy balance of the system, so they are defined from the Hamiltonian, the structure of the differential operator and the co-energy variables.
\end{itemize}
\noindent These difficulties motivate us to propose a systematic methodology that allows us to overcome them.

\subsection{Overview of the proposed modeling methodology}

\noindent In order to introduce the notations and the general approach developed in Section \ref{sec:modeling}, we first apply it to the proposed example. 
Notation: Let $\CX = \lbrace \x{1},\, \x{2},\,\x{3} \rbrace $ be the set of Cartesian coordinates in $\R{3}$, $\X \subset \CX $ the subset of coordinates where the parameters are distributed, and  $\X^c \subset \CX$ the complement of $\X$ such that $\CX = \X \,\cup \, \X^c$ and $\X \,\cap\, \X^c = \lbrace \phi \rbrace $, where $\lbrace \phi \rbrace$ denotes the empty set. In general, consider that $\Omega_T = \Omega \times \Omega^c \subset \R{3}$, where $\Omega \subset \R{\ell}$ is the $\ell$-dimensional spatial domain where the parameters of the model are distributed, and $\Omega^c$ a complementary domain such that $ \Omega \times \Omega^c = \Omega_T$ is the total volume of the elastic body (see Figure \ref{fig:papas_elast}). Then, a differential of volume $d\Omega_T$ and the integral of an arbitrary separable function $g(\CX) = g_1(\X)g_2(\X^c)$ over the volume $\Omega_T$ can be expressed as 
\begin{equation}
d\Omega_T = d\x{3}d\x{2}d\x{1} = d\X^c \,d\X = d\CX 
\quad , \quad 
\dint_{\Omega_T} g(\CX)\, d\CX =   \int_{\Omega}\! g_1(\X) \! \int_{\Omega^c} \!\! g_2(\X^c)\, d\X^c \, d\X 
\end{equation}
\noindent Note that in case of three-dimensional elasticity (see Figure \ref{fig:papa_3D}), we have $\X = \CX$ and $\X^c = \lbrace \phi \rbrace$, then $\Omega_T = \Omega$ and $d\CX = d\X$, which implies $ \int_{\Omega_T} g(\CX)\, d\CX = \int_{\Omega} g(\CX)\, d\X $, then  $ \int_{\Omega^c} g(\CX)\, d\X^c =  g(\CX) $. \\[-3mm]

\noindent Considering the Timoshenko beam, the displacement field $\textbf{u}(\CX,t)$ in \eqref{eq:u_TIM} is first rewritten as
$$
\textbf{u}(\CX,t)  = \begin{bmatrix}
- \x{3}\,\psi(\x{1},t)\\[-2mm]
0\\[-2mm]
w(\x{1},t)
\end{bmatrix} = \ub{\begin{bmatrix}
-\x{3} & 0   \\[-2mm] 0 & 0 \\[-2mm] 0 & 1 
\end{bmatrix}}{\lambda_1(\X^c)} \ub{ \begin{bmatrix}
\psi(\X,t) \\[-1mm] w(\X,t) 
\end{bmatrix} }{\textbf{r}(\X,t)} 
$$
where $\textbf{r}(\X,t)$ is the vector of primary unknowns, and $\lambda_1(\X^c)$ is the mapping of $\textbf{r}(\X,t)$ that allows us to know the deformed configuration of the three-dimensional beam at any point inside the volume $\Omega_T$. From Figure \ref{fig:papa_1D}, we see that the volume of the body can be written as $\Omega_T = \Omega \times A \subset \R{3}$, with $\Omega = (a,b) \subset \R{}$ an open set that defines the spatial domain where the parameters are distributed, and $A \subset \R{2}$ is the cross section area. 
\noindent From \eqref{eq:T_1} and $\textbf{u}(\CX,t)$, $\lambda_1(\X^c)$ allows us to find the momentum variables $p(\X,t)$ defined in \eqref{eq:energy_var_TIM} as
$$
p(\X,t) = \rho(\X) \dint_{\Omega^c}\!\! \lambda_1(\X^c)^\top \lambda_1(\X^c)\, d\X^c \, \dot{\textbf{r}}(\X,t) = \rho(\X) \dint_{A}\!\! \begin{bmatrix}
\x{3}^2 & 0 \\ 0 &  1
\end{bmatrix} dA \, \dot{\textbf{r}}(\X,t) = \begin{bmatrix}
\rho(\X) I(\X) & 0 \\ 0 & \!\!\!\!\!\rho(\X) A(\X)
\end{bmatrix} \begin{bmatrix}
\dot{\psi}(\X,t) \\ \dot{w}(\X,t)
\end{bmatrix}.
$$ 
From \eqref{eq:strain_tensor_lineal}, the only non-zero components of the strain tensor are
$$
\begin{bmatrix}
\varepsilon_{11} \\  2\varepsilon_{13}
\end{bmatrix} = \begin{bmatrix}
-\x{3}\, \p_1 \psi(\X,t) \\ \p_1 w(\X,t) - \psi(\X,t)
\end{bmatrix} = 
\ub{\begin{bmatrix}
-\x{3} & 0  \\  0 & 1
\end{bmatrix}}{\lambda_2(\X^c)}
 \ub{ \begin{bmatrix}
\p_1 & 0 \\ -1 & \p_1
\end{bmatrix}}{\mathcal{F}}
 \ub{\begin{bmatrix}
\psi(\X,t) \\  w(\X,t)
\end{bmatrix}}{\textbf{r}(\X,t)}
$$
from which it can already be seen that $\mathcal{F}$ and its formal adjoint $\mathcal{F}^*$ characterize the differential operator presented in \eqref{eq:TBT-PHS_model} as $\mathcal{J}=-\mathcal{J}^* = [0 \;\; -\!\mathcal{F}^* \; ; \; \mathcal{F} \;\;\; 0]$, and  the generalized strains $\q(\X,t)$ defined in \eqref{eq:energy_var_TIM} are given by $\q(\X,t) = \mathcal{F} {\textbf{r}}(\X,t)$. Writing the constitutive matrix as $C = [E \;\; 0 \,;\, 0 \;\;  \kappa G]$, from \eqref{eq:U_1} $\lambda_2(\X^c)$ allows us to find the the co-energy variables $e_\q(\X,t)$ defined in \eqref{eq:co-energy_var_Tim} as
$$
e_\q(\X,t) = \dint_{\Omega^c}\!\! \lambda_2(\X^c)^\top \, C \, \lambda_2(\X^c)\, d\X^c \, \q(\X,t) = \dint_{A}\!\! \begin{bmatrix}
E \x{3}^2 & 0 \\ 0 &  \kappa G
\end{bmatrix} dA \, \q(\X,t) = \begin{bmatrix}
E I(\X) & 0 \\ 0 & \!\!\!\!\! \kappa G A(\X)
\end{bmatrix} \begin{bmatrix}
\q_1(\X,t) \\ \q_2(\X,t)
\end{bmatrix}.
$$

\noindent It is worth to notice that only by using the displacement field $\textbf{u}(\X,t)$ without any application of the variational method, we are able to obtain the energy variables, co-energy variables and derive the structure of the skew-adjoint differential operator associated with the PH representation. This general idea will be rigorously generalized to a broader class of multidimensional systems in Section \ref{sec:modeling}.
%


\section{Modeling of linear elastic port-Hamiltonian systems} \label{sec:modeling}

\noindent In this section the modeling methodology is presented, where the definitions of the energy and co-energy variables are explicitly presented. Moreover, we introduce all the elements related to the infinite-dimensional PH representation of linear elastic mechanical models. 

\subsection{Considered class of systems} 

\noindent The model assumptions, the definition of the class of differential operators considered and the Lemma of integration by parts for them are presented below.

\begin{assumption}
We consider systems defined by linear elastic relations in a body of density $\rho = \rho(\X) \in \R{}$ and whose kinematic assumptions are represented by a displacement field $\textbf{u}(\CX,t) \in \R{3}$ with the following structure 
\begin{equation}
\textbf{u}(\CX,t) = \lambda_1(\X^c)\,\textbf{r}(\X,t)
\label{eq:hip_cinematica}
\end{equation}
where $\lambda_1(\X^c) \in \R{3\times n}$ is a matrix with no zero columns, and $\textbf{r}(\X,t)=[r_1(\X,t) \; \cdots \; r_n(\X,t)]^\top \in \R{n}$ is defined as the vector of generalized displacements or primary unknowns.
\end{assumption}

\noindent Note that depending on the specific definition of the displacement field $\textbf{u}(\CX,t)$ some components of the strain tensor could be zero. On the other hand, to avoid working directly with higher order tensors, the Voigt-Kelvin vector notation will be used. Then, let $\varepsilon(\CX,t) \in \R{d}$ be the non-zero components of the Voigt-strain vector $\vec{\varepsilon}(\CX,t)$ and $\sigma(\CX,t ) \in \R{d}$ the correlative components, so we can write the constitutive relation between stress and strain as 
\begin{equation}
\sigma(\CX,t) = C\, \varepsilon(\CX,t)
\label{eq:def_sigma_nonzero}
\end{equation}  
where $C = C^\top > 0 \in \R{d \times d}$ is a suitable constitutive matrix according to the problem (see \ref{ann:Voigt} for more details).

\begin{assumption}
We consider that $\p\Omega =\p\Omega_{u}\cup \p\Omega_{\sigma}$ is the boundary of the spatial domain $\Omega \subset \R{\ell}$ and that it is composed of complementary sub-boundaries $ \p\Omega_{u}$ and $ \p\Omega_{\sigma}$ (see Figure \ref{fig:papas_elast}), where the essential and natural boundary conditions (also called boundary tractions) are applied, respectively. Denote the distributed external inputs as $u_d(\X,t)$, the generalized boundary tractions as ${\tau}_\p(\mathtt{s}_{\sigma},t)$ with $\mathtt{s}_{\sigma} \in \p\Omega_{\sigma}$. Then, assume that the total external work $W_E$ defined in \eqref{eq:We_1} can be written equivalently as
\begin{equation}
W_E = - \left( \dint_{\Omega} B_d(u_d(\X,t))\cdot \textbf{r}(\X,t)\, d\X   + \dint_{\p\Omega_{\sigma}} \!\! {\tau}_\p(\mathtt{s}_{\sigma},t)\cdot  \mathcal{B}( \textbf{r}(\mathtt{s}_{\sigma},t)) \, d\mathtt{s}_{\sigma} \right)
\label{eq:def_WE_2}
\end{equation}
where $B_d(u_d(\X,t)) \in \mathbb{R}^n$ is defined as the generalized loads with $B_d(\cdot)$ a mapping operator, and $\mathcal{B}(\cdot)$ is a linear differential operator as defined in \eqref{eq:B_operator_theo_1}.
\end{assumption}

\begin{definition} \label{def:operadores_ND}
Let $\CX = \lbrace \x{1},\dots,\x{\ell} \rbrace$ be a set of pair-wise perpendicular coordinate axes and also the coordinates of an arbitrary point of an open set $\Omega \subset \R{\ell}$, $v(\CX) \in \R{m} $ and $w(\CX) \in \R{n} $ two vector functions. Consider a linear differential operator $ \mathcal{F}$ and its formal adjoint $\F^*$ such that, $\F \,w(\CX)=P_0\, w(\CX) + \sum_{k=1}^\ell \F_k \, w(\CX)$ and $\F^* \, v(\CX) = P_0^\top \, v(\CX) + \sum_{k=1}^ \ell \F_k^*\,  v(\CX)$, where $\F_k$ and its formal adjoint $\F_k^*$ are given by \cite{le2005dirac} 
\begin{align}
\F_k \, w(\CX) = &  \sum_{i=1}^N P_k(i)\, \p^i_k \, w(\CX) 
\label{eq:op_F} 
\\
\F_k^*\,v(\CX) = &  \sum_{i=1}^N (-1)^i P_k (i)^\top \p^i_k\, v(\CX)
\label{eq:op_F_N_1}
\end{align}
with  $\p_k^i = \p^i/\p \x{k}^i$, $P_0,\,P_k(i) \in \R{m \times n}$ and $N$ the order of the highest derivative with respect to any $\x{k}$.
\end{definition}

\begin{lemma} \label{theo:integration_by_parts_operator}
Consider that Definition \ref{def:operadores_ND} holds and define the open set $\Omega \subset \R{\ell}$ as an $\ell$-dimensional domain, its boundary $\p \Omega$ and $\bar{\Omega} = \Omega \cup \p\Omega$ the closure, such that $\X \in \Omega$ and $\sx \in \p\Omega$, where $\sx$ is the coordinates of an arbitrary point on $\p\Omega$. Then for any pair of functions $v(\X) \in \mathbb{R}^m$ and $w(\X) \in \mathbb{R}^n$ defined in $\bar{\Omega}$, we have that \vspace*{-2mm}
\begin{equation}
%
\int_\Omega \left(  v(\X)^\top \F\, w(\X) -  w(\X)^\top \F^*\,v(\X) \right) d\X =  \sum_{k=1}^\ell \sum_{i=1}^N \sum_{j=1}^i (-1)^{j-1} \!\! \int_{\p \Omega}\!\!   \p_k^{i-j}w(\sx)^\top P_k(i)^\top \hat{n}_k(\sx) \, \p_k^{j-1}v(\sx)  \, d \mathtt{s}
\label{eq: Integral_operator_F1}
\end{equation} \vspace*{-3mm}
\begin{equation}
\hspace{12mm} =  \int_{\p \Omega} \! \mathcal{B}( w(\sx))^\top \! \mathcal{Q}_\p(\sx) \, \mathcal{B} ( v(\sx)) \,  d \mathtt{s}
\label{eq: Integral_operator}
\end{equation}
where $\hat{n}_k(\sx)$ is the component of the outward unit normal vector to the boundary projected on the coordinate axis $\x{k}$, $\mathcal{B}(\cdot)$ is a linear differential operator defined  as 
\begin{align}
\mathcal{B}(\cdot) = &\begin{bmatrix}
(\cdot) & \p_1(\cdot) \; \cdots \; \p_\ell(\cdot) & \p_1^2(\cdot) \; \cdots \; \p_\ell^2(\cdot) & \cdots & \p_1^{N-1}(\cdot) \; \cdots \; \p_\ell^{N-1}(\cdot) 
\end{bmatrix}^\top  
\label{eq:B_operator_theo_1}
\end{align}
$ \mathcal{Q}_\p(\sx) \in \R{n+(N-1)n\ell\times m +(N-1)m\ell}$ is a matrix given by 
\begin{equation}
\mathcal{Q}_\p(\sx) = \left[ \begin{array}{l c c c c c}
P_\p(\sx) & -W_2(\sx) & W_3(\sx) & -W_4(\sx) & \cdots & (-1)^{N-1}W_{N}(\sx) \\[1mm]
V_2(\sx) & -\Lambda_3(\sx) & \Lambda_4(\sx) & \rddots{} & \rddots{} & 0 \\[-1mm]
V_3(\sx) & -\Lambda_4(\sx) & \Lambda_5(\sx)  & \rddots{} &  & \vdots  \\
\hspace{3.5mm}\vdots & \vdots & \rddots{} & \rddots{}  &  & \vdots \\[-1mm]
V_{N-1}(\sx) & -\Lambda_N(\sx) & 0  &  & & \vdots \\[1mm]
V_{N}(\sx) & 0 & 0& \cdots & \cdots & 0
\end{array} \right]
\label{eq:Q_k_theo_1}
\end{equation}
with $P_\p(\sx) \in \R{n\times m}$, $W_i(\sx) \in \R{n\times m\ell}$, $V_i(\sx) \in \R{n\ell\times m}$, and $\Lambda_i(\sx) \in \R{n\ell\times m\ell}$ defined as
\begin{equation}
\begin{array}{lcl}
P_\p(\sx) =  \displaystyle\sum_{k=1}^\ell P_k(1)^\top \hat{n}_k(\sx) & , & 
W_i(\sx) = \begin{bmatrix} P_1(i)^\top \hat{n}_1(\sx) & \cdots & P_\ell(i)^\top \hat{n}_\ell(\sx) \end{bmatrix} \\[5mm]
V_i(\sx) = \begin{bmatrix} P_1(i)^\top \hat{n}_1(\sx) \\ \vdots \\ P_\ell(i)^\top \hat{n}_\ell(\sx) \end{bmatrix} & , & 
\Lambda_i(\sx) = \begin{bmatrix} P_1(i)^\top \hat{n}_1(\sx) & & 0\\ & \ddots &  \\ 0 & & P_\ell(i)^\top \hat{n}_\ell(\sx) \end{bmatrix}
\end{array}
\label{eq:matrices_Qp_theo_1}
\end{equation}
\end{lemma}

\begin{proof}
\noindent Just by iteratively applying integration by parts on the left side of \eqref{eq: Integral_operator_F1}. The procedure is based on \cite[Theorem 3.1]{le2005dirac} and   \cite[Appendix A]{warsewa2021port}. 
\end{proof}

\begin{corollary} \label{cor:Integration_theorem_N=1}
Note that for an operator $\F$ of order $N = 1$, Lemma \ref{theo:integration_by_parts_operator} leads to
\begin{equation}
\int_\Omega \left( v(\X)^\top \F\, w(\X) - w(\X)^\top \F^*\,v(\X) \right) d\X   =   \int_{\p \Omega} \! w(\sx)^\top P_\p(\sx) \, v(\sx) \,d \mathtt{s}
\label{eq: Integral_operator orden 1}
\end{equation}
\end{corollary}

\subsection{Modeling methodology}

\noindent In this subsection, the key steps of the methodology for the systematic modeling of systems based on linear elasticity are presented. The section begins with two propositions that allows to define the energy variables from the kinematic assumptions and the differential operator, and then, using  Hamilton's principle, it is shown that this selected results lead to an infinite-dimensional port-Hamiltonian system. Finally, based on the kinematic assumptions and constitutive laws of the system, the methodology is outlined as a stet-by-step series.

\begin{proposition} \label{prop:1}
Consider linear elasticity under the kinematic assumption \eqref{eq:hip_cinematica}, the generalized momentum $p=p(\X,t) \in \R{n}$, the mass matrix $\mathcal{M}(\X)=\mathcal{M}(\X)^\top > 0 \in \R{n \times n}$, and the total kinetic energy $T[p] \in \R{}$ are defined as
\begin{align}
p(\X,t) = &\; \mathcal{M}(\X)\,\dot{\textbf{r}}(\X,t)
\label{eq:def_p_distributed} \\
\mathcal{M}(\X) = &\; \rho(\X)\dint_{\Omega^c} \lambda_1(\X^c)^\top\lambda_1(\X^c) \, d\X^c
\label{eq:def_Mass_matrix_distributed} \\
T[p] = &\; \mfrac{1}{2}\dint_{\Omega} p(\X,t)^\top \mathcal{M}(\X)^{-1}  p(\X,t) \,d\X
\label{eq:def_T_distributed_p}
\end{align}
\end{proposition}
$ $\\[-6mm]
\begin{proof}
The kinematic assumption $\textbf{u}(\CX,t)$ according to \eqref{eq:hip_cinematica} implies $\dot{\textbf{u}}(\CX,t) = \lambda_1(\X^c)\,\dot{\textbf{r}}(\X,t)$. Then, by the definition of kinetic energy in \eqref{eq:T_1} we have
$$
\begin{array}{rl}
T =  \mfrac{1}{2} \dint_{\Omega_T} \rho(\X) \, \dot{\textbf{u}}(\CX,t)^\top \dot{\textbf{u}}(\CX,t) \, d\CX = \mfrac{1}{2} \dint_{\Omega} \dot{\textbf{r}}(\X,t)^\top \ub{\rho(\X)\dint_{\Omega^c} \lambda_1(\X^c)^\top\lambda_1(\X^c) \, d\X^c}{\mathcal{M}(\X)} \, \dot{\textbf{r}}(\X,t) \,d\X \\
\end{array}
$$
but from \eqref{eq:def_p_distributed} we know that $\dot{\textbf{r}}(\X,t)= \mathcal{M}(\X)^{-1}p(\X,t)$, then we can write $T[p]$ as in \eqref{eq:def_T_distributed_p}.
\end{proof}
$ $

\begin{proposition} \label{prop:2}
Consider linear elasticity under the kinematic assumption \eqref{eq:hip_cinematica}. Assume that the strain vector $\varepsilon(\CX,t) \in \R{d}$ can be written as
\begin{equation}
\varepsilon(\CX,t) = \lambda_2(\X^c)\mathcal{F}\, \textbf{r}(\X,t)
\label{eq:factorization_prop_2}
\end{equation}

where $\lambda_2(\X^c) \in \R{d \times m}$ is a matrix, and $\mathcal{F}$ is a $(m \times n)$ linear differential operator as in Definition \ref{def:operadores_ND}, both without any zero rows or columns. Then, the generalized strains $\q = \q(\X,t) \in \R{m}$, the stiffness matrix $\mathcal{K}(\X) \in \R{m\times m}$, and the elastic potential energy $U[\q] \in \R{}$ are defined as
\begin{align}
\q(\X,t) = & \; \mathcal{F}\,\textbf{r}(\X,t)
\label{eq:def_q:distributed} \\
\mathcal{K}(\X) = &\; \dint_{\Omega^c} \lambda_2(\X^c)^\top \, C \, \lambda_2(\X^c) d\X^c
\label{eq:def_Stiffness_matrix_distributed} \\
U[\q] = & \; \mfrac{1}{2} \dint_{\Omega} \q(\X,t)^\top \mathcal{K}(\X) \, \q(\X,t) \, d\X
\label{eq:def_U_distributed_q}
\end{align}
with $C = C^\top > 0 \in \R{d \times d}$ is the constitutive matrix defined as \eqref{eq:def_sigma_nonzero}. The dimension $ m\in\mathbb{N}$ must be chosen such that $\mathcal{K}(\X)=\mathcal{K}(\X)^\top>0 $.
\end{proposition}
$ $\\[-6mm]
\begin{proof}
From \eqref{eq:factorization_prop_2} we have $\varepsilon(\CX,t) =
 \lambda_2(\X^c) \q(\X,t)$, and from \eqref{eq:def_sigma_nonzero} we have $\sigma(\CX,t)= C\, \lambda_2(\X^c) \, \q(\X,t)$. Then, by the definition of the elastic potential energy in \eqref{eq:U_1} we have
$$
\begin{array}{rl}
U =  \mfrac{1}{2} \dint_{\Omega_T} \sigma(\CX,t)^\top \varepsilon(\CX,t) \, d\CX = \mfrac{1}{2} \dint_{\Omega} \q(\X,t)^\top \ub{\dint_{\Omega^c} \lambda_2(\X^c)^\top \, C \, \lambda_2(\X^c) \, d\X^c}{\mathcal{K}(\X)} \, \q(\X,t) \,d\X
\end{array}
$$ \\[-5mm]
with the above the proof ends.
\end{proof}

\begin{theorem} \label{theo:PHS_SD}
Let consider $x(\X,t) = [p(\X,t)^\top \; \q(\X,t)^\top]^\top \in \R{(n+m)}$ as the state variable with $\textbf{u}(\CX,t)$ defined in \eqref{eq:hip_cinematica} and total external work $W_E$ in \eqref{eq:def_WE_2}. From Propositions \ref{prop:1} and \ref{prop:2} the dynamics of the system defines an infinite-dimensional port-Hamiltonian system of the form
\begin{equation}
\begin{array}{rl}
\ub{\begin{bmatrix}
f_p \\ f_\q
\end{bmatrix}}{\dot{x}} = & \ub{\begin{bmatrix}
0 & -\mathcal{F}^* \\ \mathcal{F} & 0 
\end{bmatrix}}{\mathcal{J}=-\mathcal{J}^*}\ub{\begin{bmatrix}
e_p \\ e_\q
\end{bmatrix}}{\var_x H} + \ub{\begin{bmatrix}
B_d \\ 0
\end{bmatrix}}{\mathcal{G}} u_d \\[11mm]
y_d = & \mathcal{G}^* \var_x H = B_d^*(e_p) 
\end{array}
\label{eq:def_dPHS_prop3}
\end{equation}
\begin{equation}
H[p,\q] = \mfrac{1}{2} \dint_{\Omega}  p(\X,t)^\top \mathcal{M}(\X)^{-1} p(\X,t) + \q(\X,t)^\top \mathcal{K}(\X) \q(\X,t) \,d\X 
\label{eq:def_HAMILTONIANO}
\end{equation}

where $H[p,\q]=T[p]+U[\q]$ is the Hamiltonian of the system, $f_p = \dot{p} = \mathcal{M}\,\ddot{\textbf{r}}$, $f_\q = \dot{\q} = \mathcal{F}\,\dot{\textbf{r}}$, $e_p = \var_p H = \mathcal{M}^{-1}\,p = \dot{\textbf{r}}$, $e_\q = \var_\q H = \mathcal{K}\, \q = \mathcal{K}\mathcal{F}\,\textbf{r}$, and the power exchange with the environment is given by
\begin{equation}
\HS{-18} \p_t {H} = \dint_{\Omega} y_d(\X,t)^\top u_d(\X,t) \,d\X + \dint_{\p\Omega} \mathcal{B}(e_p(\mathtt{s},t))^\top  \mathcal{Q}_\p(\mathtt{s}) \, \mathcal{B}(e_\q(\mathtt{s},t))  \, d\mathtt{s}
\label{eq:def_HAMILTONIANO_dot_OP}
\end{equation} \\[-5mm]
or equivalently by \\[-3mm]
\begin{equation}
\p_t {H} = \dint_{\Omega} y_d(\X,t)^\top u_d(\X,t) \,d\X + \sum_{k=1}^\ell \sum_{i=1}^N \sum_{j=1}^i (-1)^{j-1} \!\! \int_{\p \Omega}\!\!   \p_k^{i-j}e_p(\mathtt{s},t)^\top P_k(i)^\top  \hat{n}_{k}(\mathtt{s}) \, \p_k^{j-1}e_\q(\mathtt{s},t)  \, d \mathtt{s}
\label{eq:def_HAMILTONIANO_dot_SUM}
\end{equation} 
where $\hat{n}_{k}$ is the component of the outward unit normal vector to the boundary $\p\Omega$ projected on $\x{k}$.
\end{theorem}

\begin{proof}
First of all we have $\textbf{u}(\CX,t) = \lambda_1(\X^c)\, \textbf{r}(\X,t) $, $\var\textbf{u}(\CX,t) = \lambda_1(\X^c)\, \var\textbf{r}(\X,t) $,  $ \dot{\textbf{u}}(\CX,t) = \lambda_1(\X^c) \, \dot{\textbf{r}}(\X,t)$, $ \var\dot{\textbf{u}}(\CX,t) = \lambda_1(\X ^c) \, \var\dot{\textbf{r}}(\X,t)$, $ \varepsilon(\CX,t) =  \lambda_2(\X^c) \, \mathcal{F}\,\textbf{r}(\X,t)$, $ \var\varepsilon(\CX,t) =  \lambda_2(\X^c) \, \mathcal{F}\,\var\textbf{r}(\X,t)$, and
$$
\begin{array}{rl}
\var T = & \dint_{\Omega_T} \rho(\X)\dot{\textbf{u}}(\CX,t) \cdot \var\dot{\textbf{u}}(\CX,t) d\CX = \dint_{\Omega} \var\dot{\textbf{r}}(\X,t)^\top \mathcal{M}(\X) \dot{\textbf{r}}(\X,t) d\X \\[4mm]
\var W_E = & \!\! -  \dint_{\Omega} \var\textbf{r}(\X,t)^\top B_d(u_d(\X,t)) \, d\X   -  \dint_{\p\Omega_{\sigma}} \!\!  \mathcal{B}(\var\textbf{r}(\mathtt{s}_{\sigma},t))^\top \, {\tau}_\p(\mathtt{s}_{\sigma},t) \, d\mathtt{s}_{\sigma} \\[4mm]
\var U = & \dint_{\Omega_T} \sigma(\CX,t) \cdot \var\varepsilon(\CX,t) \, d\CX = \dint_{\Omega} e_\q(\X,t)^\top \mathcal{F}\,\var\textbf{r}(\X,t) d\X
\end{array}
$$
Applying Lemma \ref{theo:integration_by_parts_operator} to the right side of $\var U$, and considering that due to \eqref{eq:admisible_1} the integral in the portion of the boundary $\p\Omega_{u}$ where the essential boundary conditions are specified is zero, since $ \var\textbf{r}(\mathtt{s}_{u},t) = 0$, then we have 
$$
\var U = \dint_{\Omega} \var\textbf{r}(\X,t)^\top \mathcal{F}^*\,e_\q(\X,t) d\X + \int_{\p \Omega_{\sigma}} \! \mathcal{B}( \var\textbf{r}(\mathtt{s}_{\sigma},t))^\top  \mathcal{Q}_\p(\mathtt{s}_{\sigma}) \, \mathcal{B}(e_\q(\mathtt{s}_{\sigma},t)) \,   d \mathtt{s}_{\sigma}
$$
As a previous step to apply Hamilton's principle, we integrate by parts respect to time the variation of the kinetic energy $\var T$, and considering that due to \eqref{eq:admisible_2} $\var\textbf{r}(\X,t_1)=\var\textbf{r}(\X,t_2)=0$, we obtain
$$
\dint_{t_1}^{t_2} \var T \, dt = -\dint_{t_1}^{t_2}  \dint_{\Omega} \var{\textbf{r}}(\X,t)^\top \mathcal{M}(\X) \, \ddot{\textbf{r}}(\X,t)\, d\X \, dt \; + \; \ub{\dint_{\Omega} (\var{\textbf{r}}(\X,t)^\top \mathcal{M}(\X) \, \dot{\textbf{r}}(\X,t) )\big|_{t_1}^{t_2} \, d\X}{=0}
$$ 
Then, with all the above we apply Hamilton's principle and we obtain 
$$
\begin{array}{r}
\dint_{t_1}^{t_2} \!\! \left[ \begin{matrix} \blanco{|} \\[3mm] \blanco{|}   \end{matrix}  \!\!
\dint_{\Omega} \var{\textbf{r}}(\X,t)^\top \left[\mathcal{M}(\X)\, \ddot{\textbf{r}}(\X,t) + \mathcal{F}^*\,e_\q(\X,t) - B_d(u_d(\X,t)) \right] d\X \; + ... \quad\quad \quad \quad   \quad \quad \quad\quad \quad   \quad \quad\quad \quad  \quad \quad \right. \\[-2mm]
... 
\left. \dint_{\p \Omega_{\sigma}} \!\!  \mathcal{B}(\var{\textbf{r}}(\mathtt{s}_{\sigma},t))^\top \left[ \mathcal{Q}_\p(\mathtt{s}_{\sigma}) \, \mathcal{B}( e_\q(\mathtt{s}_{\sigma},t))  - {\tau}_\p(\mathtt{s}_{\sigma},t) \right]  d\mathtt{s}_{\sigma} \right] \! dt = 0
\end{array}
$$
so applying Lemma \ref{lemma:1} and Lemma \ref{lemma:2} respectively  to each term in the above expression (see Appendix \ref{ann:Lemmas_variational}), we obtain the following Lagrangian model
\begin{align}
\forall \X \in \Omega: & \quad \mathcal{M}(\X)\ddot{\textbf{r}}(\X,t) + \mathcal{F}^*\,e_\q(\X,t) - B_d(u_d(\X,t)) = 0  \label{eq:Lag_dinamica}      \\[1mm]
\forall \sx_{\sigma} \in \p\Omega_{\sigma}: & \quad\quad \quad \tau_\p(\mathtt{s}_{\sigma},t)  =  \mathcal{Q}_\p(\mathtt{s}_{\sigma}) \mathcal{B}( e_\q(\mathtt{s}_{\sigma},t))     \label{eq:Lag_BC_sigma}
\end{align}
Note that the dynamic equation  \eqref{eq:Lag_dinamica} together with $\dot{\q}(\X,t)=\mathcal{F}\, \dot{\textbf{r}}(\X,t)=\mathcal{F}\,e_p(\X,t)$  can be written equivalently as in \eqref{eq:def_dPHS_prop3} with $H[p,\q]$ 
the Hamiltonian defined in \eqref{eq:def_HAMILTONIANO}.
The power exchanged with the environment is given by
$$
\p_t {H} = \dint_{\Omega} \var_x H ^\top \dot{x}\, d\X = \dint_{\Omega} e_p^\top B_d(u_d) \, d\X - \dint_{\Omega} e_p^\top \mathcal{F}^*\,e_\q \, d\X + \dint_{\Omega} e_\q^\top \mathcal{F}\,e_p \, d\X
$$
so applying Lemma \ref{theo:integration_by_parts_operator} to the first term, and also to the last two terms in the equation above we obtain
$$
\p_t {H} =  \dint_{\Omega} u_d^\top B_d^*(e_p) \, d\X + \dint_{\p\Omega} \! \mathcal{B}( u_d)^\top \mathcal{Q}_\p \, \mathcal{B}( e_p ) \, d\sx + \dint_{\p\Omega} \! \mathcal{B}( e_p)^\top \mathcal{Q}_\p \, \mathcal{B}( e_\q ) \, d\sx
$$
From the first term  above we obtain the power conjugated distributed output $y_d(\X,t)$ defined in \eqref{eq:def_dPHS_prop3}, and the second term is equal to zero because $u_d(\X,t)$ is not defined on the boundary ($u_d = 0$ in $\p \Omega$).  So considering the above, the last expression is equal to \eqref{eq:def_HAMILTONIANO_dot_OP} and equivalent to  \eqref{eq:def_HAMILTONIANO_dot_SUM}.
\end{proof}

\begin{corollary} \label{cor:boundary_ports}
The boundary ports $u_\p,\, y_\p$ and the boundary operators $\mathcal{B}_\p,\,\mathcal{C}_\p$ are chosen according to the energy balance equation \eqref{eq:def_HAMILTONIANO_dot_OP}, such that
\begin{equation}
\dint_{\p \Omega} \! \mathcal{B}( e_p(\mathtt{s},t))^\top \mathcal{Q}_\p(\mathtt{s}) \, \mathcal{B}( e_\q (\mathtt{s},t))  \, d\sx = 
\dint_{\p \Omega} \!  y_\p(\mathtt{s},t)^\top u_\p(\mathtt{s},t)  \, d\sx 
\label{eq:boundary_ports_COR_ND}
\end{equation} 
where $\mathcal{B}_\p,$ and $\mathcal{C}_\p$ can be other defined depending on the considered boundary control variables. For example, it can be chosen $\mathcal{B}_\p = \mathcal{Q}_\p(\mathtt{s}) \mathcal{B}(\cdot)$ such that $u_\p(\mathtt{s},t) = \mathcal{Q}_\p(\mathtt{s}) \mathcal{B}( e_\q(\mathtt{s},t))$, and $\mathcal{C}_\p = \mathcal{B}(\cdot)$ such that $y_\p(\mathtt{s},t) =  \mathcal{B}(e_p(\mathtt{s},t))$. Then, according to Definition \ref{def:Stokes_Dirac} we have $\texttt{f}=[\texttt{f}_s,\texttt{f}_e,\texttt{f}_\p]^\top$ and $\texttt{e}=[\texttt{e}_s,\texttt{e}_e,\texttt{e}_\p]^\top$, where $\texttt{f}_s= \dot{x}$, $\texttt{f}_e=u_d$, $\texttt{f}_\p=u_\p$, $\texttt{e}_s= \var_x H$, $\texttt{e}_e=-y_d$, $\texttt{e}_\p=-y_\p$, then the set 
$
\mathscr{D}_{s}=\lbrace (\texttt{f},\texttt{e})\in \mathscr{B}  \,\, |\,\, \texttt{f}_s = \mathcal{J}\texttt{e}_s + \mathcal{G}\texttt{f}_e \, , \, \texttt{e}_e = -\mathcal{G}^{*} \texttt{e}_s \,,\, \texttt{f}_\p =  \mathcal{B}_\p \texttt{e}_s \, , \, \texttt{e}_\p = -\mathcal{C}_\p \texttt{e}_s  \rbrace
$ 
is a Stokes-Dirac structure \cite{brugnoli2020port}, and the energy exchange with the environment is determined by 
$
{\p_t} H =  \langle {y}_d | {u}_d \rangle_{in}^{\Omega} + \langle {y}_{\p} | {u}_{\p} \rangle_{in}^{\p \Omega} 
$.
\end{corollary}

\begin{remark} \label{rem:generally_N=1}
Note that the structure of $\mathcal{F}$ is mainly determined by Grad($\textbf{u}(\CX,t)$), where if there are no differential dependencies between the components of $\textbf{r}(\CX,t)$, the operator is constant and first order. Also note that the assumption of Hooke's law leads to linear models, but in general this could be removed to account for material nonlinearities.
\end{remark}

\begin{remark}\label{rem:boundary_ports_N1}
If the operator $\mathcal{F}$ in \eqref{eq:def_dPHS_prop3} is of order $N=1$, then the expression in \eqref{eq:boundary_ports_COR_ND} reduces to 
\begin{equation}
\dint_{\p \Omega} \! \mathcal{B}( e_p(\mathtt{s},t))^\top \mathcal{Q}_\p(\mathtt{s}) \, \mathcal{B}( e_\q(\mathtt{s},t))  \, d\sx = 
\dint_{\p\Omega} \!\! e_p(\mathtt{s},t)^\top P_\p(\mathtt{s}) \, e_\q(\mathtt{s},t) \, d\mathtt{s} = 
\dint_{\p \Omega} \!  y_\p(\mathtt{s},t)^\top u_\p(\mathtt{s},t)  \, d\sx 
\label{eq:boundary_ports_COR_1D}
\end{equation} 
so if we have $\mathcal{B}_\p = P_\p(\mathtt{s})$ and  $\mathcal{C}_\p = 1_n$, the boundary input is equal to the generalized boundary tractions, that is $u_\p(\mathtt{s},t) = {\tau}_\p(\mathtt{s},t)$, and the boundary output corresponds to the generalized boundary velocities, that is $y_\p(\mathtt{s},t) = \dot{\textbf{r}}_\p(\mathtt{s},t)$. Note that this class of differential operators ($N=1$) is considered in \cite{skrepek2019well} where it is shown the well-posedness of linear first order port-Hamiltonian systems in multidimensional domains $\Omega \subset \R{\ell}$.
\end{remark}

\noindent\textbf{Procedure:}  The modeling methodology is summarized in the following procedure. To formulate infinite-dimensional port-Hamiltonian models based on kinematic assumptions that lead to a displacement fields of the class  
$
\textbf{u}(\CX,t) = \lambda_1(\X^c) \, \textbf{r}(\X,t)
$, 
where the constitutive matrix $C$ and the density of material $\rho(\X)$ are known, follows the steps below:
\begin{enumerate}
\item Calculate the mass matrix $\mathcal{M}(\X)$ from \eqref{eq:def_Mass_matrix_distributed} and define the generalized momentum $p(\X,t)$ according to \eqref{eq:def_p_distributed}, that is \\[-0mm]
$$
\mathcal{M}(\X) = \rho(\X) \int_{\Omega^c} \! \lambda_1(\X^c)^\top \! \lambda_1(\X^c) \, d\X^c  \quad \longrightarrow \quad p(\X,t ) = \mathcal{M}(\X) \dot{\textbf{r}}(\X,t)
$$ 
\item Compute the nonzero components of the strain tensor $\varepsilon(\CX,t)$ using \eqref{eq:def_vec_epsilon} and factorize it according to \eqref{eq:factorization_prop_2}. Define the generalized strains  $\q(\X,t)$ according to \eqref{eq:def_q:distributed} and compute the stiffness matrix $\mathcal{K}(\X)$ from \eqref{eq:def_Stiffness_matrix_distributed}.\\[-0mm]
$$
\begin{array}{l}
 \varepsilon(\CX,t) =  \lambda_2(\X ^c) \mathcal{F}\, \textbf{r}(\X,t) \quad \longrightarrow \quad \q(\X,t ) = \mathcal{F}\, \textbf{r}(\X,t) \quad , \quad \mathcal{K}(\X) = \dint_{\Omega^c} \! \lambda_2(\X^c)^\top \! C\, \lambda_2(\X^c) \, d\X^c
\end{array}
$$ 
\item Apply Theorem \ref{theo:PHS_SD} to obtain the infinite-dimensional port-Hamiltonian system and define the boundary ports $u_\p,y_\p$ from Corollary \ref{cor:boundary_ports}.
$$
\mbox{PHS from Theorem \ref{theo:PHS_SD}} \quad \longrightarrow \quad \mbox{Boundary ports ($u_\p,y_\p$) from Corollary \ref{cor:boundary_ports} }
$$ 
\end{enumerate}


\begin{remark} \label{rem:Legendre_PHS}
The port-Hamiltonian model obtained in Theorem \ref{theo:PHS_SD} is different from the one obtained by the Legendre transformation of the Euler-Lagrange system  \eqref{eq:Lag_dinamica} \cite{schoberl2013analysis}. Furthermore, the latter is a field port-Lagrangian system \cite{nishida2005formal} defined using the state variable $z(\X,t) = [p(\X,t)^\top \;\; \textbf{r}(\X,t)^\top]^\top \in \mathbb{R}^{2n}$, which leads to an infinite-dimensional system associated with an algebraic skew-symmetric matrix $J=-J^\top$. This system is given by
\begin{equation}
\begin{array}{rl}
\ub{\begin{bmatrix}
\dot{p} \\ \dot{\textbf{r}}
\end{bmatrix}}{\dot{z}} = & \ub{\begin{bmatrix}
0 & -1_n \\ 1_n & 0 
\end{bmatrix}}{J=-J^\top}\ub{\begin{bmatrix}
e_p \\ e_r
\end{bmatrix}}{\var_z H} + \ub{\begin{bmatrix}
B_d \\ 0
\end{bmatrix}}{\mathcal{G}} u_d \\[10mm]
y_d = & \mathcal{G}^* \var_z H = B_d^*(e_p) 
\end{array}
\label{eq:def_dPHS_blue}
\end{equation}
where $H[p,\textbf{r}] \in \R{}$ is the Hamiltonian, which is given by
\begin{equation}
H[p,\textbf{r}] = T[p] + U[\textbf{r}] = \mfrac{1}{2} \dint_{\Omega}  p(\X,t)^\top \mathcal{M}(\X)^{-1} p(\X,t) + (\mathcal{F}\,\textbf{r}(\X,t))^\top \mathcal{K}(\X)\, (\mathcal{F}\,\textbf{r}(\X,t)) \,d\X 
\label{eq:HAMILTONIAN_Z}
\end{equation}
and $e_r(\X,t)$ given by
\begin{equation}
 e_r(\X,t) = \mathcal{F}^*(\mathcal{K(\X)}\, \mathcal{F} \, \textbf{r}(\X,t) )
\label{eq:var_Hr}
\end{equation} \\[-12mm]

is the variational derivative of the Hamiltonian $H[p,\textbf{r}]$ respect to $\textbf{r}(\X,t)$ (see Appendix \ref{ann:proof_REMARK_LEGENDRE} for the proof).
\end{remark}

\section{Reddy's plate model}

\noindent In this section, the plate model based on Reddy's third-order shear deformation theory \cite{reddy1984simple} is presented. This model generalizes the classic model of Mindlin's  plate, which turn out to be a particular case of Reddy's plate model. In addition, because this theory considers higher order terms, the kinematic assumption is closer to reality, which in particular implies that these models describe more accurately the shear stress contributions to the elastic potential energy, thus avoiding the use of correction factors and the shear locking problem in finite element approximations (as in the case of Mindlin's plate or the Timoshenko beam). Finally, since there is a better description of shear stresses, these models are particularly useful for describing the dynamics of thick or layered beams and plates. For more details, see \cite{reddy1984simple} and \cite[Chapter 11]{reddy2003mechanics}.

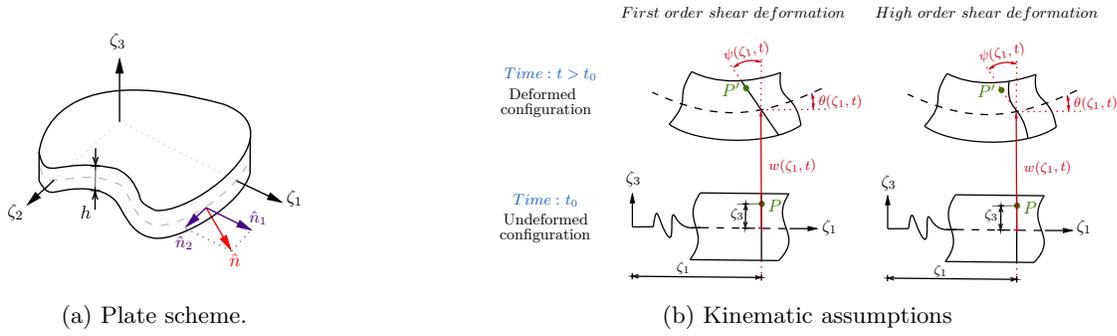
\begin{figure}[t]
\centering
\begin{minipage}[b]{.32\linewidth}
     \begin{center}
		\tikzset{every picture/.style={line width=0.75pt}} 

\scalebox{0.70}{
\begin{tikzpicture}[x=0.75pt,y=0.75pt,yscale=-1,xscale=1]

\draw [color={rgb, 255:red, 255; green, 0; blue, 0 }  ,draw opacity=1 ]   (292.68,182.67) -- (309.99,212.27) ;
\draw [shift={(311,214)}, rotate = 239.68] [fill={rgb, 255:red, 255; green, 0; blue, 0 }  ,fill opacity=1 ][line width=0.08]  [draw opacity=0] (12,-3) -- (0,0) -- (12,3) -- cycle    ;
\draw [color={rgb, 255:red, 0; green, 0; blue, 0 }  ,draw opacity=1 ]   (230.33,120.67) -- (230.33,76) ;
\draw [shift={(230.33,74)}, rotate = 90] [fill={rgb, 255:red, 0; green, 0; blue, 0 }  ,fill opacity=1 ][line width=0.08]  [draw opacity=0] (12,-3) -- (0,0) -- (12,3) -- cycle    ;
\draw    (183.05,162.85) -- (164.38,180.14) ;
\draw [shift={(162.92,181.5)}, rotate = 317.19] [fill={rgb, 255:red, 0; green, 0; blue, 0 }  ][line width=0.08]  [draw opacity=0] (12,-3) -- (0,0) -- (12,3) -- cycle    ;
\draw  [color={rgb, 255:red, 0; green, 0; blue, 0 }  ,draw opacity=1 ][fill={rgb, 255:red, 155; green, 155; blue, 155 }  ,fill opacity=0 ] (256.33,184.67) .. controls (242,177.95) and (246.75,176.76) .. (243,170) .. controls (241.17,166.7) and (239.05,159.92) .. (228.33,154.67) .. controls (186.11,145.1) and (177,166) .. (172.33,142.67) .. controls (175.67,118) and (201,79.33) .. (305,108) .. controls (320.96,113.77) and (326.05,122.81) .. (324.52,132.7) .. controls (320.75,157.07) and (276.73,186.56) .. (256.33,184.67) -- cycle ;
\draw [color={rgb, 255:red, 0; green, 0; blue, 0 }  ,draw opacity=0.25 ] [dash pattern={on 4.5pt off 4.5pt}]  (173,153.33) .. controls (176.33,165.33) and (179,163.33) .. (198.33,161.33) .. controls (217.67,159.33) and (219.67,162.67) .. (229,165.33) .. controls (238.33,168) and (238.33,193.33) .. (257,195.33) .. controls (275.67,197.33) and (319,168.67) .. (324.43,145.35) ;
\draw    (213.03,139.77) -- (213,152.67) ;
\draw [shift={(213,152.67)}, rotate = 270.15] [color={rgb, 255:red, 0; green, 0; blue, 0 }  ][line width=0.75]    (0,3.35) -- (0,-3.35)(6.56,-1.97) .. controls (4.17,-0.84) and (1.99,-0.18) .. (0,0) .. controls (1.99,0.18) and (4.17,0.84) .. (6.56,1.97)   ;
\draw    (212.67,184.18) -- (212.67,170.18) ;
\draw [shift={(212.67,170.18)}, rotate = 90] [color={rgb, 255:red, 0; green, 0; blue, 0 }  ][line width=0.75]    (0,3.35) -- (0,-3.35)(6.56,-1.97) .. controls (4.17,-0.84) and (1.99,-0.18) .. (0,0) .. controls (1.99,0.18) and (4.17,0.84) .. (6.56,1.97)   ;
\draw [color={rgb, 255:red, 0; green, 0; blue, 0 }  ,draw opacity=0.5 ]   (213,152.67) -- (212.67,170.18) ;
\draw [color={rgb, 255:red, 4; green, 3; blue, 4 }  ,draw opacity=0.5 ] [dash pattern={on 0.84pt off 2.51pt}]  (323.89,200.05) -- (311,214) ;
\draw [color={rgb, 255:red, 0; green, 0; blue, 0 }  ,draw opacity=0.5 ] [dash pattern={on 0.84pt off 2.51pt}]  (311,214) -- (277.34,197.33) ;
\draw    (314.68,163.33) -- (346.53,178.47) ;
\draw [shift={(348.33,179.33)}, rotate = 205.43] [fill={rgb, 255:red, 0; green, 0; blue, 0 }  ][line width=0.08]  [draw opacity=0] (12,-3) -- (0,0) -- (12,3) -- cycle    ;
\draw [color={rgb, 255:red, 0; green, 0; blue, 0 }  ,draw opacity=0.14 ] [dash pattern={on 0.84pt off 2.51pt}]  (230.33,120.67) -- (165.57,178.48) ;
\draw [color={rgb, 255:red, 0; green, 0; blue, 0 }  ,draw opacity=0.15 ] [dash pattern={on 0.84pt off 2.51pt}]  (314.68,163.33) -- (230.33,120.67) ;
\draw [color={rgb, 255:red, 0; green, 0; blue, 0 }  ,draw opacity=1 ]   (172.33,163.33) .. controls (179.67,176) and (181,170.67) .. (196.33,169.33) .. controls (211.67,168) and (221,171.33) .. (227,177.33) .. controls (233,183.33) and (237,202.67) .. (255.67,204.67) .. controls (274.33,206.67) and (318.33,180.67) .. (324.33,158) ;
\draw    (172.33,142.67) -- (172.33,163.33) ;
\draw    (324.52,132.7) -- (324.33,158) ;
\draw [color={rgb, 255:red, 71; green, 4; blue, 132 }  ,draw opacity=1 ]   (324.54,198.45) -- (292.68,182.67) ;
\draw [shift={(326.33,199.33)}, rotate = 206.35] [fill={rgb, 255:red, 71; green, 4; blue, 132 }  ,fill opacity=1 ][line width=0.08]  [draw opacity=0] (12,-3) -- (0,0) -- (12,3) -- cycle    ;
\draw [color={rgb, 255:red, 71; green, 4; blue, 132 }  ,draw opacity=1 ]   (292.68,182.67) -- (278.79,195.95) ;
\draw [shift={(277.34,197.33)}, rotate = 316.27] [fill={rgb, 255:red, 71; green, 4; blue, 132 }  ,fill opacity=1 ][line width=0.08]  [draw opacity=0] (12,-3) -- (0,0) -- (12,3) -- cycle    ;

\draw (349,168.25) node [anchor=north west][inner sep=0.75pt]   [align=left] {$\displaystyle \zeta _{1}$};
\draw (148,178.25) node [anchor=north west][inner sep=0.75pt]   [align=left] {$\displaystyle \zeta _{2}$};
\draw (221.67,55.25) node [anchor=north west][inner sep=0.75pt]  [color={rgb, 255:red, 0; green, 0; blue, 0 }  ,opacity=1 ] [align=left] {$\displaystyle \zeta _{3}$};
\draw (308,214.47) node [anchor=north west][inner sep=0.75pt]  [color={rgb, 255:red, 255; green, 0; blue, 0 }  ,opacity=1 ]  {$\hat{n}$};
\draw (322.33,183.68) node [anchor=north west][inner sep=0.75pt]  [font=\small,color={rgb, 255:red, 71; green, 4; blue, 132 }  ,opacity=1 ]  {$\hat{n}_{1}$};
\draw (269,201.68) node [anchor=north west][inner sep=0.75pt]  [font=\small,color={rgb, 255:red, 71; green, 4; blue, 132 }  ,opacity=1 ]  {$\hat{n}_{2}$};
\draw (200.67,179.83) node [anchor=north west][inner sep=0.75pt]    {$h$};

\end{tikzpicture}
}
	\end{center}
    \label{fig:new_models}
	\vspace{-6mm}\subcaption{Plate scheme.}
  \end{minipage}%
  \begin{minipage}[b]{.66\linewidth}
          \begin{center}
		\include{sch_6}
	\end{center}
    \label{fig:kin_assumption_reddy}
	\vspace{-8mm}\subcaption{Kinematic assumptions}
	\end{minipage} \\[-2mm]
  \caption{Plate scheme and kinematic assumptions}
\label{fig:EB_TIM_assumptions}
\end{figure} 

\subsection{Model assumptions}

\noindent The kinematic assumptions in the first-order shear deformation models (Timoshenko beam and Mindlin's plate) are that the plane sections perpendicular to the neutral axis before deformation remain plane but not necessarily perpendicular to the neutral axis after deformation. In the high-order shear deformation models plane sections perpendicular to the neutral axis before the deformation transform into curved sections after deformation (see Figure \ref{fig:EB_TIM_assumptions}). The so-called Reddy models are third-order shear deformation models due to fact that the curve is described by a third-order polynomial that always satisfies the condition of zero tangential traction boundary conditions on the surfaces of the plate. \\[-3mm]

\noindent The problem starts with the assumption on the displacement field $\textbf{u}(\CX,t)$ given by
$$
\textbf{u}(\CX,t) = \begin{bmatrix}
\textbf{u}_1(\CX,t) \\ \textbf{u}_2(\CX,t) \\ \textbf{u}_3(\CX,t)
\end{bmatrix} = \begin{bmatrix}
-(\x{3} \psi_1(\x{1},\x{2},t) + \x{3}^2 \beta_1(\x{1},\x{2},t) + \x{3}^3 \phi_1(\x{1},\x{2},t) )\\[1mm]
-(\x{3} \psi_2(\x{1},\x{2},t) + \x{3}^2 \beta_2(\x{1},\x{2},t) + \x{3}^3 \phi_2(\x{1},\x{2},t) )\\[1mm] 
w(\x{1},\x{2},t)
\end{bmatrix}
$$
where $\psi_1,\,\psi_2$ are the rotations of normals to midplane about the $\x{1}$ and $\x{2}$ axes, respectively, $w$ is the vertical displacement to the mid-plane, and the functions $\beta_1,\,\beta_2,\, \phi_1,\, \phi_2$ will be determined using the condition that transverse shear stresses vanish on the top and bottom surfaces of the plate, that is $\sigma_{13}(\x{1},\x{2},\x{3}=\pm h/2,t) = 0$ and $\sigma_{23}(\x{1},\x{2},\x{3}=\pm h/2,t) = 0$. For isotropic plates these conditions are equivalent to $\varepsilon_{13}(\x{1},\x{2},\x{3}=\pm h/2,t) = 0$ and $\varepsilon_{23}(\x{1},\x{2},\x{3}=\pm h/2,t) = 0$ \cite{reddy1984simple}. Then, from \eqref{eq:strain_tensor_lineal} we obtain
$$
2\varepsilon_{13} =  \p_1 w -  \left( \psi_1 + 2\x{3}\beta_1 + 3\x{3}^2 \phi_1 \right) \quad , \quad 
2\varepsilon_{23} = \p_2 w -  \left( \psi_2 + 2\x{3}\beta_2 + 3\x{3}^2 \phi_2 \right)
$$
and from $\varepsilon_{13}(\x{1},\x{2},\x{3} = \pm h/2,t) = 0$ and $\varepsilon_{23}(\x{1},\x{2},\x{3} = \pm h/2,t) = 0$ we obtain $\beta_1,\beta_2 = 0$ and $ \phi_1 = \mfrac{4}{3h^2} \left( \p_1 w - \psi_1 \right)$, $ \phi_2 = \mfrac{4}{3h^2} \left( \p_2 w - \psi_2 \right)$,  where the constant $\alpha = 4/(3h^2)$ is introduced. With all the above variables, the displacement field $\textbf{u}(\CX,t)$ of the Reddy's plate is given by 
\begin{equation}
\begin{array}{rl}
\textbf{u}(\CX,t)
= &  \ub{\begin{bmatrix}
-(\x{3}-\alpha \x{3}^3) & 0 & 0 &  -\alpha \x{3}^3  & 0 \\[-2mm]
 0 & -(\x{3}-\alpha \x{3}^3) & 0 & 0  & -\alpha \x{3}^3 \\[-2mm]
  0& 0&  1 & 0  & 0
\end{bmatrix}}{\lambda_1(\X^c)}
\ub{ \begin{bmatrix}
\psi_1(\X,t) \\[-1.5mm] \psi_2(\X,t) \\[-1.5mm] w(\X,t) \\[-1.5mm] \p_1 w(\X,t)  \\[-1.5mm] \p_2 w(\X,t)
\end{bmatrix} }{\textbf{r}(\X,t)} 
\end{array}
\label{eq:reddy_plate_u_field}
\end{equation}
\noindent Note that with $\alpha = 0$ the displacement field becomes the same as Mindlin's plate \cite{brugnoli2019port}. On the other hand related to the constitutive equations, considering an isotropic material, the stress $\sigma(\X,t)$ is obtained from Hooke's law $\sigma(\X,t) = C \, \varepsilon(\X,t)$, where the constitutive matrix $C = C^\top > 0 $ is given by
\begin{equation}
C = \begin{bmatrix}
C_b & 0 \\[-1.5mm] 0 & C_s 
\end{bmatrix}  \quad , \quad \mbox{with } \quad C_b = \mfrac{E}{(1-\nu^2)} \begin{bmatrix}
1 & \nu & 0 \\[-1.5mm] \nu & 1 & 0 \\[-1.5mm] 0 & 0 & \frac{(1-\nu)}{2}
\end{bmatrix}  \quad , \quad   C_s = \begin{bmatrix}
G & 0 \\[-1.5mm] 0 & G
\end{bmatrix}
\label{eq:Constitutive_Reddy_Voigt}
\end{equation}
where $C_b$ is the constitutive matrix for plane stress, and $E,G,\nu$ are material properties.

\subsection{Infinite-dimensional port-Hamiltonian model of the Reddy plate}

\noindent Now that the displacement field is known and written according to \eqref{eq:hip_cinematica}, the proposed methodology allows us to systematically find the infinite-dimensional port-Hamiltonian model of the plate that is energetically consistent with the kinematic assumptions. Following the proposed methodology we have \\
%
%
%

\noindent 1. From \eqref{eq:def_Mass_matrix_distributed} we calculate the mass matrix as
\begin{equation}
\begin{array}{r}
\hspace*{-10mm}\mathcal{M}(\X) =  \rho(\X)\dint_{-h/2}^{h/2}\!\! \begin{bmatrix}
(\x{3} ^2 -2\alpha\x{3}^4 + \alpha^2 \x{3}^6) & 0 & 0 & \alpha(\x{3}^4 - \alpha \x{3}^6) & 0 \\[-1mm]
0 & (\x{3} ^2 -2\alpha\x{3}^4 + \alpha^2 \x{3}^6) & 0 & 0 & \alpha(\x{3}^4 - \alpha \x{3}^6) \\[-1mm]
0 & 0 &  1 & 0 & 0 \\[-1mm]
 \alpha(\x{3}^4 - \alpha \x{3}^6) & 0 & 0 & \alpha^2 \x{3}^6 & 0 \\[-1mm]
0 & \alpha(\x{3}^4 - \alpha \x{3}^6) & 0 & 0 & \alpha^2 \x{3}^6
\end{bmatrix} d\x{3} \\[12mm]
 =  \rho(\X)\begin{bmatrix}
\,(\bar{I}_2-2\alpha \bar{I}_4 + \alpha^2 \bar{I}_6) & 0 & 0 & \alpha (\bar{I}_4-\alpha \bar{I}_6) & 0 \\[-1mm]
0 & (\bar{I}_2-2\alpha \bar{I}_4 + \alpha^2 \bar{I}_6) & 0 & 0 & \alpha (\bar{I}_4-\alpha \bar{I}_6) \\[-1mm]
0 &  0 & \bar{I}_0 & 0 & 0 \\[-1mm]
\alpha (\bar{I}_4-\alpha \bar{I}_6) & 0 & 0 & \alpha^2 \bar{I}_6 & 0\\[-1mm]
0 & \alpha (\bar{I}_4-\alpha \bar{I}_6) & 0 & 0 & \alpha^2 \bar{I}_6
\end{bmatrix} \hspace*{6mm}
\end{array}
\label{eq:Mass_Reddy}
\end{equation}
where $\bar{I}_i \in \R{}$ with $i=0,2,4,\dots$ is defined as 
$
\bar{I}_i = \dint_{-h/2}^{h/2} \x{3}^i \, d\x{3} = \mfrac{h^{i+1}}{2^i(i+1)}
$. 
Then, from \eqref{eq:def_p_distributed} we have
\begin{equation}
p(\X,t) = \mathcal{M}(\X)\dot{\textbf{r}}(\X,t) = \begin{bmatrix} \rho(\bar{I}_2-2\alpha \bar{I}_4 + \alpha^2 \bar{I}_6) \dot{\psi}_1(\X,t)  +  \rho\alpha(\bar{I}_4-\alpha \bar{I}_6)\p_1 \dot{w}(\X,t)  \\[-1mm]
\rho(\bar{I}_2-2\alpha \bar{I}_4 + \alpha^2 \bar{I}_6) \dot{\psi}_2(\X,t)  +  \rho\alpha(\bar{I}_4-\alpha \bar{I}_6)\p_2 \dot{w}(\X,t)  \\[-1mm]
\rho \bar{I}_0 \dot{w}(\X,t) \\[-1mm]
\rho\alpha (\bar{I}_4-\alpha \bar{I}_6) \dot{\psi}_1(\X,t) + \rho \alpha^2 \bar{I}_6 \p_1 \dot{w}(\X,t) \\[-1mm]
\rho\alpha (\bar{I}_4-\alpha \bar{I}_6) \dot{\psi}_2(\X,t) + \rho \alpha^2 \bar{I}_6 \p_2 \dot{w}(\X,t)
\end{bmatrix} = \begin{bmatrix}
{p_1}(\X,t) \\[-1mm] {p_2}(\X,t) \\[-1mm] {p_3}(\X,t) \\[-1mm] {p_4}(\X,t) \\[-1mm] {p_5}(\X,t)
\end{bmatrix} 
\label{eq:momentum_Reddy_Voigt}
\end{equation}

\noindent 2. On the other hand, from \eqref{eq:def_vec_epsilon} we obtain the non-zero strain components $\varepsilon \subset \vec{\varepsilon}$, that is
\begin{equation}
\hspace{-2mm}\varepsilon(\CX,t) \!= \! \begin{bmatrix}
\varepsilon_b \\ \hline \\[-6mm] \varepsilon_s
\end{bmatrix} \! = \! \begin{bmatrix}
\varepsilon_1 \\[-1mm] \varepsilon_2 \\[-1mm] \varepsilon_4 \\ \hline\\[-6mm] \varepsilon_5 \\[-1mm] \varepsilon_6
\end{bmatrix} \! = \! \begin{bmatrix}
-(\x{3}-\alpha \x{3}^3)\, [ \p_1 \psi_1(\X,t)] - \alpha \x{3}^3 [\p_1^2 w(\X,t)] \\[-1mm]
-(\x{3}-\alpha \x{3}^3)\, [ \p_2 \psi_2(\X,t)] - \alpha \x{3}^3 [\p_2^2 w(\X,t)] \\[-1mm] 
-(\x{3}-\alpha \x{3}^3)[\p_2 \psi_1(\X,t) + \p_1 \psi_2(\X,t)] -\alpha \x{3}^3[\p_2\, \p_1 w(\X,t) + \p_1\, \p_2 w(\X,t)] \\[0mm] \hline \\[-5.5mm]
(1-3\alpha \x{3}^2)[\p_1 w(\X,t) - \psi_1(\X,t)] \\[-1mm]
(1-3\alpha \x{3}^2)[\p_2 w(\X,t) - \psi_2(\X,t)]
\end{bmatrix}
\label{eq:eps_highlight}
\end{equation}
\noindent From Proposition \ref{prop:2} we choose $m=8$ since there are eight functions that are independent of $\X^c$ in the strain vector $\varepsilon(\CX,t)$ (highlighted in square brackets in \eqref{eq:eps_highlight}). Then we seek to write  $\varepsilon(\CX,t)$ according to \eqref{eq:factorization_prop_2} considering that $\lambda_2(\X^c) \in \R{d\times m = 5 \times 8}$ and $\mathcal{F}$ of dimension $(m \times n)= (8 \times 5)$. That is
\begin{equation}
\lambda_2(\X^c) = \begin{bmatrix}
-(\x{3}-\alpha \x{3}^3) & 0  & 0 & 0 & 0 & -\alpha \x{3}^2 & 0 & 0 \\[-1mm]
 0 & \!\!\!\!\!\!-(\x{3}-\alpha \x{3}^3) & 0 & 0 & 0 & 0 & -\alpha \x{3}^2 & 0 \\[-1mm]
 0 & 0 & \!\!\!\!\!\!-(\x{3}-\alpha \x{3}^3) & 0 & 0 & 0 & 0 & -\alpha \x{3}^2 \\[-1mm]
 0 & 0 & 0 & \!\!\!\!(1-3\alpha\x{3}^2) & 0 & 0 & 0 & 0\\[-1mm]
 0 & 0 & 0 & 0 & \!\!\!\!(1-3\alpha\x{3}^2) & 0 & 0 & 0
\end{bmatrix}
\end{equation}
$ $\\[-7mm]
\begin{equation}
\;\;\mathcal{F} = \left[ \!\! \begin{array}{c c | c | c c}
\p_1 & 0 & 0 & 0 & 0 \\[-1.5mm] 0 & \p_2 & 0 & 0 & 0 \\[-1.5mm] \p_2 & \p_1 & 0 & 0 & 0 \\ \hline \\[-6mm] -1 & 0 & \p_1 & 0 & 0 \\[-1.5mm] 0 & -1 & \p_2 & 0 & 0 \\ \hline \\[-6mm]
0 & 0 & 0 & \p_1 & 0 \\[-1.5mm]
0 & 0 & 0 & 0 & \p_2 \\[-1.5mm]
0 & 0 & 0 & \p_2 & \p_1
\end{array} \!\! \right]
\;\to \;\;
\q(\X,t) = \mathcal{F}\,\textbf{r}(\X,t) = \begin{bmatrix}
\p_1 \psi_1(\X,t) \\[-1mm] \p_2 \psi_2(\X,t) \\[-1mm] \p_2 \psi_1(\X,t) + \p_1 \psi_2(\X,t) \\[-1mm] \p_1 w(\X,t)-\psi_1(\X,t) \\[-1mm] \p_2 w(\X,t)-\psi_2(\X,t) \\[-1mm] \p_1^2 w(\X,t) \\[-1mm] \p_2^2 w(\X,t) \\[-1mm] \p_2 \p_1 w(\X,t)  + \p_1\p_2 w(\X,t)
\end{bmatrix} = \begin{bmatrix}
 {\q_1}(\X,t) \\[-1mm] {\q_2}(\X,t) \\[-1mm] {\q_3}(\X,t) \\[-1mm] {\q_4}(\X,t) \\[-1mm] {\q_5}(\X,t) \\[-1mm] {\q_6}(\X,t) \\[-1mm] {\q_7}(\X,t) \\[-1mm] {\q_8}(\X,t)
\end{bmatrix} 
\label{eq:F_Reddy_plate}
\end{equation}
with $\mathcal{F}$ a differential operator of order $N=1$. From \eqref{eq:def_Stiffness_matrix_distributed} we have
\begin{equation}
\begin{array}{r}
\mathcal{K}(\X) =  \dint_{-h/2}^{h/2}  \begin{bmatrix}
(\x{3} ^2 -2\alpha\x{3}^4 + \alpha^2 \x{3}^6)  \, C_b & 0 &  \alpha(\x{3}^4 - \alpha \x{3}^6) \, C_b \\[-1mm]
0 & (1-6\alpha\x{3}^2+9\alpha^2\x{3}^4)\,C_s & 0 \\[-1mm]
\alpha(\x{3}^4 - \alpha \x{3}^6) \, C_b & 0 & \alpha^2\x{3}^6 \, C_b
\end{bmatrix}  d\x{3} \\[6mm]
=    \begin{bmatrix}
(\bar{I}_2 -2\alpha\bar{I}_4 + \alpha^2 \bar{I}_6)\,C_b & 0 & \alpha(\bar{I}_4 - \alpha \bar{I}_6) \, C_b \\[-1mm]
0 & (\bar{I}_0-6\alpha\bar{I}_2+9\alpha^2\bar{I}_4)\,C_s & 0 \\[-1mm]
\alpha(\bar{I}_4 - \alpha \bar{I}_6) \, C_b & 0 & \alpha^2\bar{I}_6 \, C_b
\end{bmatrix} \hspace*{5.5mm}
\end{array}
\label{eq:Stiffnes_Reddy}
\end{equation}
$ $\\[-8mm]
\noindent 3. Considering that there are no distributed inputs, from Theorem \ref{theo:PHS_SD} we have
\begin{equation}
\begin{bmatrix}
\\[-6mm]
\dot{p}(\X,t) \\[-1mm] \dot{\q}(\X,t)
\end{bmatrix} = \begin{bmatrix}
\\[-6mm]
0 & -\mathcal{F}^* \\[-1mm] \mathcal{F} & 0
\end{bmatrix} \begin{bmatrix}
\\[-6mm]
\mathcal{M}^{-1}(\X) & 0 \\[-1mm] 0 & \mathcal{K}(\X)
\end{bmatrix} \begin{bmatrix}
\\[-6mm]
{p}(\X,t) \\[-1mm] {\q}(\X,t)
\end{bmatrix} 
\label{eq:Reddy_plate_model_Voigt}
\end{equation}
with Hamiltonian $H[p,\q]$ defined as in \eqref{eq:def_HAMILTONIANO}, and boundary variables defined from \eqref{eq:boundary_ports_COR_1D} as \\[-2mm]
\begin{equation}
\begin{array}{rl}
\p_t{H}  =  
\dint_{\p\Omega} 
\begin{bmatrix}
\\[-7mm]
e_{p_1} \\[-2mm] e_{p_2} \\[-2mm] e_{p_3} \\[-2mm] e_{p_4} \\[-2mm] e_{p_5}
\end{bmatrix}^\top 
\ub{\begin{bmatrix} 
\hat{n}_1 & 0 & \hat{n}_2 & 0 & 0 & 0 & 0 & 0 \\[-1.5mm]
0 & \hat{n}_2 & \hat{n}_1 & 0 & 0 & 0 & 0 & 0 \\[-1.5mm]
0 & 0 & 0 & \hat{n}_1 & \hat{n}_2 & 0 & 0 & 0 \\[-1.5mm]
0 & 0 & 0 & 0 & 0 & \hat{n}_1 & 0 & \hat{n}_2 \\[-1.5mm]
0 & 0 & 0 & 0 & 0 & 0 & \hat{n}_2 & \hat{n}_1
 \end{bmatrix}}{P_\p}
\begin{bmatrix}
\\[-7mm]
e_{\q_1} \\[-2mm] e_{\q_2} \\[-2mm] e_{\q_3} \\[-2mm] e_{\q_4} \\[-2mm] e_{\q_5} \\[-2mm] e_{\q_6} \\[-2mm] e_{\q_7} \\[-2mm] e_{\q_8} 
\end{bmatrix} \, d\sx 
= \dint_{\p \Omega} \!  y_\p^\top u_\p  \, d\sx 
\end{array}
\label{eq:Boundary_variables_Reddy_Voigt}
\end{equation}
%

\begin{remark}
Despite the fact that the methodology uses the Voigt-Kelvin notation (see Appendix \ref{ann:Voigt}) and a Cartesian reference system to obtain the models, using the similarities $\mathbb{L} \sim \mbox{Grad}$, and $-\mathbb{L}^* \sim \mbox{Div}$, then $\mbox{Grad} = - \mbox{Div}^*$ (see \cite[Theorem 4]{brugnoli2019port} for the proof), these models can be written with tensor notation which has the advantage of being independent of the coordinate system.
\end{remark}

\subsection{Tensor representation of the port-Hamiltonian Reddy plate model}

\noindent In order to write the Reddy plate model using tensor notation, first note that the differential operator $\mathcal{F}$ defined in \eqref{eq:F_Reddy_plate} and its formal adjoint $\mathcal{F}^*$ can be written using intrinsic tensor operators, and second, we can rewrite the  generalized displacements $\textbf{r}(\X,t)$ in \eqref{eq:reddy_plate_u_field} as 
\begin{equation}
\mathcal{F} = \begin{bmatrix}
\mbox{Grad} & 0 & 0 \\[-1mm]
-1_2 & \mbox{grad} & 0 \\[-1mm]
0 & 0 & \mbox{Grad}
\end{bmatrix} \quad , \quad 
\mathcal{F}^* = -\begin{bmatrix}
\mbox{Div} & 1_2 & 0 \\[-1mm]
0 & \mbox{div} & 0 \\[-1mm]
0 & 0 & \mbox{Div}
\end{bmatrix} \quad , \quad 
\textbf{r}(\X,t) = \begin{bmatrix}
\psi(\X,t)  \\[-1mm] w(\X,t)  \\[-1mm] \theta(\X,t)
\end{bmatrix} 
\end{equation}
where $\psi(\X,t) = [\psi_1(\X,t) \;\; \psi_2(\X,t)]^\top \in \R{2}$ groups the angles rotated by the cross section with respect to each coordinate axis, $w(\X,t) \in \R{}$ remains representing the vertical displacement of a point in the mid-plane of the plate, and $\theta(\X,t)= [\p_1 w(\X,t) \;\; \p_2 w(\X,t)]^\top \in \R{2}$ groups the first spatial derivatives of $w(\X,t)$ with respect to each coordinate axis. The mass matrix $\mathcal{M}(\X)$ in \eqref{eq:Mass_Reddy} and stiffness matrix  $ {\mathcal{K}}(\X)$ in \eqref{eq:Stiffnes_Reddy} can be rewritten as
\begin{equation}
\mathcal{M}(\X) = \rho(\X)\begin{bmatrix}
c_1 1_2 & 0 & c_2 1_2 \\[-1mm] 0 & \bar{I}_0 & 0 \\[-1mm]
c_2 1_2 & 0 & c_3 1_2
\end{bmatrix}
\quad , \quad
\doubleunderline{\mathcal{K}}(\X) = \begin{bmatrix}
c_1 \doubleunderline{C_{b}}(\cdot) & 0 & c_2 \doubleunderline{C_{b}}(\cdot) \\[-1mm] 0 & c_4 C_s & 0 \\[-1mm]
c_2 \doubleunderline{C_{b}}(\cdot) & 0 & c_3 \doubleunderline{C_{b}}(\cdot)
\end{bmatrix}
\end{equation}
with $c_1 = (\bar{I}_2 -2\alpha\bar{I}_4 + \alpha^2 \bar{I}_6)$, $c_2 = \alpha(\bar{I}_4 - \alpha \bar{I}_6)$, $c_3 = \alpha^2\bar{I}_6$, $c_4 = (\bar{I}_0-6\alpha\bar{I}_2+9\alpha^2\bar{I}_4)$, $C_s$ the constitutive matrix for shear stress as defined in \eqref{eq:Constitutive_Reddy_Voigt}, and $\doubleunderline{C_b}(\cdot)= \frac{E}{1-\nu^2}\left[(1-\nu)(\cdot) + \nu \mbox{tr}(\cdot)1_2 \right]$ the constitutive tensor for plane stress. With the above we redefine the energy variables as
\begin{equation}
\hspace{-0mm} p(\X,t) \! = \! \begin{bmatrix}
p_{\psi}(\X,t) \\[-1mm] p_{w}(\X,t) \\[-1mm] p_{\theta}(\X,t)
\end{bmatrix} \! = \! \begin{bmatrix}
\rho(\X)c_1 \dot{\psi}(\X,t) + \rho(\X)c_2 \dot{\theta}(\X,t) \\[-1mm] \rho(\X)\hat{I}_0 \dot{w}(\X,t) \\[-1mm]
\rho(\X)c_2 \dot{\psi}(\X,t) + \rho(\X)c_3 \dot{\theta}(\X,t)
\end{bmatrix}  , \;\;
\underline{\q}(\X,t) \! = \! \begin{bmatrix}
\q_{\psi}(\X,t) \\[-1mm] \q_{w}(\X,t) \\[-1mm] \q_{\theta}(\X,t)
\end{bmatrix} \! =  \! \begin{bmatrix}
\mbox{Grad}(\psi(\X,t)) \\[-1mm] \mbox{grad}(w(\X,t))-\psi(\X,t) \\[-1mm] \mbox{Grad}(\theta(\X,t)) 
\end{bmatrix}
\end{equation}
and co-energy variables as  
\begin{equation}
{e_p}(\X,t) = \begin{bmatrix}
e_{p_\psi}(\X,t) \\[-1mm] e_{p_w}(\X,t) \\[-1mm] e_{p_\theta}(\X,t)  
\end{bmatrix} = \begin{bmatrix}
\dot{\psi}(\X,t) \\[-1mm] \dot{w}(\X,t) \\[-1mm] \dot{\theta}(\X,t)  
\end{bmatrix}
\quad , \quad
\underline{e_\q}(\X,t) = 
\begin{bmatrix}
e_{\q_\psi}(\X,t) \\[-1mm] e_{\q_w}(\X,t) \\[-1mm] e_{\q_\theta}(\X,t)  
\end{bmatrix} =
\begin{bmatrix}
c_1 \doubleunderline{C_{b}}(\q_\psi(\X,t)) + c_2 \doubleunderline{C_{b}}(\q_\theta(\X,t)) \\[-1mm]
c_4 C_s \, \q_w(\X,t) \\[-1mm]
c_2 \doubleunderline{C_{b}}(\q_\psi(\X,t)) + c_3 \doubleunderline{C_{b}}(\q_\theta(\X,t))
\end{bmatrix}
\end{equation}
where $\q_{\psi}(\X,t)$, $\q_{\theta}(\X,t)$, $e_{\q_\psi}(\X,t)$, $e_{\q_\theta}(\X,t) \in \R{2\times 2}$ are second-order tensor fields, $p_{\psi}(\X,t)$, $p_{\theta}(\X,t)$, $\q_{w}(\X,t)$, $e_{p_{\psi}}(\X,t)$, $e_{p_{\theta}}(\X,t)$, $e_{\q_w}(\X,t) \in \R{2}$ are vector fields, and $p_{w}(\X,t)$, $e_{p_{w}}(\X,t) \in \R{}$ are scalar fields. 
Then, the Reddy's plate model in \eqref{eq:Reddy_plate_model_Voigt} written using tensor notation is given by
\begin{equation}
\begin{array}{r}
\ub{\begin{bmatrix}
\dot{p}_{\psi} \\[-1mm] \dot{p}_{w} \\[-1mm] \dot{p}_{\theta} \\[0.5mm]
\hline \\[-6mm]
\dot{\q}_{\psi} \\[-1mm] \dot{\q}_{w} \\[-1mm] \dot{\q}_{\theta}
\end{bmatrix}}{\dot{x}} =   \ub{\left[ \!\begin{array}{c c c | c c c}
\!0 & \!\!\!\!0 & \!\!\! 0 & \mbox{Div} \!\! & 1_2 \!\!\!\! & 0 \\[-1mm]
\!0 & \!\!\!\!0 & \!\!\! 0 & 0 & \mbox{div} \!\! & 0 \\[-1mm]
\!0 & \!\!\!\!0 & \!\!\! 0 & 0 & 0& \mbox{Div} \!\! \\[0.5mm]
\hline \\[-6mm]
\! \mbox{Grad}  & \!\!\!\! 0 & \! 0 &  0 & 0 & 0 \\[-1mm]
\!\!\!\! -1_2  &  \!\!\!\! \mbox{grad} \!\! & \! 0 & 0 & 0 & 0 \\[-1mm]
\!0 &\!\!\!\! 0 & \!\!\!\! \mbox{Grad} \! & 0 & 0 & 0
\end{array} \! \right]}{\mathcal{J}=-\mathcal{J}^*} \ub{\begin{bmatrix}
{e}_{p_\psi} \\[-1mm] {e}_{p_w} \\[-1mm] {e}_{p_\theta} \\[0.5mm]
\hline \\[-6mm]
e_{\q_{\psi}} \\[-1mm] e_{\q_{w}} \\[-1mm] e_{\q_{\theta}}
\end{bmatrix}}{\var_x H} 
\end{array}
\label{eq:Reddy_plate_Tensor}
\end{equation}
with Hamiltonian functional $H[p,\underline{\q}]>0 \in \R{}$ given by 
\begin{equation}
H[p,\underline{\q}] = \mfrac{1}{2} \dint_{\Omega} \! \left( [\mathcal{M}^{-1}p] \cdot \! p + [c_1 \doubleunderline{C_b}(\q_{\psi}) + c_2 \doubleunderline{C_b}(\q_{\theta})]\!:\!\q_{\psi} + c_4 C_s\q_w \!\cdot \! \q_w + [c_2 \doubleunderline{C_b}(\q_{\psi}) + c_3 \doubleunderline{C_b}(\q_{\theta})]\!:\!\q_{\theta}\right) d\X
\label{eq:Hamiltonian_Reddy_plate_Tensor}
\end{equation}
\noindent The boundary variables are obtained from the energy balance which is given by
\begin{equation}
\begin{array}{rl}
\p_t H =  \langle \var_x H,\, \dot{x} \rangle_{in}^{\Omega}  = & \hspace{-3mm}
 \dint_{\Omega} \left\{ [e_{p_\psi}\!\cdot \mbox{Div}(e_{\q_\psi}) + e_{\q_\psi}\!:\mbox{Grad}(e_{p_\psi})] + [e_{p_\theta}\!\cdot \mbox{Div}(e_{\q_\theta}) + e_{\q_\theta}\!:\mbox{Grad}(e_{p_\theta})] \right. \,+\\[1mm]
 & \left. \hspace{4mm}[e_{p_w}\, \mbox{div}(e_{\q_w}) + e_{\q_w}\!\cdot \mbox{grad}(e_{p_w})] \right\}  d\X
\end{array}
\end{equation}
so applying the integration by parts theorem for symmetric tensors to the first two terms of the integral above (see \cite[Theorem 8]{brugnoli2020port}), and the divergence theorem to the third term, then we obtain
\begin{equation}
\begin{array}{r}
\p_t H =  \dint_{\p\Omega} \left[ (e_{\q_\psi}\, \hat{n})\!\cdot e_{p_\psi} + (e_{\q_\theta}\, \hat{n})\!\cdot e_{p_\theta} + e_{p_w}\,(e_{\q_w}\!\cdot \hat{n}) \right] d\mathtt{s} = \dint_{\p\Omega} y_\p^\top u_\p \,d\mathtt{s}
\end{array}
\label{eq:boundary_Reddy_tensor}
\end{equation}
which is completely analogous to the expression in \eqref{eq:Boundary_variables_Reddy_Voigt}. 
\begin{remark} \label{rem:Reddy_equivalnce}
\noindent Note that the energy and co-energy variables related to the generalized strains in the model \eqref{eq:Reddy_plate_Tensor} can be written in terms of the variables of the model \eqref{eq:Reddy_plate_model_Voigt} as
$$
\begin{array}{rrr}
\q_{\psi}(\X,t) = \begin{bmatrix}
\q_1(\X,t) & \frac{1}{2}\q_3(\X,t) \\[0mm] \frac{1}{2}\q_3(\X,t) & \q_2(\X,t) 
\end{bmatrix} , &
\q_{w}(\X,t) = \begin{bmatrix}
\q_4(\X,t)  \\  \q_5(\X,t) 
\end{bmatrix} , &
\q_{\theta}(\X,t) = \begin{bmatrix}
\q_6(\X,t) & \frac{1}{2}\q_8(\X,t) \\[0mm] \frac{1}{2}\q_8(\X,t) & \q_7(\X,t) 
\end{bmatrix}  \\[6mm]
e_{\q_{\psi}}(\X,t) = \begin{bmatrix}
e_{\q_1}(\X,t) & e_{\q_3}(\X,t) \\[0mm] e_{\q_3}(\X,t) & e_{\q_2}(\X,t) 
\end{bmatrix} , &
e_{\q_{w}}(\X,t) = \begin{bmatrix}
e_{\q_4}(\X,t)  \\  e_{\q_5}(\X,t) 
\end{bmatrix} , &
e_{\q_{\theta}}(\X,t) = \begin{bmatrix}
e_{\q_6}(\X,t) & e_{\q_8}(\X,t) \\[0mm] e_{\q_8}(\X,t) & e_{\q_7}(\X,t) 
\end{bmatrix} 
\end{array}
$$
where $\q_{i}(\X,t)$, $e_{\q_i}(\X,t)$ with $i=1,...,8$ are the energy and co-energy variables related to the generalized strains of the model \eqref{eq:Reddy_plate_model_Voigt}, respectively.
\end{remark}

\begin{remark} 
The third-order shear deformation theory presented in this section is also applicable for beams, and note that dynamic models based on the first-order shear deformation theory are obtained from the Reddy's theory by setting $\alpha = 0$ \cite{reddy2003mechanics}. In this case, $\alpha = 0$ leads to the well known Mindlin plate model, and the port-Hamiltonian representation obtained by this methodology is equivalent to the one obtained first in \cite{macchelli2005port}, and then generalized using tensor notation in \cite{brugnoli2019port}.
\end{remark}

\textbf{Other examples:} More examples are shown in Appendix \ref{ann:examples}. The models presented are the well known: {D.1} One-dimensional elasticity (Truss bar), {D.2} Two-dimensional elasticity, {D.3} Three-dimensional elasticity, {D.4} Mindlin's plate, {D.5} Vibrating string, {D.6} Torsion in circular bars, {D.8$^*$} Euler-Bernoulli beam, {D.9$^*$} Kirchhoff-Love plate. Furthermore, new port-Hamiltonian representations are presented for {D.7} Reddy beam, {D.8} Rayleigh beam, {D.9} Kirchhoff-Rayleigh plate, which to the best of our knowledge is the first time they are presented as port-Hamiltonian systems.


\section{Conclusion and future work}

\noindent In this paper a three-steps methodology is proposed to systematically derive an infinite-dimensional port-Hamiltonian representation of multidimensional linear elastic models, subject to a given class of kinematic assumptions and constitutive relationships. The methodology assumes as a starting point that the displacement field can be factorized, which allows to define in a first step the mass matrix and the generalized momentum variables. In the second step, using the factorization of the non-zero components of the strain tensor, the stiffness matrix is calculated and the differential operator is  characterized, which allowing to define the generalized strain variables. Finally, an energetically consistent port-Hamiltonian representation of the model is proposed. This is mainly demonstrated using Lemma \ref{theo:integration_by_parts_operator} which is defined for the considered class of differential operators and Hamilton's principle.\\[-3mm]\noindent

\noindent It is shown that the proposed methodology allows to derive port-Hamiltonian representations in multidimensional spatial domains. It is effective in finding classical models such as the Timoshenko beam, Mindlin's plate or the general three-dimensional elasticity problem (among others), and other less classical models such as the ones based on more specific kinematic assumptions like the Reddy's third-order shear deformation theory. The main advantages of this procedure regarding the usual existing methods in the literature are: first, that it considerably reduces the amount of algebraic work to derive these models since no variational principle has to be applied and integration by parts has been done once for all, in general, over multidimensional domains. Second, no intuition is required to choose the set of state variables that guarantees the existence of an associated skew-adjoint differential operator, which are explicitly defined. In addition, the structure of the differential operator is determined as soon as the state variables are chosen. Third, an expression of the boundary variables is proposed which allows to define the boundary inputs and boundary outputs ports such that the proposed model satisfy the Stoke-Dirac structure. Lastly, since the method starts from kinematic assumptions and constitutive laws, it is not only suitable for rewriting pre-existing models within the port-Hamiltonian framework, but also potentially allows to directly derive new models in PH form.\\[-3mm]

\noindent As future work we will consider the extension of this methodology to the case of nonlinear elasticity, regarding both geometric and material nonlinearities. Also, the extension to constrained and multiphysics problems can be approached by using Hamilton's principle \cite{bedford1985hamilton}. Furthermore, in the same way that variational methods such as Hamilton's principle unify Lagrangian modeling and finite element discretization, it remains to be studied under what conditions or choices, this methodology unifies both port-Hamiltonian modeling and structure-preserving finite element discretization.

\begin{ack}                               
\noindent The first author acknowledges financial support from ANID/Becas/Doctorado Nacional/2021-21211290 (Chile) and the ISITE-BFC project - CPHS2D (France).  The second author acknowledges the EIPHI Graduate School (contract ANR-17-EURE-0002). The third author acknowledges the MSCA Project MODCONFLEX 101073558 and the ANR Project IMPACTS ANR-21-CE48-0018. The fourth author acknowledges Chilean FONDECYT 1231896 and CONICYT BASAL FB0008 projects.
\end{ack}

\appendix

\section{Lemmas from variational calculus} \label{ann:Lemmas_variational}

\noindent The two following lemmas from variational calculus serve as justification to obtain the equations of motion and the boundary conditions of models based on Hamilton's principle.

\begin{lemma} \label{lemma:1} 
(\cite{gurtin1973linear}, p.224.) Let $\mathcal{W}$ be an inner product space, and consider a $C^0$ field $h: \bar{\Omega}\times[t_1,t_2] \to \mathcal{W}$ with $\bar{\Omega}$ the closure $\bar{\Omega} = \Omega \cup \p \Omega$. If the equation
\begin{equation}
\int_{t_1}^{t_2} \! \int_\Omega h(\X,t) \cdot \eta(\X,t) \, d\X \,dt = 0
\label{eq:fundamental_lemma_variational_1}
\end{equation}
holds for every $C^\infty$ field $\eta:\bar{\Omega}\times [t_1,t_2]\to \mathcal{W}$ that vanishes at time $t_1$, at time $t_2$, and on $\p\Omega$, then $h(\X,t)=0$ on $\bar{\Omega}\times [t_1,t_2]$. 
\end{lemma}

\begin{lemma} \label{lemma:2}
(\cite{gurtin1973linear}, p.224.) Suppose that $\p\Omega$ consists of complementary regular sub-surfaces $\p\Omega_u$ and $\p\Omega_\sigma$. Let $\mathcal{W}$ be an inner product space, and consider a function $g: \p\Omega_\sigma \times[t_1,t_2] \to \mathcal{W}$ that is piecewise regular and continuous in time. If the equation
\begin{equation}
\int_{t_1}^{t_2} \! \int_{\p\Omega_\sigma} g(\sx,t) \cdot \eta(\sx,t) \, d\sx \, dt = 0
\label{eq:fundamental_lemma_variational_2}
\end{equation}
holds for every $C^\infty$ field $\eta:\bar{\Omega}\times [t_1,t_2]\to \mathcal{W}$ that vanishes at time $t_1$, at time $t_2$, and on $\p\Omega_u$, then $g(\sx,t)=0$ on $ \p\Omega_\sigma \times [t_1,t_2]$.
\end{lemma}

\section{Voigt-Kelvin notation}\label{ann:Voigt}

\noindent Since the stress and strain tensors are symmetric, they only have six independent components. The Voigt-Kelvin notation defines
$$
\begin{array}{ccclll}
\sigma_1 = \sigma_{11}, &\sigma_2 = \sigma_{22}, &\sigma_3 = \sigma_{33}, &\sigma_4 = \sigma_{12}, &\sigma_5 = \sigma_{13}, &\sigma_6 = \sigma_{23} \\
\varepsilon_1 = \varepsilon_{11}, &\varepsilon_2 = \varepsilon_{22},&\varepsilon_3 = \varepsilon_{33},&\varepsilon_4 = 2\varepsilon_{12},&\varepsilon_5 = 2\varepsilon_{13},&\varepsilon_6 = 2\varepsilon_{23}
\end{array}
$$

where the independent components of both tensors are grouped into the so-called Voigt-stress vector $\vec{\sigma}(\CX,t) \in \R{6}$ and Voigt-strain vector $\vec{\varepsilon}(\CX,t) \in \R{6}$, which are respectively given by
\begin{align}
\vec{\sigma}(\CX,t) = & \begin{bmatrix}
\sigma_1(\CX,t) & \cdots & \sigma_6(\CX,t)
\end{bmatrix}^\top \\
\vec{\varepsilon}(\CX,t) = & \begin{bmatrix}
\varepsilon_1(\CX,t) & \cdots & \varepsilon_6(\CX,t)
\end{bmatrix}^\top 
\end{align}
In addition, using the Voigt-Kelvin notation it is possible to express the constitutive relation as 
$$
\vec{\sigma}(\CX,t) = {C}\, \vec{\varepsilon}(\CX,t)  \quad \sim \quad  \underline{\sigma}(\CX,t)=\doubleunderline{C}: \underline{\varepsilon}(\CX,t)
$$
where ${C} = {C}^\top > 0 \in \R{6 \times 6}$ is a constitutive matrix. For example, for isotropic materials the constitutive fourth-order tensor for 3D elasticity ($\doubleunderline{C_{3D}}(\cdot) = 2\mu(\cdot) + \lambda\mbox{tr}(\cdot)1_3 $), and for plane stress in 2D ($\doubleunderline{C_{2D}}(\cdot) = \frac{E}{1-\nu^2}\left[ (1-\nu)(\cdot) + \nu \mbox{tr}(\cdot) 1_2 \right] $) reduce to
$$
\begin{array}{rcl}
\underline{\sigma_{3D}} = \doubleunderline{C_{3D}}: \underline{\varepsilon_{3D}}& \Leftrightarrow &  \begin{bmatrix}
\sigma_{11} \\[-1.5mm] \sigma_{22} \\[-1.5mm] \sigma_{33} \\[-1.5mm] \sigma_{12} \\[-1.5mm] \sigma_{13} \\[-1.5mm] \sigma_{23} 
\end{bmatrix} = \begin{bmatrix}
2\mu + \lambda & \lambda & \lambda & 0 & 0 & 0 \\[-1.5mm] 
\lambda & 2\mu +  \lambda & \lambda & 0 & 0 & 0 \\[-1.5mm] 
\lambda & \lambda & 2\mu + \lambda & 0 & 0 & 0 \\[-1.5mm] 
0 & 0 & 0 & \mu & 0 & 0 \\[-1.5mm]
0 & 0 & 0 & 0 & \mu & 0 \\[-1.5mm]
0 & 0 & 0 & 0 & 0 & \mu  
\end{bmatrix}\begin{bmatrix}
\varepsilon_{11} \\[-1.5mm] \varepsilon_{22} \\[-1.5mm] \varepsilon_{33} \\[-1.5mm] 2\varepsilon_{12} \\[-1.5mm] 2\varepsilon_{13} \\[-1.5mm] 2\varepsilon_{23} 
\end{bmatrix} \\[14mm]
\underline{\sigma_{2D}} = \doubleunderline{C_{2D}}: \underline{\varepsilon_{2D}}& \Leftrightarrow &  \begin{bmatrix}
\sigma_{11} \\[-1.5mm] \sigma_{22}  \\[-1.5mm] \sigma_{12}
\end{bmatrix} = \mfrac{E}{1-\nu^2}\begin{bmatrix} 
1   & \nu & 0 \\[-1.5mm]
\nu & 1   & 0 \\[-1.5mm] 
0   & 0   & \frac{1-\nu}{2}
\end{bmatrix}\begin{bmatrix}
\varepsilon_{11} \\[-1.5mm] \varepsilon_{22} \\[-1.5mm] 2\varepsilon_{12} 
\end{bmatrix}
\end{array}
$$ 

\noindent where $\mu$ and $\lambda$ are the Lamé constants defined as $\mu = \frac{E}{2(1+\nu)} = G, \; \lambda = \frac{\nu E}{(1+\nu)(1-2\nu)} 
$, where $E$  is the Young modulus, $\nu$ is the Poisson's ratio (ratio between transverse elongation and axial shortening) and the constant $\mu = G$ is also known as the shear modulus. Note also that the tensor contraction now reduces to 
\begin{equation}
\underline{\sigma}(\CX,t):\underline{\varepsilon}(\CX,t) \;=\; \vec{\sigma}(\CX,t)\cdot \vec{\varepsilon}(\CX,t) 
\end{equation}
Assuming a Cartesian reference system $\CX = \lbrace \x{1},\x{2},\x{3} \rbrace$, the Voigt-strain vector can be obtained by 
\begin{equation}
\vec{\varepsilon}(\CX,t) = \mathbb{L}\,\textbf{u}(\CX,t) \quad \sim \quad \underline{\varepsilon}(\CX,t) = \Grad(\textbf{u}(\CX,t))
\label{eq:def_vec_epsilon}
\end{equation}
with $\mathbb{L}$ a linear differential operator of dimension $(6\times 3)$ given by 
\begin{equation}
\mathbb{L} = \begin{bmatrix}
\p_1 &    0 & 0    \\[-1.5mm]
0    & \p_2 & 0    \\[-1.5mm]
0    &   0  & \p_3 \\[-1.5mm]
\p_2 & \p_1 & 0    \\[-1.5mm]
\p_3 &   0  & \p_1 \\[-1.5mm]
0    & \p_3 & \p_2 
\end{bmatrix} \; \sim \; \mbox{Grad}, \quad  \quad \mbox{with } \p_k = \mfrac{\p}{\p \x{k}^{}}, \; k = \lbrace 1,2,3 \rbrace
\label{eq:def_operator_GRAD}
\end{equation} 
\noindent This is also true for two-dimensional tensor fields, where in that case $\mathbb{L}_{2D}$ is of dimension ($3 \times 2$) and is given by
\begin{equation}
\mathbb{L}_{2D} = \begin{bmatrix}
\p_1 &    0     \\[-1.5mm]
0    & \p_2     \\[-1.5mm]
\p_2 & \p_1  
\end{bmatrix} \; \sim \; \mbox{Grad}, \quad  \quad \mbox{with } \p_k = \mfrac{\p}{\p \x{k}^{}}, \; k = \lbrace 1,2 \rbrace
\label{eq:def_operator_GRAD_2D}
\end{equation}
\noindent Note that $\mbox{Div}(\underline{\sigma})= \nabla \cdot \underline{\sigma} = \mathbb{L}^\top\vec{\sigma} = -\mathbb{L}^*\vec{\sigma} \sim - \mbox{Grad}^*(\underline{\sigma})$, where Div($\cdot$) is the tensor divergence operator. Then by similarity it can be seen that $\mbox{Div} = - \mbox{Grad}^*$ (see \cite[Theorem 4]{brugnoli2019port} for the proof).

\section{Proof of Remark \ref{rem:Legendre_PHS} } \label{ann:proof_REMARK_LEGENDRE}
\noindent The kinetic energy $T[\dot{\textbf{r}}]$ and the elastic potential energy $U[{\textbf{r}}]$ are defined as
\begin{equation}
T[\dot{\textbf{r}}] = \dint_\Omega \!\! \mathscr{T}[\dot{\textbf{r}}] \, d\X = \mfrac{1}{2}\dint_\Omega \dot{\textbf{r}}^\top \mathcal{M}\, \dot{\textbf{r}} \, d\X \quad , \quad U[{\textbf{r}}] = \dint_\Omega \!\! \mathscr{U}[{\textbf{r}}] \, d\X = \mfrac{1}{2} \dint_\Omega (\mathcal{F} \textbf{r} )^\top \mathcal{K}\, (\mathcal{F} \textbf{r}) \, d\X 
\label{eq:ann_definition_T_U}
\end{equation}
then, the Hamiltonian density $\mathscr{H}[\dot{\textbf{r}},\textbf{r}]$ is defined as the Legendre transform of $\mathscr{L}[\dot{\textbf{r}},\textbf{r}] = \mathscr{T}[\dot{\textbf{r}}] - \mathscr{U}[\textbf{r}]$. So, applying the Legendre operator to $\mathscr{L}[\dot{\textbf{r}},\textbf{r}]$ we have
$$
\begin{array}{rl}
\mathscr{H}[\dot{\textbf{r}},\textbf{r}] = \left[\dot{\textbf{r}}\cdot \mfrac{\p}{\p \dot{\textbf{r}}}-1\right]( \mathscr{L}[\dot{\textbf{r}},\textbf{r}])  =  \mfrac{1}{2} \dot{\textbf{r}}^\top \mathcal{M}\, \dot{\textbf{r}} + \mfrac{1}{2} (\mathcal{F} \textbf{r} )^\top \mathcal{K}\, (\mathcal{F} \textbf{r}) = \mathscr{T}[\dot{\textbf{r}}] + \mathscr{U}[\textbf{r}]
\end{array}
$$
then, by definition the momentum variable is given by $ p = \p \mathscr{L}/\p \dot{\textbf{r}} = \mathcal{M} \dot{\textbf{r}} $, which implies  $ \dot{\textbf{r}} = \mathcal{M}^{-1}\, p
$ and the Hamiltonian respect to $z = [p^\top \;\; \textbf{r}^\top]^\top$ is the total energy defined in \eqref{eq:HAMILTONIAN_Z}.
On the other hand, to proof \eqref{eq:var_Hr} first consider the following generic functionals 
$$
\begin{array}{rl}
U[\textbf{r}] =  \dint_\Omega g(\X,\textbf{r}(\X),\mathcal{F}\textbf{r}(\X)) \, d\X , \quad
U^\star[\textbf{r}^\star] =  U[\textbf{r}(\X) \!+\!\alpha\gamma(\X)] = \dint_\Omega g(\X,\textbf{r}^\star,(\mathcal{F}\textbf{r}^\star)) \, d\X
\end{array}
$$
where $\gamma(\X)$ is an arbitrary function that vanishes in $\p\Omega$, and $\alpha \in \R{}$ an scalar. By definition the first variation of $U$ is given by $\var U = \lim_{\alpha\to 0} dU^\star/d\alpha $, then we have
$$
\begin{array}{rl}
\var U = \displaystyle\lim_{\alpha \to 0} \mfrac{d}{d\alpha} \dint_\Omega \! g(\X,r^\star,(\mathcal{F}\textbf{r}^\star)) d\X = 
 \displaystyle\lim_{\alpha \to 0} \dint_\Omega \! \left( \mfrac{\p g}{\p \textbf{r}^\star} \!\cdot\! \mfrac{\p \textbf{r}^\star}{\p \alpha}  \!+\!  \mfrac{\p g}{\p (\mathcal{F}\textbf{r}^\star)} \!\cdot\! \mfrac{\p (\mathcal{F}\textbf{r}^\star)}{\p \alpha} \right) d\X
=  \displaystyle\lim_{\alpha \to 0} \dint_\Omega \! \left( \mfrac{\p g}{\p \textbf{r}^\star} \!\cdot\! \gamma  \!+\!  v \!\cdot\! \mathcal{F}\gamma \right) d\X 
\end{array}
$$
with $v = \frac{\p g}{\p (\mathcal{F}\textbf{r}^\star)}$, so applying Theorem \ref{theo:integration_by_parts_operator} to the last term above we have
$$
\begin{array}{rl}
\var U =  \displaystyle\lim_{\alpha \to 0} \left( \dint_\Omega \mfrac{\p g}{\p \textbf{r}^\star} \cdot \gamma \, d\X + \dint_\Omega \gamma^\top \mathcal{F}^*v \, d\X +  \int_{\p \Omega} \! \mathcal{B} (\gamma (\sx))^\top   \mathcal{Q}(\sx) \, \mathcal{B}( v(\sx)) \, d \mathtt{s} \right)
\end{array}
$$
where the last term is equal to zero because $\gamma = 0 $ in $\p \Omega$. Then,
$$
\begin{array}{rl}
\var U =  \displaystyle\lim_{\alpha \to 0} \dint_\Omega \gamma \cdot \left[ \mfrac{\p g}{\p \textbf{r}^\star} + \mathcal{F}^*\left( \mfrac{\p g}{\p (\mathcal{F}\textbf{r}^\star)}\right)  \right]\, d\X =  \dint_\Omega \var \textbf{r} \cdot \left[ \mfrac{\p g}{\p \textbf{r}} + \mathcal{F}^*\left( \mfrac{\p g}{\p (\mathcal{F}\textbf{r})}\right)  \right]  \, d\X = \dint_\Omega \var \textbf{r} \cdot \mfrac{\var U}{\var \textbf{r}}  \, d\X
\end{array}
$$
finally, from the last expression and considering $U[\textbf{r}]$ as in \eqref{eq:ann_definition_T_U}, we obtain that $\frac{\var U}{ \var \textbf{r}} = \frac{\var H}{ \var \textbf{r}}  = \mathcal{F}^*(\mathcal{K}\, \mathcal{F}\textbf{r})$.

\section{Other examples}\label{ann:examples}

The models presented here are: \textbf{D.1} One-dimensional elasticity (Truss bar), \textbf{D.2} Two-dimensional elasticity, \textbf{D.3} Three-dimensional elasticity, \textbf{D.4} Mindlin's plate, \textbf{D.5} Vibrating string, \textbf{D.6} Torsion in circular bars, \textbf{D.7} Reddy beam, \textbf{D.8} Rayleigh beam, \textbf{D.8$^*$} Euler-Bernoulli beam, \textbf{D.9} Kirchhoff-Rayleigh plate, and \textbf{D.9$^*$} Kirchhoff-Love plate.

\titleformat*{\subsection}{\normalfont\bfseries}
\numberwithin{equation}{subsection}
\numberwithin{figure}{subsection}

\subsection{One-dimensional elasticity (truss bar)}

\noindent Consider a three-dimensional body that can be treated as a one-dimensional structure as shown in Figure \ref{fig:papa_1D}. \noindent The displacement field is given by
\begin{equation}
\textbf{u}(\CX,t) = \ub{\begin{bmatrix}
1   \\[-2mm] 0   \\[-2mm] 0 
\end{bmatrix}}{\lambda_1(\X^c)} \ub{  u_1(\X,t) }{\textbf{r}(\X,t)} 
\end{equation}
with $\X = \lbrace \x{1} \rbrace $, $\X^c = \lbrace \x{2},\x{3} \rbrace$, $\Omega = (a,b) \subset \R{}$, $\Omega^c = A \subset \R{2}$, and $u_1(\X,t) \in \R{}$ the displacement in the direction of the $\x{1}$ axis. From \eqref{eq:def_Mass_matrix_distributed} and \eqref{eq:def_p_distributed} we have
\begin{equation}
\mathcal{M}(\X) = \rho(\X) \dint_{A}\!\! 1 \, dA = \rho(\X)A(\X)
\quad ,\quad 
p(\X,t) = \mathcal{M}(\X)\dot{\textbf{r}}(\X,t) = \rho(\X) A(\X)\dot{u}_1(\X,t)
\end{equation}
where $A(\X)$ is the area of the cross section. From \eqref{eq:def_vec_epsilon}, the non-zero components of the strain tensor $\varepsilon \subset \vec{\varepsilon}$ are given by
\begin{equation}
\varepsilon(\CX,t) = 
\varepsilon_1(\X,t)  = 
\p_1 u_1(\X,t) 
\end{equation}
From Proposition \ref{prop:2} we choose $m=1$, then 
\begin{equation}
\varepsilon(\CX,t) = \ub{ 1 }{\lambda_2(\X^c)} \ub{ \p_1 }{\mathcal{F}} \ub{ u_1(\X,t) }{\textbf{r}(\X,t)} \quad \to \quad \q(\X,t) = \mathcal{F}\,\textbf{r}(\X,t) = \varepsilon_1(\X,t)
\end{equation}
with $\mathcal{F}$ a differential operator of order $N=1$. The stress $\sigma(\CX,t)$ is obtained from $\sigma(\CX,t) = C \, \varepsilon(\CX,t) \in \R{}$, where the constitutive matrix is given by $C = E$ with $E$ the Young's modulus. Then, from \eqref{eq:def_Stiffness_matrix_distributed} we have 
$
\mathcal{K}(\X) =  E \int_{A}  1  \,dA =  EA(\X)  
$.
Considering that there are no distributed inputs, from Theorem \ref{theo:PHS_SD} we have
\begin{equation}
\begin{bmatrix}
\dot{p} \\[-1mm] \dot{\q} 
\end{bmatrix} = \begin{bmatrix}
0 & \p_1 \\[-1mm] \p_1 & 0
\end{bmatrix} \begin{bmatrix}
e_{p} \\[-1mm] e_{\q} 
\end{bmatrix} 
\end{equation}
with Hamiltonian defined as 
\begin{equation}
H[p,\q] = \mfrac{1}{2}\dint_{a}^{b} \mfrac{p^2}{\rho A} +  EA \q^2 \, d\X
\end{equation}

Since $\Omega$ is the segment $(a,b)$, the boundary points are $\sx = \lbrace a, b \rbrace$, then $\hat{n}_{1}(\sx = a)=-1$ and $\hat{n}_{1}(\sx = b)=1$. The boundary ports are defined from Remark \ref{rem:boundary_ports_N1} as
\begin{equation}
\p_t{H} =  \dint_{\p\Omega} \!\! e_p\, e_\q \,\hat{n}_{1}\, d\mathtt{s}
= ( e_p\, e_\q)|_{a}^b =  -e_p(a) e_\q(a)+ e_p(b)e_\q(b) = \dint_{\p \Omega} \! y_\p^\top u_\p \,d\sx
\end{equation}


\subsection{Two-dimensional elasticity}

Consider a three-dimensional body that can be treated as a two-dimensional structure as shown in Figure \ref{fig:papa_2D} where $h \in \R{}$ is the thickness (constant). The displacement field is given by
\begin{equation}
\textbf{u}(\CX,t) = \ub{\begin{bmatrix}
1 & 0  \\[-2mm] 0 & 1  \\[-2mm] 0 & 0 
\end{bmatrix}}{\lambda_1(\X^c)} \ub{\begin{bmatrix}
u_1(\X,t) \\[-2mm] u_2(\X,t)
\end{bmatrix}}{\textbf{r}(\X,t)} 
\end{equation}
with $\X = \lbrace \x{1}, \x{2} \rbrace $, $\X^c = \lbrace \x{3} \rbrace$, $\Omega \subset \R{2}$, $\Omega^c = (-\frac{h}{2}\times \frac{h}{2}) \subset \R{}$, $u_1(\X,t) \in \R{}$ and $u_2(\X,t) \in \R{}$ are the displacements in the direction of the $\x{1}$ and $\x{2}$ axes, respectively. From \eqref{eq:def_Mass_matrix_distributed} and \eqref{eq:def_p_distributed} we have
\begin{equation}
\mathcal{M}(\X) = \rho(\X) \dint_{-h/2}^{h/2}\!\! \begin{bmatrix}
1 & 0 \\[-2mm] 0 & 1
\end{bmatrix} d\x{3} = \rho(\X)h I_{2 \times 2}
\quad , \quad
p(\X,t) = \mathcal{M}(\X)\dot{\textbf{r}}(\X,t) = \begin{bmatrix}
\rho(\X)h\,\dot{u}_1(\X,t) \\[-1mm] \rho(\X)h\,\dot{u}_2(\X,t)
\end{bmatrix}
\end{equation}
From \eqref{eq:def_vec_epsilon}, the non-zero components of the strain tensor $\varepsilon \subset \vec{\varepsilon}$ are given by
\begin{equation}
{\varepsilon}(\CX,t)= \begin{bmatrix}
{\varepsilon}_1(\X,t) \\[-1mm] {\varepsilon}_2(\X,t) \\[-1mm] {\varepsilon}_4(\X,t) \end{bmatrix} = \begin{bmatrix}
\p_1 u_1(\X,t) \\[-1mm] \p_2 u_2(\X,t) \\[-1mm] \p_2 u_1(\X,t) + \p_1 u_2(\X,t)
\end{bmatrix}
\end{equation}
From Proposition \ref{prop:2} we choose $m=3$, then 
\begin{equation}
\varepsilon(\CX,t) = \ub{\begin{bmatrix}
1 & 0  & 0 \\[-2mm] 0 & 1 & 0 \\[-2mm] 0 & 0 & 1
\end{bmatrix}}{\lambda_2(\X^c)} \ub{\begin{bmatrix}
\p_1 & 0 \\[-2mm] 0 & \p_2 \\[-2mm] \p_2 & \p_1
\end{bmatrix}}{\mathcal{F}} \ub{\begin{bmatrix}
u_1(\X,t) \\[-1mm] u_2(\X,t)
\end{bmatrix}}{\textbf{r}(\X,t)} \quad \to \quad \q(\X,t) = \mathcal{F}\,\textbf{r}(\X,t) =  \begin{bmatrix}
{\varepsilon}_1(\X,t) \\[-1mm] {\varepsilon}_2(\X,t) \\[-1mm] {\varepsilon}_4(\X,t) \end{bmatrix}
\end{equation}
with $\mathcal{F}$ a differential operator of order $N=1$. The stress $\sigma(\CX,t)$ is obtained from $\sigma(\CX,t) = C_{2D} \, \varepsilon(\CX,t) \in \R{3}$, where  $C_{2D} \in \R{3 \times 3}$ represent a constitutive matrix according to plane stress hypothesis, or  plane strain hypothesis. Them, from \eqref{eq:def_Stiffness_matrix_distributed} we have $ \mathcal{K}(\X) = \int_{-h/2}^{h/2}  C_{2D}  \,d\x{3} = h\, C_{2D}  $. 
Considering that there are no distributed inputs, from Theorem \ref{theo:PHS_SD} we have
\begin{equation}
\begin{bmatrix}
\dot{p}_1 \\[-1.5mm] \dot{p}_2 \\[-1.5mm] \dot{\q}_1 \\[-1.5mm] \dot{\q}_2 \\[-1.5mm] \dot{\q}_3
\end{bmatrix} = \left[ \begin{array}{c c | c c c}
0 & 0 &  \p_1 & 0 & \p_2 \\[-1.5mm]
0 & 0 & 0 & \p_2 & \p_1 \\[0mm]
\hline \\[-6mm]
\p_1 & 0 & 0 & 0 & 0 \\[-1.5mm]
0 & \p_2 & 0 & 0 & 0 \\[-1.5mm]
\p_2 & \p_1 & 0 & 0 & 0
\end{array} \right] \begin{bmatrix}
e_{p_1} \\[-1.5mm] e_{p_2} \\[-1.5mm] e_{\q_1} \\[-1.5mm] e_{\q_2} \\[-1.5mm] e_{\q_3}
\end{bmatrix} 
\label{eq:2D_elast}
\end{equation}
with Hamiltonian defined as in \eqref{eq:def_HAMILTONIANO}. The boundary ports are defined from Remark \ref{rem:boundary_ports_N1} by
\begin{equation}
\begin{array}{rl}
\p_t{H}  =  
\dint_{\p\Omega} 
\begin{bmatrix}
\\[-7mm]
e_{p_1} \\[-2mm] e_{p_2}
\end{bmatrix}^\top 
\ub{\begin{bmatrix} 
\hat{n}_1 & 0 & \hat{n}_2 \\[-1.5mm]
0 & \hat{n}_2 & \hat{n}_1 
\end{bmatrix}}{P_\p}
\begin{bmatrix}
\\[-7mm]
e_{\q_1} \\[-2mm] e_{\q_2} \\[-2mm] e_{\q_3} 
\end{bmatrix} d\sx 
= \dint_{\p \Omega} \!  y_\p^\top u_\p  \, d\sx 
\end{array}
\label{eq:2D_elast_BC}
\end{equation}

$ $\\[-15mm]
\subsubsection{Tensor representation}

Noticing that $\F = \mbox{Grad}$, and $\F^* = -\mbox{Div}$, we redefine the following variables
$$
\underline{\q}(\X,t) = \mbox{Grad}(\textbf{r}(\X,t)) = \begin{bmatrix}
\q_1(\X,t) & \frac{1}{2}\q_3(\X,t) \\ \frac{1}{2}\q_3(\X,t) & \q_2(\X,t)
\end{bmatrix} 
, \quad 
\doubleunderline{K}(\X) = h \doubleunderline{C_{2D}}(\cdot) 
\to \quad
e_{\underline{\q}}(\X,t) = h \doubleunderline{C_{2D}}(\underline{\q}(\X,t)) = \begin{bmatrix}
e_{\q_1}(\X,t) & e_{\q_3}(\X,t) \\ e_{\q_3}(\X,t) & e_{\q_2}(\X,t)
\end{bmatrix} 
$$
with $\doubleunderline{C_{2D}}(\cdot) = \frac{E}{1-\nu^2}\left[ (1-\nu)(\cdot) + \nu \mbox{tr}(\cdot) 1_2 \right] $ for isotropic materials. With tha above, the model \eqref{eq:2D_elast} written using tensor notation is given by 
\begin{equation}
\begin{bmatrix}
\dot{p} \\[-1mm] \dot{\underline{\q}}
\end{bmatrix} = \begin{bmatrix}
0 & \Div \\[-1mm] \Grad & 0 
\end{bmatrix} \begin{bmatrix}
e_p \\[-1mm] e_{\underline{\q}}
\end{bmatrix}
\end{equation}
with Hamiltonian functional $H[p,\underline{\q}] $ given by 
\begin{equation}
H[p,\underline{\q}] = \mfrac{1}{2} \dint_{\Omega} \! \mfrac{1}{\rho h} \, p \cdot \! p + h \doubleunderline{C_{2D}}(\underline{\q}) \!:\! \underline{\q}  \, d\X
\end{equation}
\noindent The boundary variables are obtained from the energy balance which is given by
\begin{equation}
\begin{array}{rl}
\p_t H =  \langle \var_x H,\, \dot{x} \rangle_{in}^{\Omega}  = & \!\!
 \dint_{\Omega} e_{p}\!\cdot \mbox{Div}(e_{\underline{\q}}) + e_{\underline{\q}}\!:\mbox{Grad}(e_{p})  \,  d\X
\end{array}
\end{equation}
so applying the integration by parts theorem for symmetric tensors (see \cite[Theorem 8]{brugnoli2020port}),  we obtain
\begin{equation}
\begin{array}{r}
\p_t H =  \dint_{\p\Omega}  (e_{\underline{\q}}\, \hat{n})\!\cdot e_{p} \,  d\mathtt{s} = \dint_{\p\Omega} y_\p^\top u_\p \,d\mathtt{s}
\end{array}
\end{equation}
where $ \hat{n} = [\hat{n}_1 \;\; \hat{n}_2]^\top $. The above expression is completely analogous to the expression in \eqref{eq:2D_elast_BC}.


\subsection{Three-dimensional elasticity}

Consider a three-dimensional body as shown in Figure \ref{fig:papa_3D}. The displacement field is given by
\begin{equation}
\textbf{u}(\CX,t) = \ub{\begin{bmatrix}
1 & 0 & 0  \\[-2mm] 0 & 1  & 0 \\[-2mm] 0 & 0 & 1
\end{bmatrix}}{\lambda_1(\X^c)} \ub{\begin{bmatrix}
u_1(\X,t) \\[-2mm] u_2(\X,t) \\[-2mm] u_3(\X,t)
\end{bmatrix}}{\textbf{r}(\X,t)} 
\end{equation}
with $\X = \lbrace \x{1}, \x{2} , \x{3} \rbrace = \CX $, $\X^c = \lbrace \phi \rbrace$, $\Omega \subset \R{3}$, $u_1(\X,t) \in \R{}$, $u_2(\X,t) \in \R{}$, and $u_3(\X,t) \in \R{}$ are the displacements in the direction of the $\x{1}$, $\x{2}$ and $\x{3}$ axes, respectively. From \eqref{eq:def_Mass_matrix_distributed} and \eqref{eq:def_p_distributed} we have
\begin{equation}
\mathcal{M}(\X) = \rho(\X)I_{3 \times 3}
\quad , \quad
p(\X,t) = \mathcal{M}(\X)\dot{\textbf{r}}(\X,t) = \rho(\X) \dot{\textbf{u}}(\X,t)
\end{equation}
From \eqref{eq:def_vec_epsilon}, the non-zero components of the strain tensor $\varepsilon \subset \vec{\varepsilon}$ are given by ${\varepsilon}(\CX,t) = \vec{\varepsilon}(\CX,t) = \mathbb{L} \textbf{u} (\X,t) $. 
From Proposition \ref{prop:2} we choose $m=6$, then 
\begin{equation}
\varepsilon(\CX,t) = \ub{I_{3 \times 3}}{\lambda_2(\X^c)} \ub{ \mathbb{L}}{\mathcal{F}} \ub{\textbf{u}}{\textbf{r}(\X,t)} \quad \to \quad \q(\X,t) = \mathcal{F}\,\textbf{r}(\X,t) =  \vec{\varepsilon}(\X,t)
\end{equation}
with $\mathcal{F}$ a differential operator of order $N=1$. The stress $\sigma(\CX,t)$ is obtained from $\sigma(\CX,t) = C_{3D} \, \varepsilon(\CX,t) \in \R{6}$, where  $C_{3D} \in \R{6 \times 6}$ is the constitutive matrix for 3D elasticity. Then, from \eqref{eq:def_Stiffness_matrix_distributed} we have $ \mathcal{K}(\X) =  C_{3D}  $. 
Considering that there are no distributed inputs, from Theorem \ref{theo:PHS_SD} we have
\begin{equation}
\begin{bmatrix}
\dot{p} \\[-1mm] \dot{\q}
\end{bmatrix} = \begin{bmatrix}
0 & -\mathbb{L}^* \\[-1mm] \mathbb{L} & 0 
\end{bmatrix} \begin{bmatrix}
e_p \\[-1mm] e_\q
\end{bmatrix} 
\label{eq:3D_elast}
\end{equation}
with Hamiltonian defined as in \eqref{eq:def_HAMILTONIANO}. The boundary ports are defined from Remark \ref{rem:boundary_ports_N1} by
\begin{equation}
\begin{array}{rl}
\p_t{H}  =  
\dint_{\p\Omega} 
\begin{bmatrix}
\\[-7mm]
e_{p_1} \\[-2mm] e_{p_2} \\[-2mm] e_{p_2}
\end{bmatrix}^\top 
\ub{\begin{bmatrix} 
\hat{n}_1 & 0 & 0 & \hat{n}_2 & \hat{n}_3 & 0 \\[-1.5mm]
0 & \hat{n}_2 & 0 & \hat{n}_1 & 0 & \hat{n}_3 \\[-1.5mm]
0 & 0 & \hat{n}_3 & 0 & \hat{n}_1 & \hat{n}_2
\end{bmatrix}}{P_\p}
\begin{bmatrix}
\\[-7mm]
e_{\q_1} \\[-2mm] e_{\q_2} \\[-2mm] e_{\q_3} \\[-2mm] e_{\q_4} \\[-2mm] e_{\q_5} \\[-2mm] e_{\q_6} 
\end{bmatrix} d\sx 
= \dint_{\p \Omega} \!  y_\p^\top u_\p  \, d\sx 
\end{array}
\label{eq:3D_elast_BC}
\end{equation}

$ $\\[-15mm]
\subsubsection{Tensor representation}

Noticing that $\F = \mbox{Grad}$, and $\F^* = -\mbox{Div}$, we redefine the following variables
$$
\underline{\q}(\X,t) = \mbox{Grad}(\textbf{r}(\X,t)) = \begin{bmatrix}
\q_1(\X,t) &  \frac{1}{2}\q_4(\X,t) & \frac{1}{2}\q_5(\X,t) \\ \frac{1}{2}\q_4(\X,t) & \q_2(\X,t) & \frac{1}{2}\q_6(\X,t) \\ \frac{1}{2}\q_5(\X,t) & \frac{1}{2}\q_6(\X,t) & \q_3(\X,t)
\end{bmatrix} 
, \quad 
e_{\underline{\q}}(\X,t) = \doubleunderline{C_{3D}}(\underline{\q}(\X,t)) = \begin{bmatrix}
e_{\q_1}(\X,t) &  e_{\q_4}(\X,t) & e_{\q_5}(\X,t) \\ e_{\q_4}(\X,t) & e_{\q_2}(\X,t) & e_{\q_6}(\X,t) \\ e_{\q_5}(\X,t) & e_{\q_6}(\X,t) & e_{\q_3}(\X,t)
\end{bmatrix} 
$$
with $\doubleunderline{C_{3D}}(\cdot) = 2\mu(\cdot) + \lambda\mbox{tr}(\cdot)1_3  $ for isotropic materials. With tha above, the model \eqref{eq:3D_elast} written using tensor notation is given by 
\begin{equation}
\begin{bmatrix}
\dot{p} \\[-1mm] \dot{\underline{\q}}
\end{bmatrix} = \begin{bmatrix}
0 & \Div \\[-1mm] \Grad & 0 
\end{bmatrix} \begin{bmatrix}
e_p \\[-1mm] e_{\underline{\q}}
\end{bmatrix}
\end{equation}
with Hamiltonian functional $H[p,\underline{\q}] $ given by 
\begin{equation}
H[p,\underline{\q}] = \mfrac{1}{2} \dint_{\Omega} \!  \mfrac{1}{\rho}\, p \cdot \! p + \doubleunderline{C_{3D}}(\underline{\q}) \!:\! \underline{\q}  \, d\X
\end{equation}
\noindent The boundary variables are obtained from the energy balance which is given by
\begin{equation}
\begin{array}{rl}
\p_t H =  \langle \var_x H,\, \dot{x} \rangle_{in}^{\Omega}  = & \!\!
 \dint_{\Omega} e_{p}\!\cdot \mbox{Div}(e_{\underline{\q}}) + e_{\underline{\q}}\!:\mbox{Grad}(e_{p})  \,  d\X
\end{array}
\end{equation}
so applying the integration by parts theorem for symmetric tensors (see \cite[Theorem 8]{brugnoli2020port}),  we obtain
\begin{equation}
\begin{array}{r}
\p_t H =  \dint_{\p\Omega}  (e_{\underline{\q}}\, \hat{n})\!\cdot e_{p} \,  d\mathtt{s} = \dint_{\p\Omega} y_\p^\top u_\p \,d\mathtt{s}
\end{array}
\end{equation}
where $ \hat{n} = [\hat{n}_1 \;\; \hat{n}_2 \;\; \hat{n}_3]^\top $. The above expression is completely analogous to the expression in \eqref{eq:3D_elast_BC}.


\subsection{Mindlin's plate}

Consider a three-dimensional body that can be treated as a two-dimensional structure as shown in Figure \ref{fig:papa_2D} where $h \in \R{}$ is the thickness (constant). This model is based on the first order shear deformation theory, that is, plane sections normal to the neutral line before deformation remain plane but not necessarily normal to the neutral line after deformation.Then, the displacement field is given by
\begin{equation}
\textbf{u}(\CX,t) = \ub{\begin{bmatrix}
-\x{3} & 0 & 0  \\[-2mm] 0 & -\x{3} & 0 \\[-2mm] 0 & 0 & 1 
\end{bmatrix}}{\lambda_1(\X^c)} \ub{\begin{bmatrix}
\psi_1(\X,t) \\[-2mm] \psi_2(\X,t) \\[-2mm] w(\X,t)
\end{bmatrix}}{\textbf{r}(\X,t)} 
\end{equation}
with $\X = \lbrace \x{1}, \x{2} \rbrace $, $\X^c = \lbrace \x{3} \rbrace$, $\Omega \subset \R{2}$, $\Omega^c = (-\frac{h}{2}\times \frac{h}{2}) \subset \R{}$, $\psi_1(\X,t) \in \R{}$ and $\psi_2(\X,t) \in \R{}$ are the angles rotated by the cross sections, and $w(\X,t) \in \R{}$ is the vertical displacement in the direction of the axis $\x{3}$. From \eqref{eq:def_Mass_matrix_distributed} and \eqref{eq:def_p_distributed} we have
\begin{equation}
\begin{array}{r}
\hspace*{-10mm}\mathcal{M}(\X) =  \rho(\X)\dint_{-h/2}^{h/2}\!\! \begin{bmatrix}
\x{3}^2  & 0 & 0  \\[-1mm]
0 & \x{3}^2 & 0  \\[-1mm]
0 & 0 &  1 
\end{bmatrix} d\x{3} =  \rho(\X)\begin{bmatrix}
\,\bar{I}_2  & 0 &  0 \\[-1mm]
0 & \bar{I}_2 & 0 \\[-1mm]
0 &  0 & \bar{I}_0
\end{bmatrix}
, \quad
p(\X,t) = \mathcal{M}(\X)\dot{\textbf{r}}(\X,t) = 
\begin{bmatrix} 
\rho(\X)\bar{I}_2 \dot{\psi}_1(\X,t) \\[-1mm]
\rho(\X)\bar{I}_2 \dot{\psi}_2(\X,t) \\[-1mm]
\rho(\X) \bar{I}_0 \dot{w}(\X,t) 
\end{bmatrix} 
\end{array}
\end{equation}
where $\bar{I}_i \in \R{}$ with $i=0,2\dots$ is defined as 
$
\bar{I}_i = \int_{-h/2}^{h/2} \x{3}^i \, d\x{3} = \mfrac{h^{i+1}}{2^i(i+1)}
$. %
From \eqref{eq:def_vec_epsilon}, the non-zero components of the strain tensor $\varepsilon \subset \vec{\varepsilon}$ are given by
\begin{equation}
\hspace{-2mm}\varepsilon(\CX,t) \!= \! \begin{bmatrix}
\varepsilon_b \\ \hline \\[-6mm] \varepsilon_s
\end{bmatrix} \! = \! \begin{bmatrix}
\varepsilon_1 \\[-1mm] \varepsilon_2 \\[-1mm] \varepsilon_4 \\ \hline\\[-6mm] \varepsilon_5 \\[-1mm] \varepsilon_6
\end{bmatrix} \! = \! \begin{bmatrix}
-\x{3} [ \p_1 \psi_1(\X,t)] \\[-1mm]
-\x{3} [ \p_2 \psi_2(\X,t)] \\[-1mm] 
-\x{3} [\p_2 \psi_1(\X,t) + \p_1 \psi_2(\X,t)] \\[0mm] \hline \\[-5.5mm]
[\p_1 w(\X,t) - \psi_1(\X,t)] \\[-1mm]
[\p_2 w(\X,t) - \psi_2(\X,t)]
\end{bmatrix}
\label{eq:eps_highlight_Mindlin}
\end{equation}
\noindent From Proposition \ref{prop:2} we choose $m=5$ since there are five functions that are independent of $\X^c$ in the strain vector $\varepsilon(\CX,t)$ (highlighted in square brackets in \eqref{eq:eps_highlight_Mindlin}). Then we have
\begin{equation}
\hspace{-3mm}\varepsilon(\CX,t) = \ub{\begin{bmatrix}
-\x{3} & 0  & 0 & 0 & 0  \\[-2mm]
 0 & -\x{3} & 0 & 0 & 0  \\[-2mm]
 0 & 0 & -\x{3} & 0 & 0 \\[-2mm]
 0 & 0 & 0 & 1  & 0 \\[-2mm]
 0 & 0 & 0 & 0  & 1
\end{bmatrix}
}{\lambda_2(\X^c)}
\ub{ 
\left[ \!\! \begin{array}{c c | c }
\p_1 & 0 & 0 \\[-2mm] 0 & \p_2 & 0  \\[-2mm] \p_2 & \p_1 & 0  \\ \hline \\[-6mm] -1 & 0 & \p_1  \\[-2mm] 0 & -1 & \p_2
\end{array} \!\! \right]
}{\F}
\ub{
\begin{bmatrix}
\psi_1(\X,t) \\[-2mm]
\psi_2(\X,t) \\[-2mm]
w(\X,t) 
\end{bmatrix}
}{\textbf{r}(\X,t)}
\quad \to \quad
\q(\X,t) = \mathcal{F}\,\textbf{r}(\X,t) = \begin{bmatrix}
\p_1 \psi_1(\X,t) \\[-1mm] \p_2 \psi_2(\X,t) \\[-1mm] \p_2 \psi_1(\X,t) + \p_1 \psi_2(\X,t) \\[-1mm] \p_1 w(\X,t)-\psi_1(\X,t) \\[-1mm] \p_2 w(\X,t)-\psi_2(\X,t) 
\end{bmatrix}
\end{equation}
with $\mathcal{F}$ a differential operator of order $N=1$. The stress $\sigma(\CX,t)$ is obtained from Hooke's law $\sigma(\CX,t) = C \, \varepsilon(\CX,t)$, where the constitutive matrix $C = C^\top > 0 $ is given by
\begin{equation}
C = \begin{bmatrix}
C_b & 0 \\[-1.5mm] 0 & C_s 
\end{bmatrix}  \quad , \quad \mbox{with } \quad C_b = \mfrac{E}{(1-\nu^2)} \begin{bmatrix}
1 & \nu & 0 \\[-1.5mm] \nu & 1 & 0 \\[-1.5mm] 0 & 0 & \frac{(1-\nu)}{2}
\end{bmatrix}  \quad , \quad   C_s = \begin{bmatrix}
G & 0 \\[-1.5mm] 0 & G
\end{bmatrix}
\label{eq:C_mindlin}
\end{equation}
where $C_b$ is the constitutive matrix for plane stress, and $E,G,\nu$ are material properties. Then, 
from \eqref{eq:def_Stiffness_matrix_distributed} we have 
$
\begin{array}{r}
\mathcal{K}(\X) =  \dint_{-h/2}^{h/2}  \begin{bmatrix}
\x{3}^2 \, C_b & 0 \\[-1mm]
0 & 1 \,C_s  
\end{bmatrix}  d\x{3} =    \begin{bmatrix}
\bar{I}_2 \,C_b & 0  \\[-1mm]
0 & \bar{I}_0\,C_s
\end{bmatrix}
\end{array}
$. Considering that there are no distributed inputs, from Theorem \ref{theo:PHS_SD} we have
\begin{equation}
 \begin{bmatrix}
\dot{p}_1 \\[-2mm] \dot{p}_2 \\[-2mm] \dot{p}_3 \\[-2mm] \dot{\q}_1 \\[-2mm] \dot{\q}_2 \\[-2mm] \dot{\q}_3 \\[-2mm] \dot{\q}_4 \\[-2mm] \dot{\q}_5 
\end{bmatrix}  = 
\left[ \!\! \begin{array}{c c c | c c c c c }
0 & 0 & 0 &  \p_1 & 0 & \p_2 & 1 & 0 \\[-2mm]
0 & 0 & 0 &  0 & \p_2 & \p_1  & 0  & 1  \\[-2mm]
0 & 0 & 0 &  0 & 0 & 0 & \p_1 & \p_2 \\[-0.5mm]
\hline \\[-6.5mm]
\p_1 & 0 & 0 & 0 & 0 & 0 & 0 & 0 \\[-2mm]
0 & \p_2 & 0 & 0 & 0 & 0 & 0 & 0 \\[-2mm]
\p_2 & \p_1  & 0 & 0 & 0 & 0 & 0 & 0 \\[-2mm]
-1 & 0 & \p_1 & 0 & 0 & 0 & 0 & 0 \\[-2mm]
0 & -1 & \p_2 & 0 & 0 & 0 & 0 & 0 
\end{array} \!\! \right]
\begin{bmatrix}
e_{p_1} \\[-2mm] e_{p_2} \\[-2mm] e_{p_3} \\[-2mm] e_{\q_1} \\[-2mm] e_{\q_2} \\[-2mm] e_{\q_3} \\[-2mm] e_{\q_4} \\[-2mm] e_{\q_5}
\end{bmatrix}
\label{eq:Mindlin_plate_model_Voigt}
\end{equation}
with Hamiltonian defined as in \eqref{eq:def_HAMILTONIANO}. The boundary ports are defined from Remark \ref{rem:boundary_ports_N1} by
\begin{equation}
\begin{array}{rl}
\p_t{H}  =  
\dint_{\p\Omega} 
\begin{bmatrix}
\\[-7mm]
e_{p_1} \\[-2mm] e_{p_2}  \\[-2mm] e_{p_3}
\end{bmatrix}^\top 
\ub{\begin{bmatrix} 
\hat{n}_1 & 0 & \hat{n}_2 & 0 & 0\\[-2mm]
0 & \hat{n}_2 & \hat{n}_1 & 0 & 0\\[-2mm]
0 & 0 & 0 & \hat{n}_1 & \hat{n}_2
\end{bmatrix}}{P_\p}
\begin{bmatrix}
\\[-7mm]
e_{\q_1} \\[-2mm] e_{\q_2} \\[-2mm] e_{\q_3} \\[-2mm] e_{\q_4} \\[-2mm] e_{\q_5} 
\end{bmatrix} d\sx 
= \dint_{\p \Omega} \!  y_\p^\top u_\p  \, d\sx 
\end{array}
\label{eq:Mindlin_plate_BC}
\end{equation}

$ $\\[-15mm]
\subsubsection{Tensor representation}

\noindent Notice that  $ \mathcal{F} = \begin{bsmallmatrix} \mbox{Grad} & 0  \\[0mm]
-1_2 & \mbox{grad} \end{bsmallmatrix} $, $ \mathcal{F}^* = -\begin{bsmallmatrix}
\mbox{Div} & 1_2  \\[0mm] 0 & \mbox{div} \end{bsmallmatrix}$, $ \textbf{r}(\X,t) = \begin{bsmallmatrix}
\psi(\X,t)  \\[0mm] w(\X,t) \end{bsmallmatrix} $, 
where $\psi(\X,t) = [\psi_1(\X,t) \;\; \psi_2(\X,t)]^\top \in \R{2}$ groups the angles rotated by the cross section with respect to each coordinate axis. The mass matrix $\mathcal{M}(\X)$ and stiffness matrix  $ {\mathcal{K}}(\X)$ can be rewritten as
\begin{equation}
\mathcal{M}(\X) = \rho(\X)\begin{bmatrix}
\bar{I}_2 1_2 & 0  \\[-1mm] 0 & \bar{I}_0 
\end{bmatrix}
\quad , \quad
\doubleunderline{\mathcal{K}}(\X) = \begin{bmatrix}
\bar{I}_2 \doubleunderline{C_{b}}(\cdot) & 0  \\[-1mm] 0 & \bar{I}_0 C_s 
\end{bmatrix}
\end{equation}
with $C_s$ the constitutive matrix for shear stress as defined in \eqref{eq:C_mindlin}, and $\doubleunderline{C_b}(\cdot)= \frac{E}{1-\nu^2}\left[(1-\nu)(\cdot) + \nu \mbox{tr}(\cdot)1_2 \right]$ the constitutive tensor for plane stress. With the above we redefine the energy variables as
\begin{equation}
\hspace{-0mm} p(\X,t) \! = \! \begin{bmatrix}
p_{\psi}(\X,t) \\[-1mm] p_{w}(\X,t) 
\end{bmatrix} \! = \! \begin{bmatrix}
\rho(\X) \bar{I}_2 \dot{\psi}(\X,t) \\[-1mm] \rho(\X)\bar{I}_0 \dot{w}(\X,t) 
\end{bmatrix}  , \;\;
\underline{\q}(\X,t) \! = \! \begin{bmatrix}
\q_{\psi}(\X,t) \\[-1mm] \q_{w}(\X,t) 
\end{bmatrix} \! =  \! \begin{bmatrix}
\mbox{Grad}(\psi(\X,t)) \\[-1mm] \mbox{grad}(w(\X,t))-\psi(\X,t)
\end{bmatrix}
\end{equation}
and co-energy variables as  
\begin{equation}
{e_p}(\X,t) = \begin{bmatrix}
e_{p_\psi}(\X,t) \\[-1mm] e_{p_w}(\X,t)  
\end{bmatrix} = \begin{bmatrix}
\dot{\psi}(\X,t) \\[-1mm] \dot{w}(\X,t) 
\end{bmatrix}
\quad , \quad
\underline{e_\q}(\X,t) = 
\begin{bmatrix}
e_{\q_\psi}(\X,t) \\[-1mm] e_{\q_w}(\X,t) 
\end{bmatrix} =
\begin{bmatrix}
\bar{I}_2 \doubleunderline{C_{b}}(\q_\psi(\X,t))  \\[-1mm]
\bar{I}_0 C_s \, \q_w(\X,t) 
\end{bmatrix}
\end{equation}
$ $\\[-5mm]
where 
$$
\begin{array}{rrrr}
\q_{\psi}(\X,t) = \begin{bmatrix}
\q_1(\X,t) & \frac{1}{2}\q_3(\X,t) \\[0mm] \frac{1}{2}\q_3(\X,t) & \q_2(\X,t) 
\end{bmatrix} , &
\q_{w}(\X,t) = \begin{bmatrix}
\q_4(\X,t)  \\  \q_5(\X,t) 
\end{bmatrix} , &
e_{\q_{\psi}}(\X,t) = \begin{bmatrix}
e_{\q_1}(\X,t) & e_{\q_3}(\X,t) \\[0mm] e_{\q_3}(\X,t) & e_{\q_2}(\X,t) 
\end{bmatrix} , &
e_{\q_{w}}(\X,t) = \begin{bmatrix}
e_{\q_4}(\X,t)  \\  e_{\q_5}(\X,t) 
\end{bmatrix}
\end{array}
$$
Then, the model in \eqref{eq:Mindlin_plate_model_Voigt} written using tensor notation is given by
\begin{equation}
\begin{array}{r}
\ub{\begin{bmatrix}
\dot{p}_{\psi} \\[-1mm] \dot{p}_{w} \\[0.5mm]
\hline \\[-6mm]
\dot{\q}_{\psi} \\[-1mm] \dot{\q}_{w}
\end{bmatrix}}{\dot{x}} =   \ub{\left[ \!\begin{array}{c c | c c}
\!0 & \!\!\!0  & \mbox{Div} \!\! & 1_2 \!\!\!\!  \\[-1mm]
\!0 & \!\!\!0  & 0 & \mbox{div} \!\!  \\[-0.5mm]
\hline \\[-6mm]
\! \mbox{Grad}  & \!\!\! 0 &  0 & 0  \\[-1mm]
\!\!\!\! -1_2  &  \!\!\!\! \mbox{grad} \! & 0 & 0 
\end{array} \! \right]}{\mathcal{J}=-\mathcal{J}^*} \ub{\begin{bmatrix}
{e}_{p_\psi} \\[-1mm] {e}_{p_w} \\[0.5mm]
\hline \\[-6mm]
e_{\q_{\psi}} \\[-1mm] e_{\q_{w}} 
\end{bmatrix}}{\var_x H} 
\end{array}
\end{equation}
with Hamiltonian functional $H[p,\underline{\q}]$ given by 
\begin{equation}
H[p,\underline{\q}] = \mfrac{1}{2} \dint_{\Omega} \!  \mfrac{1}{\rho \bar{I}_2} \, p_\psi \cdot \! p_\psi + \mfrac{p_w ^2}{\rho \bar{I}_0}  + \bar{I}_2 \doubleunderline{C_b}(\q_{\psi}) \!:\!\q_{\psi} + \bar{I}_0 C_s\q_w \!\cdot \! \q_w \, d\X
\end{equation}
\noindent The boundary variables are obtained from the energy balance which is given by
\begin{equation}
\begin{array}{rl}
\p_t H =  \langle \var_x H,\, \dot{x} \rangle_{in}^{\Omega}  = & 
 \dint_{\Omega}  [e_{p_\psi}\!\cdot \mbox{Div}(e_{\q_\psi}) + e_{\q_\psi}\!:\mbox{Grad}(e_{p_\psi})] + [e_{p_w}\, \mbox{div}(e_{\q_w}) + e_{\q_w}\!\cdot \mbox{grad}(e_{p_w})] \,  d\X
\end{array}
\end{equation}
so applying the integration by parts theorem for symmetric tensors to the first term of the integral above (see \cite[Theorem 8]{brugnoli2020port}), and the divergence theorem to the second term, then we obtain
\begin{equation}
\begin{array}{r}
\p_t H =  \dint_{\p\Omega} \left[ (e_{\q_\psi}\, \hat{n})\!\cdot e_{p_\psi}  + e_{p_w}\,(e_{\q_w}\!\cdot \hat{n}) \right] d\mathtt{s} = \dint_{\p\Omega} y_\p^\top u_\p \,d\mathtt{s}
\end{array}
\end{equation}
where $ \hat{n} = [\hat{n}_1 \;\; \hat{n}_2]^\top $. The above expression is completely  analogous to  \eqref{eq:Mindlin_plate_BC}. For more details see \cite{macchelli2005port} and \cite{brugnoli2019port}.

$ $\\[-11.5mm]
\subsection{Vibrating string}

\noindent Consider a three-dimensional body that can be treated as a one-dimensional structure as shown in Figure \ref{fig:papa_1D}. \noindent The displacement field is given by
\begin{equation}
\textbf{u}(\CX,t) = \ub{\begin{bmatrix}
0   \\[-2.5mm] 0   \\[-2.5mm] 1 
\end{bmatrix}}{\lambda_1(\X^c)} \ub{ w(\X,t) }{\textbf{r}(\X,t)} 
\end{equation}
with $\X = \lbrace \x{1} \rbrace $, $\X^c = \lbrace \x{2},\x{3} \rbrace$, $\Omega = (a,b) \subset \R{}$, $\Omega^c = A \subset \R{2}$, and $w(\X,t) \in \R{}$ is the vertical displacement in the direction of the $\x{3}$ axis. From \eqref{eq:def_Mass_matrix_distributed} and \eqref{eq:def_p_distributed} we have
\begin{equation}
\mathcal{M}(\X) = \rho(\X) \dint_{A}\!\! 1 \, dA = \rho(\X)A(\X)
\quad ,\quad 
p(\X,t) = \mathcal{M}(\X)\dot{\textbf{r}}(\X,t) = \rho(\X) A(\X)\dot{w}(\X,t)
\end{equation}
where $A(\X)$ is the cross section area. From \eqref{eq:def_vec_epsilon}, the non-zero components of the strain tensor $\varepsilon \subset \vec{\varepsilon}$ are given by
\begin{equation}
\varepsilon(\CX,t) = 
\varepsilon_5(\X,t)  = 
\p_1 w(\X,t) 
\end{equation}

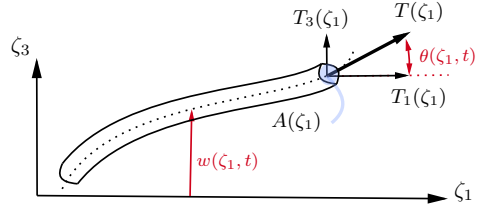
\begin{figure}[t]
\begin{center}
	\tikzset{every picture/.style={line width=0.75pt}} 

\scalebox{0.85}{
\begin{tikzpicture}[x=0.75pt,y=0.75pt,yscale=-1,xscale=1]

\draw    (96,281.25) -- (95.76,201.5) ;
\draw [shift={(95.75,199.5)}, rotate = 89.82] [fill={rgb, 255:red, 0; green, 0; blue, 0 }  ][line width=0.08]  [draw opacity=0] (12,-3) -- (0,0) -- (12,3) -- cycle    ;
\draw    (96,281.25) -- (337.75,283.08) ;
\draw [shift={(339.75,283.1)}, rotate = 180.43] [fill={rgb, 255:red, 0; green, 0; blue, 0 }  ][line width=0.08]  [draw opacity=0] (12,-3) -- (0,0) -- (12,3) -- cycle    ;
\draw    (111,262.25) .. controls (147.88,215.25) and (240.75,218.93) .. (264.75,202.93) ;
\draw [color={rgb, 255:red, 208; green, 2; blue, 27 }  ,draw opacity=1 ]   (187.71,231.93) -- (186.75,281.5) ;
\draw [shift={(187.75,229.93)}, rotate = 91.11] [fill={rgb, 255:red, 208; green, 2; blue, 27 }  ,fill opacity=1 ][line width=0.08]  [draw opacity=0] (7.2,-1.8) -- (0,0) -- (7.2,1.8) -- cycle    ;
\draw  [dash pattern={on 0.84pt off 2.51pt}]  (110.88,277.25) .. controls (142.75,213.93) and (276.75,229.93) .. (283.75,194.93) ;
\draw    (111,262.25) .. controls (107.88,269.25) and (108.88,273.25) .. (120,274.25) ;
\draw    (264.75,202.93) .. controls (261.63,209.93) and (262.63,213.93) .. (273.75,214.93) ;
\draw    (264.75,202.93) .. controls (274.63,203.93) and (276.63,205.93) .. (273.75,214.93) ;
\draw [color={rgb, 255:red, 0; green, 0; blue, 0 }  ,draw opacity=1 ][line width=1.5]    (314.22,185.81) -- (267.95,210.35) ;
\draw [shift={(317.75,183.93)}, rotate = 152.06] [fill={rgb, 255:red, 0; green, 0; blue, 0 }  ,fill opacity=1 ][line width=0.08]  [draw opacity=0] (10.92,-2.73) -- (0,0) -- (10.92,2.73) -- cycle    ;
\draw [color={rgb, 255:red, 208; green, 2; blue, 27 }  ,draw opacity=1 ] [dash pattern={on 0.84pt off 2.51pt}]  (267.95,210.35) -- (340.17,209.68) ;
\draw [color={rgb, 255:red, 208; green, 2; blue, 27 }  ,draw opacity=1 ]   (315.25,190.04) .. controls (315.77,191.96) and (316.22,193.59) .. (316.57,197.77) .. controls (317,202.76) and (316.9,201.56) .. (316.77,207.95) ;
\draw [shift={(316.74,209.93)}, rotate = 270.96] [fill={rgb, 255:red, 208; green, 2; blue, 27 }  ,fill opacity=1 ][line width=0.08]  [draw opacity=0] (7.2,-1.8) -- (0,0) -- (7.2,1.8) -- cycle    ;
\draw [shift={(314.75,188.1)}, rotate = 76.54] [fill={rgb, 255:red, 208; green, 2; blue, 27 }  ,fill opacity=1 ][line width=0.08]  [draw opacity=0] (7.2,-1.8) -- (0,0) -- (7.2,1.8) -- cycle    ;
\draw    (120,274.25) .. controls (156.88,227.25) and (249.75,230.93) .. (273.75,214.93) ;
\draw [color={rgb, 255:red, 0; green, 0; blue, 0 }  ,draw opacity=1 ][line width=0.75]    (314.75,209.95) -- (267.95,210.35) ;
\draw [shift={(316.75,209.93)}, rotate = 179.51] [fill={rgb, 255:red, 0; green, 0; blue, 0 }  ,fill opacity=1 ][line width=0.08]  [draw opacity=0] (8.4,-2.1) -- (0,0) -- (8.4,2.1) -- cycle    ;
\draw [color={rgb, 255:red, 0; green, 0; blue, 0 }  ,draw opacity=1 ][line width=0.75]    (267.77,186.93) -- (267.95,210.35) ;
\draw [shift={(267.75,184.93)}, rotate = 89.55] [fill={rgb, 255:red, 0; green, 0; blue, 0 }  ,fill opacity=1 ][line width=0.08]  [draw opacity=0] (8.4,-2.1) -- (0,0) -- (8.4,2.1) -- cycle    ;
\draw  [draw opacity=0][fill={rgb, 255:red, 119; green, 155; blue, 255 }  ,fill opacity=0.34 ] (263.6,209.35) .. controls (263.6,205.84) and (266,203) .. (268.95,203) .. controls (271.9,203) and (274.3,205.84) .. (274.3,209.35) .. controls (274.3,212.86) and (271.9,215.7) .. (268.95,215.7) .. controls (266,215.7) and (263.6,212.86) .. (263.6,209.35) -- cycle ;
\draw [color={rgb, 255:red, 119; green, 155; blue, 255 }  ,draw opacity=0.34 ][line width=1.5]    (270.75,238.1) .. controls (281.75,229.1) and (277.75,220.1) .. (267.95,210.35) ;

\draw (342,271.25) node [anchor=north west][inner sep=0.75pt]   [align=left] {$\displaystyle \zeta _{1}$};
\draw (79,185.25) node [anchor=north west][inner sep=0.75pt]   [align=left] {$\displaystyle \zeta _{3}$};
\draw (189.13,254.47) node [anchor=north west][inner sep=0.75pt]  [font=\scriptsize,color={rgb, 255:red, 208; green, 2; blue, 27 }  ,opacity=1 ,rotate=-359.38]  {$w( \zeta _{1} ,t)$};
\draw (233.5,229.8) node [anchor=north west][inner sep=0.75pt]  [font=\footnotesize]  {$A( \zeta _{1})$};
\draw (321.27,192.17) node [anchor=north west][inner sep=0.75pt]  [font=\scriptsize,color={rgb, 255:red, 208; green, 2; blue, 27 }  ,opacity=1 ,rotate=-359.38]  {$\theta ( \zeta _{1} ,t)$};
\draw (306.5,164.8) node [anchor=north west][inner sep=0.75pt]  [font=\footnotesize]  {$T( \zeta _{1})$};
\draw (304.5,215.8) node [anchor=north west][inner sep=0.75pt]  [font=\footnotesize]  {$T_{1}( \zeta _{1})$};
\draw (246.5,168.8) node [anchor=north west][inner sep=0.75pt]  [font=\footnotesize]  {$T_{3}( \zeta _{1})$};

\end{tikzpicture}
}
\end{center}
\vspace{-10mm}\caption{Vibrating string scheme.}
\label{fig:1D_string}
\end{figure}

From Proposition \ref{prop:2} we choose $m=1$, then 
\begin{equation}
\varepsilon(\CX,t) = \ub{ 1 }{\lambda_2(\X^c)} \ub{ \p_1 }{\mathcal{F}} \ub{ w(\X,t) }{\textbf{r}(\X,t)} \quad \to \quad \q(\X,t) = \mathcal{F}\,\textbf{r}(\X,t) = \varepsilon_5(\X,t)
\end{equation}
with $\mathcal{F}$ a differential operator of order $N=1$. The stress $\sigma(\CX,t)$ is obtained from $\sigma(\CX,t) = C \, \varepsilon(\CX,t) \in \R{}$, where $C$ is chosen as $C = T(\X)/A(\X)$, with $T(\X)$ the internal tension in the string. To understand this choice of $C$, first notice that $\sigma(\CX,t) = \sigma_5(\X,t) = \sigma_{13}(\X,t)$, which represents the shear stress on the face perpendicular to the $\x{1}$ axis and in the direction of the axis $ \x{3}$. The way to calculate this stress is by $\sigma_{13}(\X,t) = T_3(\X)/A(\X)$, with $T_3(\X)$ the projection of $T(\X)$ on the $\x{3}$ axis. From the geometry we know that $T_3(\X) = T(\X) sin(\theta(\X,t))$, where $\theta(\X,t)$ is the angle shown in Figure \ref{fig:1D_string}. Using the assumption of small displacements and rotations, $\theta(\X,t) \approx \p_1 w(\X,t)$ and  $T_3(\X) \approx T(\X) \theta(\X,t) = T(\X) \p_1 w(\X,t)$. With all the above we have
$$
\sigma(\X,t) = \sigma_{13}(\X,t) = T_3(\X)/A(\X) = \mfrac{T(\X)}{A(\X)} \p_1 w(\X,t) = \mfrac{T(\X)}{A(\X)} \varepsilon_{13}(\X,t) = C \, \varepsilon(\X,t)
$$ 
then by association the constitutive relation is equivalent to $C = T(\X)/A(\X)$. Then, from \eqref{eq:def_Stiffness_matrix_distributed} we have 
$
\mathcal{K}(\X) =  \frac{T(\X)}{A(\X)} \int_{A}  1  \,d\x{2}\x{3} =  T(\X)  
$.
Considering that there are no distributed inputs, from Theorem \ref{theo:PHS_SD} we have
\begin{equation}
\begin{bmatrix}
\dot{p} \\[-1mm] \dot{\q} 
\end{bmatrix} = \begin{bmatrix}
0 & \p_1 \\[-1mm] \p_1 & 0
\end{bmatrix} \begin{bmatrix}
e_{p} \\[-1mm] e_{\q} 
\end{bmatrix} 
\end{equation}
with Hamiltonian defined as 
\begin{equation}
H[p,\q] = \mfrac{1}{2}\dint_{a}^{b} \mfrac{p^2}{\rho A} +  T \q^2 \, d\X
\end{equation}

Since $\Omega$ is the segment $(a,b)$, the boundary points are $\sx = \lbrace a, b \rbrace$, then $\hat{n}_{1}(\sx = a)=-1$ and $\hat{n}_{1}(\sx = b)=1$. The boundary ports are defined from Remark \ref{rem:boundary_ports_N1} as
\begin{equation}
\p_t{H} =  \dint_{\p\Omega} \!\! e_p\, e_\q \,\hat{n}_{1}\, d\mathtt{s}
=  -e_p(a) e_\q(a)+ e_p(b)e_\q(b) = \dint_{\p \Omega} \! y_\p^\top u_\p \,d\sx
\end{equation}


\subsection{Torsion in circular bars}

\noindent Consider a three-dimensional body that can be treated as a one-dimensional structure as shown in Figure \ref{fig:papa_1D}. \noindent The displacement field is given by
\begin{equation}
\textbf{u}(\CX,t) = \ub{\begin{bmatrix}
0   \\[-2mm] \;\;\x{3}   \\[-2mm] -\x{2} 
\end{bmatrix}}{\lambda_1(\X^c)} \ub{ \theta(\X,t) }{\textbf{r}(\X,t)} 
\end{equation}
with $\X = \lbrace \x{1} \rbrace $, $\X^c = \lbrace \x{2},\x{3} \rbrace$, $\Omega = (a,b) \subset \R{}$, $\Omega^c = A \subset \R{2}$, and $\theta(\X,t) \in \R{}$ is the angle rotated by the cross section in the direction of the $\x{1}$ axis. From \eqref{eq:def_Mass_matrix_distributed} and \eqref{eq:def_p_distributed} we have
\begin{equation}
\mathcal{M}(\X) = \rho(\X) \dint_{A}\!\! (\x{2}^2 + \x{3}^2) \, dA = \rho(\X)I_p(\X)
\quad ,\quad 
p(\X,t) = \mathcal{M}(\X)\dot{\textbf{r}}(\X,t) = \rho(\X) I_p(\X)\dot{\theta}(\X,t)
\end{equation}
where $I_p(\X)$ is the polar moment of inertia of the cross section. From \eqref{eq:def_vec_epsilon}, the non-zero components of the strain tensor $\varepsilon \subset \vec{\varepsilon}$ are given by
\begin{equation}
{\varepsilon}(\CX,t)= \begin{bmatrix}
\\[-7mm]
{\varepsilon}_4(\CX,t) \\[-1mm] {\varepsilon}_5(\CX,t) \end{bmatrix} = \begin{bmatrix}
\\[-7mm]
\;\;\x{3}\,\p_1\theta(\X,t) \\[-1mm] -\x{2}\,\p_1 \theta(\X,t)
\end{bmatrix}
\end{equation}
From Proposition \ref{prop:2} we choose $m=2$, then 
\begin{equation}
\varepsilon(\CX,t) = \ub{ \begin{bmatrix}
\\[-7mm]
\x{3} & 0 \\[-1mm] 0 & -\x{2}
\end{bmatrix} }{\lambda_2(\X^c)} \ub{ \begin{bmatrix}
\p_1 \\[-1mm] \p_1
\end{bmatrix} }{\mathcal{F}} \ub{ \theta(\X,t) }{\textbf{r}(\X,t)} \quad \to \quad \q(\X_1,t) = \mathcal{F}\,\textbf{r}(\X,t) = \begin{bmatrix}
\p_1 \theta(\X,t) \\[-1mm] \p_1 \theta(\X,t)
\end{bmatrix}
\end{equation}
with $\mathcal{F}$ a differential operator of order $N=1$. The stress $\sigma(\CX,t)$ is obtained from $\sigma(\CX,t) = C \, \varepsilon(\CX,t) \in \R{2}$, where the constitutive matrix is given by $C = \begin{bsmallmatrix} G & 0 \\ 0 & G \end{bsmallmatrix}$ with $G$ the shear modulus. Then, from \eqref{eq:def_Stiffness_matrix_distributed} we have 
$
\mathcal{K}(\X) = G \dint_{A}  \begin{bmatrix}
\x{3}^2 & 0 \\[-2mm] 0 & \x{2}^2
\end{bmatrix}  \,dA =  \begin{bmatrix}
G I_{t_3}(\X) & 0 \\[-2mm] 0 & G I_{t_2}(\X) 
\end{bmatrix}    
$,
with $I_{t_2}(\X)$ and $I_{t_2}(\X)$ the transverse moments of inertia of the cross section. Notice that for a circular cross section we have $I_{t_2} = I_{t_3} \equiv I_t$, and $I_p = 2 I_t$. Considering that there are no distributed inputs, from Theorem \ref{theo:PHS_SD} we have
Considering that there are no distributed inputs, from \eqref{eq:def_dPHS_prop3} we have
\begin{equation}
\begin{bmatrix}
\\[-6mm]
\dot{p} \\[-1.5mm] \dot{\q}_1 \\[-1.5mm] \dot{\q}_2
\end{bmatrix} = \left[ \begin{array}{c | c c}
\\[-6mm]
0 & \p_1 &  \p_1 \\[-0.5mm]
\hline \\[-6.5mm]
\p_1 & 0 & 0  \\[-1.5mm]
\p_1 & 0 & 0 
\end{array} \right] \begin{bmatrix}
\\[-7mm]
e_{p_1} \\[-1.5mm] e_{\q_1} \\[-1.5mm] e_{\q_2}
\end{bmatrix} 
\label{eq:torsion_2var}
\end{equation}
with Hamiltonian defined as in \eqref{eq:def_HAMILTONIANO} given by
\begin{equation}
H[p,\q] = \mfrac{1}{2}\dint_{a}^{b} \mfrac{p^2}{\rho A} +  G I_t \q_1^2 +  G I_t \q_2^2 \, d\X
\end{equation}
The boundary ports are defined from Remark \ref{rem:boundary_ports_N1} by
\begin{equation}
\begin{array}{rl}
\p_t{H}  =  
\dint_{\p\Omega} 
e_{p}
\ub{\begin{bmatrix} 
\hat{n}_1 & \hat{n}_1
\end{bmatrix}}{P_\p}
\begin{bmatrix}
\\[-7mm]
e_{\q_1} \\[-2mm] e_{\q_2}
\end{bmatrix} d\sx 
= -2e_p(a)e_{\q_1}(a) + 2e_p(b)e_{\q_1}(b)
= \dint_{\p \Omega} \!  y_\p^\top u_\p  \, d\sx 
\end{array}
\label{eq:torsion_2var_BC}
\end{equation}
$ $\\[-4mm]
since $e_{\q_1}(\X,t) = e_{\q_2}(\X,t)$. Moreover, due to $\q_1(\X,t)=\q_2(\X,t) \equiv \bar{\q}(\X,t) $ and $I_p = 2I_t$, we have $e_{\q_1}(\X,t)+e_{\q_2}(\X,t) \equiv  e_{\bar{\q}}(\X,t) = G I_p \bar{\q}(\X,t)$. Then, the model \eqref{eq:torsion_2var} can be written equivalently as
\begin{equation}
\begin{bmatrix}
\\[-6mm]
\dot{p} \\[-1mm] \dot{\bar{\q}} 
\end{bmatrix} = \begin{bmatrix}
\\[-6mm]
0 & \p_1 \\[-1mm] \p_1 & 0
\end{bmatrix} \begin{bmatrix}
\\[-7mm]
e_{p} \\[-1mm] e_{\bar{\q}} 
\end{bmatrix} 
\end{equation}
$ $\\[-4mm]
with Hamiltonian defined as 
\begin{equation}
H[p,\bar{\q}] = \mfrac{1}{2}\dint_{a}^{b} \mfrac{p^2}{\rho A} +  G I_p \bar{\q}^2 \, d\X
\end{equation}
$ $\\[-2mm]
and boundary ports defined from Remark \ref{rem:boundary_ports_N1} as
\begin{equation}
\p_t{H} =  \dint_{\p\Omega} \!\! e_p\, e_{\bar{\q}} \,\hat{n}_{1}\, d\mathtt{s}
=  -e_p(a) e_{\bar{\q}}(a)+ e_p(b)e_{\bar{\q}}(b) = \dint_{\p \Omega} \! y_\p^\top u_\p \,d\sx
\end{equation}
which is completely analogous to the expression in \eqref{eq:torsion_2var_BC}. 


\subsection{Reddy's beam}

\noindent Consider a three-dimensional body that can be treated as a one-dimensional structure as shown in Figure \ref{fig:papa_1D}. \noindent Based on the Reddy's third order shear deformation theory, the displacement field is given by
\begin{equation}
\textbf{u}(\CX,t) = \ub{\begin{bmatrix}
-(\x{3}-\alpha \x{3}^3) & 0 & -\alpha \x{3}^3  \\[-2mm] 0 & 0 & 0 \\[-2mm] 0& 1 & 0 
\end{bmatrix}}{\lambda_1(\X^c)} \ub{ \begin{bmatrix}
\psi(\X,t) \\[-2mm] w(\X,t) \\[-2mm] \p_1 w(\X,t)  
\end{bmatrix} }{\textbf{r}(\X,t)} 
\end{equation}
with $\X = \lbrace \x{1} \rbrace $, $\X^c = \lbrace \x{2},\x{3} \rbrace$, $\Omega = (a,b) \subset \R{}$, $\Omega^c = A \subset \R{2}$,  $\alpha = \frac{4}{3h^2}$,  $\psi(\X,t) \in \R{}$ is the rotation of normals to neutral axis (see Figure \ref{fig:EB_TIM_assumptions}.b), $w(\X,t) \in \R{}$ is the vertical displacemt in the direction of the $\x{3}$ axis, and $\p_1 w(\X,t) \in \R{}$ is the slope of the neutral axis. From \eqref{eq:def_Mass_matrix_distributed} and \eqref{eq:def_p_distributed} we have
\begin{equation}
\mathcal{M}(\X) = \rho(\X)\dint_{A}\!\! \begin{bmatrix}
(\x{3} ^2 -2\alpha\x{3}^4 + \alpha^2 \x{3}^6) & 0 & \alpha(\x{3}^4 - \alpha \x{3}^6)\\[-1.5mm]
0 &  1 & 0 \\[-1.5mm]
\alpha(\x{3}^4 - \alpha \x{3}^6) & 0 & \alpha^2 \x{3}^6
\end{bmatrix} dA = \rho(\X)\begin{bmatrix}
(I_2-2\alpha I_4 + \alpha^2 I_6) & 0 & \alpha (I_4-\alpha I_6)\\[-1.5mm] 0 & I_0 & 0 \\[-1.5mm]
\alpha (I_4-\alpha I_6) & 0 & \alpha^2 I_6
\end{bmatrix}
\end{equation}
and
\begin{equation}
p(\X,t) = \mathcal{M}(\X)\dot{\textbf{r}}(\X,t) = \begin{bmatrix} \rho(\X)(I_2-2\alpha I_4 + \alpha^2 I_6) \dot{\psi}(\X,t)  +  \rho(\X)\alpha(I_4-\alpha I_6)\p_1 \dot{w}(\X,t)  \\[-1mm] 
\rho(\X) I_0 \dot{w}(\X,t) \\[-1mm]
\rho(\X)\alpha (I_4-\alpha I_6) \dot{\psi}(\X,t) + \rho(\X) \alpha^2 I_6 \p_1 \dot{w}(\X,t)
\end{bmatrix}
\end{equation}
where $I_i \in \R{}$ with $i=0,1,2,\dots$ is defined as
$
I_i(\X) = \int_A \x{3}^i \, dA
$. 
Note that $I_0(\X)$ is the area of the cross section, $I_2(\X)$ is the second moment of inertia of the cross section, and in general $I_i(\X)$ is the $i^{th}$ moment of inertia of the cross section. From \eqref{eq:def_vec_epsilon}, the non-zero components of the strain tensor $\varepsilon \subset \vec{\varepsilon}$ are given by
\begin{equation}
{\varepsilon}(\CX,t)= \begin{bmatrix}
\\[-7mm]
{\varepsilon}_1(\CX,t) \\[-1mm] {\varepsilon}_5(\CX,t) \end{bmatrix} = \begin{bmatrix}
\\[-7mm]
-(\x{3}-\alpha \x{3}^3)\, [\p_1 \psi(\X,t)] - \alpha \x{3}^3 [\p_1^2 w(\X,t)] \\[1mm] (1-3\alpha \x{3}^2)[\p_1 w(\X,t) - \psi(\X,t)]
\end{bmatrix}
\label{eq:eps_highlight_Reddy_beam}
\end{equation}
\noindent From Proposition \ref{prop:2} we choose $m=3$ since there are three functions that are independent of $\X^c$ in the strain vector $\varepsilon(\CX,t)$ (highlighted in square brackets in \eqref{eq:eps_highlight_Reddy_beam}). Then we have
\begin{equation}
\hspace{-3mm}\varepsilon(\CX,t) = \ub{\begin{bmatrix}
 -(\x{3}-\alpha \x{3}^3) & 0  & - \alpha \x{3}^3 \\[-2mm]
 0 & (1-3\alpha \x{3}^2) & 0 
\end{bmatrix}
}{\lambda_2(\X^c)}
\ub{ 
\left[ \!\! \begin{array}{c c c }
\p_1 & 0 & 0 \\[-2mm] -1 & \p_1 & 0  \\[-2mm] 0 & 0 & \p_1 
\end{array} \!\! \right]
}{\F}
\ub{
\begin{bmatrix}
\psi(\X,t) \\[-2mm]
w(\X,t) \\[-2mm]
\p_1 w(\X,t) 
\end{bmatrix}
}{\textbf{r}(\X,t)}
\quad \to \quad
\q(\X,t) = \mathcal{F}\,\textbf{r}(\X,t) = \begin{bmatrix}
\p_1 \psi(\X,t) \\[-1.5mm] \p_1 w(\X,t) - \psi(\X,t) \\[-1.5mm] \p_1^2 w(\X,t)
\end{bmatrix}
\end{equation}
with $\mathcal{F}$ a differential operator of order $N=1$. The stress $\sigma(\CX,t)$ is obtained from Hooke's law $\sigma(\CX,t) = C \, \varepsilon(\CX,t)$, where $C = \begin{bsmallmatrix} E & 0 \\ 0 & G \end{bsmallmatrix} $, where $E,G$ are Young's modulus and shear modulus, respectively. Note that the correction factor $\kappa$ used in the Timoshenko beam (first-order shear deformation theory) is not necessary here since the kinematic assumption ensures that the shear strain is zero on the free surfaces and varies parabolically through the $\x{3}$ axis.  Then, 
from \eqref{eq:def_Stiffness_matrix_distributed} we have
\begin{equation}
\begin{array}{rl}
\mathcal{K}(\X) = & \dint_{A}  \begin{bmatrix}
E (\x{3} ^2 -2\alpha\x{3}^4 + \alpha^2 \x{3}^6) & 0 & E \alpha(\x{3}^4 - \alpha \x{3}^6)\\[-2mm]
 0 & G (1 -6\alpha\x{3}^2+9\alpha^2\x{3}^4) & 0 \\[-2mm]
 E \alpha(\x{3}^4 - \alpha \x{3}^6) & 0 & E \alpha^2 \x{3}^6
\end{bmatrix} dA \\[7mm]
=  & \HS{5}\begin{bmatrix}
E (I_2 -2\alpha I_4 + \alpha^2 I_6) & 0 & E \alpha(I_4 - \alpha I_6)\\[-2mm]
 0 & G (I_0 -6\alpha I_2+9\alpha^2I_4) & 0 \\[-2mm]
 E \alpha(I_4 - \alpha I_6) & 0 & E \alpha^2 I_6
\end{bmatrix}
\end{array}
\end{equation}
Considering that there are no distributed inputs, from \eqref{eq:def_dPHS_prop3} we have
\begin{equation}
\begin{bmatrix}
\dot{p}_1 \\[-1.5mm] \dot{p}_2 \\[-1.5mm] \dot{p}_3 \\[-1.5mm] \dot{\q}_1 \\[-1.5mm] \dot{\q}_2 \\[-1.5mm] \dot{\q}_3 
\end{bmatrix} = \left[ \begin{array}{c c c | c c c }
0 & 0 & 0 &  \p_1 & 1 & 0 \\[-1.5mm]
0 & 0 & 0 & 0 & \p_1 & 0  \\[-1.5mm]
0 & 0 & 0 & 0 & 0 & \p_1   \\[-0.5mm]
\hline \\[-6mm]
\p_1 & 0 & 0 & 0 & 0 & 0  \\[-1.5mm]
-1 & \p_1 & 0 & 0 & 0 & 0  \\[-1.5mm]
0 & 0 & \p_1 & 0 & 0 & 0
\end{array} \right] \begin{bmatrix}
e_{p_1} \\[-1.5mm] e_{p_2} \\[-1.5mm] e_{p_3} \\[-1.5mm] e_{\q_1} \\[-1.5mm] e_{\q_2} \\[-1.5mm] e_{\q_3}
\end{bmatrix} 
\label{eq:Reddy_beam}
\end{equation}
with Hamiltonian defined as in \eqref{eq:def_HAMILTONIANO}. The boundary ports are defined from Remark \ref{rem:boundary_ports_N1} by
\begin{equation}
\p_t H = \dint_{\p\Omega} \!\! e_p^\top 1_3 \, e_\q \,\hat{n}_{1}\, d\mathtt{s}
= -e_p(a)^\top e_\q(a)+ e_p(b)^\top e_\q(b) = \dint_{\p \Omega} \! y_\p^\top u_\p \,d\sx
\label{eq:Reddy_beam_BC}
\end{equation}


\subsection{Rayleigh beam}

\noindent Consider a three-dimensional body that can be treated as a one-dimensional structure as shown in Figure \ref{fig:papa_1D}. Since plane sections normal to the neutral line before deformation remain plane and normal to the neutral line after deformation, the displacement field of the Rayleigh beam is given by
\begin{equation}
\textbf{u}(\X,t) = \ub{\begin{bmatrix}
-\x{3} & 0   \\[-2mm] 0 & 0 \\[-2mm] 0 & 1 
\end{bmatrix}}{\lambda_1(\X^c)} \ub{ \begin{bmatrix}
\\[-6mm]
\p_1 w(\X,t) \\[-2mm] w(\X,t) 
\end{bmatrix} }{\textbf{r}(\X,t)}
\label{eq:disp_Rayleigh_beam}
\end{equation}
with $\X = \lbrace \x{1} \rbrace $, $\X^c = \lbrace \x{2},\x{3} \rbrace$, $\Omega = (a,b) \subset \R{}$, $\Omega^c = A \subset \R{2}$, $w(\X,t) \in \R{}$ is the vertical displacement in the direction of the $\x{3}$ axis, and $\p_1 w(\X,t) \in \R{}$ is the slope. From \eqref{eq:def_Mass_matrix_distributed} and \eqref{eq:def_p_distributed} we have
\begin{equation}
\mathcal{M}(\X) = \rho(\X) \dint_{A}\! \begin{bmatrix}
\x{3}^2 & 0 \\[-2mm] 0 &  1
\end{bmatrix} dA = \rho(\X)\begin{bmatrix}
I(\X) & 0 \\[-2mm] 0 & A(\X)
\end{bmatrix}
\quad , \quad
p(\X,t) = \mathcal{M}(\X)\dot{\textbf{r}}(\X,t) = \left[ \!\begin{array}{l} \rho(\X)I(\X) \p_1\dot{w}(\X,t) \\[-1mm] \rho(\X)A(\X) \dot{w}(\X,t)
\end{array} \! \right] 
\end{equation}
where $A(\X)$ is the area of the cross section, and $I(\X)$ is the second moment of inertia of the cross section. From \eqref{eq:def_vec_epsilon}, the only non-zero component of the strain tensor $\varepsilon \subset \vec{\varepsilon}$ is given by $\varepsilon(\CX,t) = \varepsilon_1  = -\x{3}\,\p_1^2 w(\X,t)$, then we choose $m=1$ since there is only one function independent of $\X^c$. From Proposition \ref{prop:2} we have
\begin{equation}
\varepsilon_1(\CX,t) = \ub{ \mfrac{-\x{3}}{2} }{\lambda_2(\X^c)} \ub{ \begin{bmatrix}
\p_1 & \p_1^2
\end{bmatrix} }{\mathcal{F}} \ub{ \begin{bmatrix}
\p_1  w(\X,t) \\[-1mm]  w(\X,t)
\end{bmatrix} }{{\textbf{r}}(\X,t)} \quad \to \quad \q(\X,t) = \mathcal{F}\,{\textbf{r}}(\X,t) = 2\,\p_1^2 w(\X,t)  
\end{equation}
with $\mathcal{F}$ a differential operator of order $N=2$. The stress $\sigma(\CX,t)$ is obtained from $\sigma(\CX,t) = C \, \varepsilon(\CX,t) \in \R{}$, where the constitutive matrix is given by $C = E$ with $E$ the Young's modulus. Then, from \eqref{eq:def_Stiffness_matrix_distributed} we have 
\begin{equation}
\mathcal{K}(\X) = \mfrac{E}{4} \!\dint_{A} \!  \x{3}^2  \,dA =  \mfrac{EI(\X)}{4}  
\end{equation}
Considering distributed inputs and that the work done is given by \eqref{eq:def_WE_2}, from Theorem \ref{theo:PHS_SD} we have
\begin{equation}
\begin{array}{rl}
\begin{bmatrix}
\dot{p}_1 \\[-1mm] \dot{p}_2 \\[-1mm] \dot{\q} 
\end{bmatrix} = & \begin{bmatrix}
0 & 0 & \p_1 \\[-1mm] 0 & 0 & -\p_1^2 \\[-1mm] \p_1 & \p_1^2 & 0
\end{bmatrix} \begin{bmatrix}
e_{p_1} \\[-1mm] e_{p_2} \\[-1mm] e_{\q} 
\end{bmatrix}
\end{array}
\label{eq:Rayleigh_beam}
\end{equation}
with Hamiltonian defined as in \eqref{eq:def_HAMILTONIANO} given by 
\begin{equation}
H[p,\q] = \mfrac{1}{2}\dint_{\Omega} \mfrac{{p}_1^2}{\rho I} + \mfrac{{p}_2^2}{\rho A} +  \mfrac{EI}{4} \q^2 \, d\X
\label{eq:H_Rayleigh_beam}
\end{equation}
The boundary ports are defined from Corollary \ref{cor:boundary_ports} by
\begin{equation}
\begin{array}{rl}
\p_t {H}  = \dint_{\p\Omega} \!\! \mathcal{B}(e_{{p}})^\top  \mathcal{Q}_\p(\mathtt{s}) \mathcal{B}( e_{\q}) \,  d\mathtt{s} = & (  e_{p_1}e_\q + \p_1 e_{p_2}e_\q - e_{p_2}\p_1 e_\q)|_a^b \\
 = & (  2 \p_1 e_{p_2}e_\q - e_{p_2}\p_1 e_\q)|_a^b = \dint_{\p\Omega} y_\p ^\top u_\p \,  d\mathtt{s}
\end{array}
\label{eq:Rayleigh_BC}
\end{equation}
where the second line in \eqref{eq:Rayleigh_BC} is due to $ e_{p_1} = \p_1 e_{p_2}$, which is a consequence of the differential relationship between the components of the displacement vector $\textbf{r}(\X,t)$ in \eqref{eq:disp_Rayleigh_beam}.

\subsection*{D.8$^\star$ \hspace{0.0mm} Euler-Bernoulli beam}

The Euler-Bernoulli beam is a particular case of Rayleigh beam where the effect of rotary inertia $\rho(\X) I(\X)$ is neglected. Then, the Euler-Bernoulli beam can be obtained from the Rayleigh model by eliminating the first equation of \eqref{eq:Rayleigh_beam}, eliminating the first row and column of $\mathcal{J} = \bsm{0 & -\F^*\\ \F & 0}$, and eliminating the contribution of $p_1(\X,t)$ in the Hamiltonian \eqref{eq:H_Rayleigh_beam}.  Also, for convenience we define $\bar{\q}(\X,t) = \frac{1}{2}\q(\X,t) = \p_1 w(\X,t)$, hence $e_{\q}= \frac{EI}{4}\q^2 \equiv e_{\bar{\q}}= EI \bar{\q}^2$. With the above the Euler-Bernoulli beam model is given by
\begin{equation}
\begin{array}{rl}
\begin{bmatrix}
\dot{p}_2 \\[-1mm] \dot{\bar{\q}} 
\end{bmatrix} = & \begin{bmatrix}
 0 & -\p_1^2 \\[-1mm]  \p_1^2 & 0
\end{bmatrix} \begin{bmatrix}
e_{p_2} \\[-1mm] e_{\bar{\q}} 
\end{bmatrix}
\end{array}
\label{eq:EB_beam}
\end{equation}
with Hamiltonian given by 
\begin{equation}
H[p_2,\bar{\q}] = \mfrac{1}{2}\dint_{\Omega} \, \mfrac{{p}_2^2}{\rho A} +  EI \bar{\q}^2 \, d\X
\label{eq:H_EB_beam}
\end{equation}
The boundary ports are defined from Corollary \ref{cor:boundary_ports} considering $\bar{\F} \equiv \p_1 ^2$, then we have
\begin{equation}
\p_t {H} = \dint_{\p\Omega} \!\! \mathcal{B}(e_{{p_2}})^\top  \mathcal{Q}_\p(\mathtt{s}) \mathcal{B}( e_{\bar{\q}}) \,  d\mathtt{s} = (  \p_1 e_{p_2}e_{\bar{\q}} - e_{p_2}\p_1 e_{\bar{\q}})|_a^b = \dint_{\p\Omega} y_\p ^\top u_\p \,  d\mathtt{s}
\end{equation}
which is completely analogous to the expression in \eqref{eq:Rayleigh_BC}.


\subsection{Kirchhoff-Rayleigh plate}

Consider a three-dimensional body that can be treated as a two-dimensional structure as shown in Figure \ref{fig:papa_2D} where $h \in \R{}$ is the thickness (constant). Similarly to the Rayleigh beam, this model is based on the following kinematic assumption: plane sections normal to the neutral line before deformation remain plane and normal to the neutral line after deformation. Then, the displacement field is given by
\begin{equation}
\textbf{u}(\CX,t) = \ub{\begin{bmatrix}
0 & -\x{3} & 0  \\[-2mm] -\x{3} & 0 & 0 \\[-2mm] 0 & 0 & 1 
\end{bmatrix}}{\lambda_1(\X^c)} \ub{ \begin{bmatrix}
\p_2 w(\X,t) \\[-1mm] \p_1 w(\X,t) \\[-1mm]  w(\X,t) 
\end{bmatrix} }{\textbf{r}(\X,t)}
\label{eq:u_KR}
\end{equation}
with $\X = \lbrace \x{1}, \x{2} \rbrace $, $\X^c = \lbrace \x{3} \rbrace$, $\Omega \subset \R{2}$, $\Omega^c = (-\frac{h}{2}\times \frac{h}{2}) \subset \R{}$, $w(\X,t) \in \R{}$ is the vertical displacement in the direction of the axis $\x{3}$, $\p_1 w(\X,t) \in \R{}$ and $\p_2 w(\X,t) \in \R{}$ are the slopes. From \eqref{eq:def_Mass_matrix_distributed} and \eqref{eq:def_p_distributed} we have
\begin{equation}
\mathcal{M}(\X) = \rho(\X) \dint_{-h/2}^{h/2}\!\! \begin{bmatrix}
\x{3}^2 & 0 & 0 \\[-2mm] 0 & \x{3}^2 & 0 \\[-2mm] 0 & 0 & 1
\end{bmatrix} d\x{3} = \rho(\X)\begin{bmatrix}
\bar{I}_2 & 0 & 0 \\[-2mm] 0 & \bar{I}_2 & 0 \\[-2mm] 0 & 0 & \bar{I}_0
\end{bmatrix}
\quad , \quad
p(\X,t) = \mathcal{M}(\X)\dot{\textbf{r}}(\X,t) =  \begin{bmatrix}
\rho(\X)\bar{I}_2\,\p_2\dot{w}(\X,t) \\[-1mm] \rho(\X)\bar{I}_2\,\p_1\dot{w}(\X,t) \\[-1mm] \rho(\X)\bar{I}_0\, \dot{w}(\X,t)
\end{bmatrix}
\end{equation}
where $\bar{I}_i \in \R{}$ with $i=0,2\dots$ is defined as 
$
\bar{I}_i = \int_{-h/2}^{h/2} \x{3}^i \, d\x{3} = \mfrac{h^{i+1}}{2^i(i+1)}
$. From \eqref{eq:def_vec_epsilon}, the non-zero components of the strain tensor $\varepsilon \subset \vec{\varepsilon}$ are given by
\begin{equation}
{\varepsilon}(\CX,t)= \begin{bmatrix}
{\varepsilon}_1(\CX,t) \\[-1mm] {\varepsilon}_2(\CX,t) \\[-1mm] {\varepsilon}_4(\CX,t)
\end{bmatrix} = \begin{bmatrix}
-\x{3}\, [\p_1^2 w(\X,t)]  \\[-1mm] -\x{3}\, [\p_2^2 w(\X,t)] \\[-1mm] -\x{3} \, [\p_2 \p_1 w(\X,t) + \p_1 \p_2 w(\X,t)]
\end{bmatrix}
\label{eq:eps_highlight_KR}
\end{equation}
\noindent From Proposition \ref{prop:2} we choose $m=3$ since there are three functions that are independent of $\X^c$ in the strain vector $\varepsilon(\CX,t)$ (highlighted in square brackets in \eqref{eq:eps_highlight_KR}). Then we have
\begin{equation}
\hspace{-3mm}\varepsilon(\CX,t) = \ub{\begin{bmatrix}
-\x{3} & 0  & 0  \\[-2mm]
 0 & -\x{3} & 0  \\[-2mm]
 0 & 0 & -\x{3} 
\end{bmatrix}
}{\lambda_2(\X^c)}
\ub{ 
\left[ \!\! \begin{array}{c c c }
0 & 0 & \p_1^2 \\[-1mm] 0 & 0 & \p_2^2  \\[-1mm] \p_1 & \p_2 & 0
\end{array} \!\! \right]
}{\F}
\ub{
\begin{bmatrix}
\p_2 w(\X,t) \\[-1mm]
\p_1 w(\X,t) \\[-1mm]
w(\X,t) 
\end{bmatrix}
}{\textbf{r}(\X,t)}
\quad \to \quad
\q(\X,t) = \mathcal{F}\,\textbf{r}(\X,t) = \begin{bmatrix}
\p_1^2 w(\X,t) \\[-1mm] \p_2^2 w(\X,t) \\[-1mm] \p_2 \p_1 w(\X,t) + \p_1 \p_2 w(\X,t) 
\end{bmatrix}
\end{equation}
with $\mathcal{F}$ a differential operator of order $N=2$. 
Considering an isotropic material, the stress $\sigma(\CX,t)$ is obtained from $\sigma(\CX,t) = C_b \, \varepsilon(\CX,t) \in \R{3}$, where $
 C_b = \frac{E}{(1-\nu^2)} \begin{bsmallmatrix}
1 & \nu & 0 \\[-0mm] \nu & 1 & 0 \\[-0mm] 0 & 0 & \frac{(1-\nu)}{2}
\end{bsmallmatrix}
$, 
where $E,\nu$ are Young's modulus and Poisson ratio, respectively. Then, from \eqref{eq:def_Stiffness_matrix_distributed} we have $ \mathcal{K}(\X) = \int_{-h/2}^{h/2}  \x{3}^2\, C_b \,d\x{3} = \bar{I}_2 \,C_b 
$. 
Considering that there are no distributed inputs, from \eqref{eq:def_dPHS_prop3} we have
\begin{equation}
\begin{bmatrix}
\dot{p}_1 \\[-1.5mm] \dot{p}_2 \\[-1.5mm] \dot{p}_3 \\[-1.5mm] \dot{\q}_1 \\[-1.5mm] \dot{\q}_2 \\[-1.5mm] \dot{\q}_3 
\end{bmatrix} = \left[ \begin{array}{c c c | c c c }
0 & 0 & 0 &  0 & 0 & \p_1 \\[-1.5mm]
0 & 0 & 0 & 0 & 0 & \p_2  \\[-1.5mm]
0 & 0 & 0 & -\p_1^2 & -\p_2^2 & 0   \\[-0.5mm]
\hline \\[-6mm]
0 & 0 & \p_1^2 & 0 & 0 & 0  \\[-1.5mm]
0 & 0 & \p_2^2 & 0 & 0 & 0  \\[-1.5mm]
\p_1 & \p_2 & 0 & 0 & 0 & 0
\end{array} \right] \begin{bmatrix}
e_{p_1} \\[-1.5mm] e_{p_2} \\[-1.5mm] e_{p_3} \\[-1.5mm] e_{\q_1} \\[-1.5mm] e_{\q_2} \\[-1.5mm] e_{\q_3}
\end{bmatrix} 
\label{eq:KR_plate}
\end{equation}
with Hamiltonian defined as in \eqref{eq:def_HAMILTONIANO}, that is \\[-1mm]
\begin{equation}
H[p,\q] = \mfrac{1}{2} \dint_{\Omega} \mfrac{ p_1^2}{\rho \bar{I}_2} + \mfrac{ p_2^2}{\rho \bar{I}_2} + \mfrac{ p_3^2}{\rho \bar{I}_0} + \bar{I}_2 C_b\, \q \cdot \q \, d\X
\label{eq:H_KR_plate}
\end{equation}
The boundary ports are defined from Corollary \ref{cor:boundary_ports} by
\begin{equation}
\begin{array}{rl}
\p_t H = &\dint_{\p\Omega} \!\! \mathcal{B}(e_{{p}})^\top  \mathcal{Q}_\p(\mathtt{s}) \mathcal{B}( e_{\q}) \,  d\mathtt{s} =  \dint_{\p\Omega} e_{p_1}\hat{n}_1 e_{\q_3} + e_{p_2}\hat{n}_2 e_{\q_3} - e_{p_3}(\hat{n}_1 \p_1 e_{\q_1} + \hat{n}_2 \p_2 e_{\q_2} ) + \p_1 e_{p_3}\hat{n}_1 e_{\q_1} + \p_2 e_{p_3}\hat{n}_2 e_{\q_2}      \, d\mathtt{s} \\[2mm]
= & 
\dint_{\p\Omega} e_{p_1}( \hat{n}_2 e_{\q_2} + \hat{n}_1 e_{\q_3}) + e_{p_2} (\hat{n}_1 e_{\q_1} + \hat{n}_2 e_{\q_3}) - e_{p_3}(\hat{n}_1 \p_1 e_{\q_1} + \hat{n}_2 \p_2 e_{\q_2} ) \, d\mathtt{s} 
 = \dint_{\p\Omega} y_\p ^\top u_\p \,  d\mathtt{s}
\end{array}
\label{eq:KR_plate_BC}
\end{equation}
where the second line in \eqref{eq:KR_plate_BC} is due to $ \p_1 e_{p_3}= e_{p_2} $, and $  \p_2 e_{p_3} = e_{p_1   }$, which is a consequence of the differential relationships between the components of the displacement vector $\textbf{r}(\X,t)$ in \eqref{eq:u_KR}.

\subsubsection{Tensor representation}

\noindent Consider $ \mathcal{F} = \begin{bsmallmatrix} 0 & \mbox{Grad$\,\circ\,$grad}  \\[0mm]
\mbox{div} & 0 \end{bsmallmatrix} $, $ \mathcal{F}^* = -\begin{bsmallmatrix}
0 & \mbox{grad}  \\[0mm] -\mbox{div$\,\circ\,$Div} & 0 \end{bsmallmatrix}$, $ \textbf{r}(\X,t) = \begin{bsmallmatrix}
\theta(\X,t)  \\[0mm] w(\X,t) \end{bsmallmatrix} $, 
where $\theta(\X,t) = [\p_2 w(\X,t) \;\; \p_1 w(\X,t)]^\top \in \R{2}$ groups the slopes. The mass matrix $\mathcal{M}(\X)$ and stiffness matrix  $ {\mathcal{K}}(\X)$ can be rewritten as 
$
\mathcal{M}(\X) = \rho(\X)\begin{bsmallmatrix}
\bar{I}_2 1_2 & 0  \\[0mm] 0 & \bar{I}_0 
\end{bsmallmatrix}
$, and $
\doubleunderline{\mathcal{K}}(\X) = \doubleunderline{C}(\cdot) 
$, 
where $\doubleunderline{C}(\cdot)= D \left[(1-\nu)\mbox{diag}(\cdot)1_{1 \times 2} \odot 1_2 + \nu \mbox{tr}(\cdot)1_2 \right]$ is a suitable constitutive tensor with $D = \frac{\bar{I}_2 E }{(1-\nu^2)} $ the plate bending stiffness, $1_{1\times 2} = [1 \;\, 1] \in \R{1 \times 2}$, $1_2 \in \R{2 \times 2}$ is the identity matrix, $\mbox{diag}(\cdot)$ is a column vector with the elements of the diagonal of $(\cdot)$, and $\odot$ denotes the element-wise product (Hadamard product). With the above we redefine the energy variables as
\begin{equation}
\hspace{-0mm} p(\X,t) \! = \! \begin{bmatrix}
p_{\theta}(\X,t) \\[-1mm] p_{w}(\X,t) 
\end{bmatrix} \! = \! \begin{bmatrix}
\rho(\X) \bar{I}_2 \dot{\theta}(\X,t) \\[-1mm] \rho(\X)\bar{I}_0 \dot{w}(\X,t) 
\end{bmatrix}  , \;\;
\underline{\q}(\X,t) \! = \! \begin{bmatrix}
\q_{\theta}(\X,t) \\[-1mm] \q_{w}(\X,t) 
\end{bmatrix} \! =  \! \begin{bmatrix}
\mbox{Grad$\,\circ\,$grad}(w(\X,t)) \\[-1mm] \mbox{div}(\theta(\X,t))
\end{bmatrix}
\end{equation}
and co-energy variables as  
\begin{equation}
{e_p}(\X,t) = \begin{bmatrix}
e_{p_\theta}(\X,t) \\[-1mm] e_{p_w}(\X,t)  
\end{bmatrix} = \begin{bmatrix}
\dot{\theta}(\X,t) \\[-1mm] \dot{w}(\X,t) 
\end{bmatrix}
\quad , \quad
\underline{e_\q}(\X,t) = 
\begin{bmatrix}
e_{\q_\theta}(\X,t) \\[-1mm] e_{\q_w}(\X,t) 
\end{bmatrix} =
\begin{bmatrix}
\doubleunderline{C}(\q_\theta(\X,t))  \\[1mm]
\frac{D(1-v)}{2} \q_w(\X,t) 
\end{bmatrix}
\end{equation}
$ $\\[-5mm]
where
$$
\begin{array}{rrrr}
\q_{\theta}(\X,t) = \begin{bmatrix}
\q_1(\X,t) & \frac{1}{2}\q_3(\X,t) \\[0mm] \frac{1}{2}\q_3(\X,t) & \q_2(\X,t) 
\end{bmatrix} , &
\q_{w}(\X,t) = \q_3(\X,t) , & \quad
e_{\q_{\theta}}(\X,t) = \begin{bmatrix}
e_{\q_1}(\X,t) & 0 \\[0mm] 0 & e_{\q_2}(\X,t) 
\end{bmatrix} , &
e_{\q_{w}}(\X,t) = e_{\q_3}(\X,t) 
\end{array}
$$
Then, the model in \eqref{eq:KR_plate} written using tensor notation is given by
\begin{equation}
\begin{array}{r}
\ub{\begin{bmatrix}
\dot{p}_{\theta} \\[-1mm] \dot{p}_{w} \\[0.5mm]
\hline \\[-6mm]
\dot{\q}_{\theta} \\[-1mm] \dot{\q}_{w}
\end{bmatrix}}{\dot{x}} =   \ub{\left[ \!\begin{array}{c c | c c}
0 & 0  & 0  & \mbox{grad}  \\[-1mm]
0 & 0  &-\mbox{div$\,\circ\,$Div}  & 0  \\[-0.5mm]
\hline \\[-6mm]
0  &   \mbox{Grad$\,\circ\,$grad} &  0 & 0  \\[-1mm]
\mbox{div}  &  0 & 0 & 0 
\end{array} \! \right]}{\mathcal{J}=-\mathcal{J}^*} \ub{\begin{bmatrix}
{e}_{p_\theta} \\[-1mm] {e}_{p_w} \\[0.5mm]
\hline \\[-6mm]
e_{\q_{\theta}} \\[-1mm] e_{\q_{w}} 
\end{bmatrix}}{\var_x H} 
\end{array}
\label{eq:KR_plate_tensor}
\end{equation}
with Hamiltonian functional $H[p,\underline{\q}]$ given by 
\begin{equation}
H[p,\underline{\q}] = \mfrac{1}{2} \dint_{\Omega} \!  \mfrac{1}{\rho \bar{I}_2} \, p_\theta \cdot \! p_\theta + \frac{p_w ^2}{\rho \bar{I}_0}  + \doubleunderline{C}(\q_{\theta}) \!:\!\q_{\theta} + \mfrac{D(1-\nu)}{2} \q_w^2 \, d\X
\label{eq:H_KR_plate_tensor}
\end{equation}
\noindent The boundary variables are obtained from the energy balance which is given by
\begin{equation}
\begin{array}{rl}
\p_t H =  \langle \var_x H,\, \dot{x} \rangle_{in}^{\Omega}  = & 
 \dint_{\Omega}  [e_{p_\theta}\!\cdot \mbox{grad}(e_{\q_w}) +  e_{\q_w} \mbox{div}(e_{p_\theta}) ] +  [e_{\q_\theta}\!:\mbox{Grad$\,\circ\,$grad}(e_{p_w}) - e_{p_w} \mbox{div$\,\circ\,$Div}(e_{\q_\theta})] \,  d\X
\end{array}
\end{equation}
so applying the divergence theorem to the first term above, and integration by parts to the second term (see \cite[Remark 6]{brugnoli2019port2}), then we obtain
\begin{equation}
\begin{array}{r}
\p_t H =  \dint_{\p\Omega} \left[ e_{\q_w}\,(e_{p_\theta}\!\cdot \hat{n})  \right] + \left[ (\hat{n} \otimes \mbox{grad}(e_{p_w}))\!:e_{\q_\theta} - \hat{n}\!\cdot \mbox{Div}(e_{\q_\theta})\,e_{p_w} \right] d\mathtt{s} = \dint_{\p\Omega} y_\p^\top u_\p \,d\mathtt{s}
\end{array}
\end{equation}
where $ \hat{n} = [\hat{n}_1 \;\; \hat{n}_2]^\top $, and $\otimes$ denotes the dyadic product. The above expression is completely  analogous to  \eqref{eq:KR_plate_BC}.

\subsection*{D.9$^\star$ \hspace{0.0mm} Kirchhoff-Love plate}

The Kirchhoff-Love plate model is a particular case of the Kirchhoff-Rayleigh plate where the effect of rotary inertia $\rho(\X) \bar{I}_2$ is neglected. Then, the Kirchhoff-Love plate can be obtained from the Kirchhoff-Rayleigh model by eliminating the contribution of $p_1(\X,t)$ and $p_2(\X,t)$ in the Hamiltonian \eqref{eq:H_KR_plate}, and  moving the effect of the first two equations of \eqref{eq:KR_plate} to the third, that is, by eliminating the first two rows and columns of $\mathcal{J} = \bsm{0 & -\F^*\\ \F & 0}$ and redefining the differential operator. 
\begin{equation}
\F = \begin{bmatrix}
0 & 0 & \p_1^2 \\[-1.5mm]
0 & 0 & \p_2^2 \\[-1.5mm]
\p_1 & \p_2 & 0 
\end{bmatrix} 
\quad \to \quad
\bar{\F} = \begin{bmatrix}
 \p_1^2 \\[-1.5mm]
 \p_2^2 \\[-1.5mm]
2\p_1\p_2 
\end{bmatrix}
,\quad
\bar{\F}^* = \begin{bmatrix}
 \p_1^2 &  \p_2^2 & 2\p_1\p_2 
\end{bmatrix}
\end{equation}
With the above the Kircchoff-Love plate model is given by
\begin{equation}
\begin{array}{rl}
\begin{bmatrix}
\dot{p}_3 \\[-1.5mm] \dot{{\q}}_1  \\[-1.5mm] \dot{{\q}}_2  \\[-1.5mm] \dot{{\q}}_3
\end{bmatrix} = & \begin{bmatrix}
 0 & -\p_1^2 & -\p_2^2 & -2\p_1\p_2 \\[-1.5mm]  
\p_1^2 &  0 & 0 & 0  \\[-1.5mm] 
\p_2^2 &  0 & 0 & 0  \\[-1.5mm]
2\p_1\p_2 &  0 & 0 & 0 
\end{bmatrix} \begin{bmatrix}
e_{p_3} \\[-1.5mm] e_{{\q_1}} \\[-1.5mm] e_{{\q_2}} \\[-1.5mm] e_{{\q_3}}
\end{bmatrix}
\end{array}
\label{eq:KL_plate}
\end{equation}
with Hamiltonian given by 
\begin{equation}
H[p_3,\q] = \mfrac{1}{2} \dint_{\Omega} \mfrac{ p_3^2}{\rho \bar{I}_0} + \bar{I}_2 C_b\, \q \cdot \q \, d\X
\label{eq:H_KL_plate}
\end{equation}
Note that $\bar{\F}$ does not belong to the class of differential operators considered in Definition \ref{def:operadores_ND}, therefore the boundary terms cannot be obtained from Corollary \ref{cor:boundary_ports}. Despite this, they can be obtained from the energy balance and applying Green theorem considering the split mixed derivative $2\p_1\p_2 = \p_1\p_2 + \p_2\p_1$, that is 
\begin{equation}
\begin{array}{rl}
\HS{-1}\p_t {H} = & \langle \var_x H,\, \dot{x} \rangle_{in}^{\Omega}  = 
\dint_{\Omega} - e_{p_3}(\p_1^2 e_{\q_1} + \p_2^2 e_{\q_2} + 2\p_1\p_2 e_{\q_3}) + e_{\q_1}\p_1^2 e_{p_3}  + e_{\q_2}\p_2^2 e_{p_3} +  e_{\q_3}2\p_1\p_2 e_{p_3}\, d\X \\
= & \!\!\!\dint_{\p\Omega} \!\!\! \hat{n}_1(e_{\q_1} \p_1 e_{p_3} \!+\! e_{\q_3} \p_2 e_{p_3} \!-\! e_{p_3} \p_1 e_{\q_1}  ) \!+\! \hat{n}_2(e_{\q_2} \p_2 e_{p_3} \!+\! e_{\q_3} \p_1 e_{p_3} \!-\! e_{p_3} \p_2 e_{\q_2}) \!-\! e_{p_3}( \hat{n}_1\p_2 e_{\q_3} \!+\! \hat{n}_2\p_1 e_{\q_3})  \,  d\mathtt{s} 
\end{array}
\label{eq:KL_plate_BC}
\end{equation}
which is different to the expression in \eqref{eq:KR_plate_BC}, since the above \eqref{eq:KL_plate_BC} has the last additional term that arises from the mixed derivatives.

\subsubsection{Tensor representation}

Consider $\bar{\F} = \mbox{Grad$\,\circ\,$grad}$, and $\bar{\F}^* = \mbox{div$\,\circ\,$Div}$. The mass matrix $\mathcal{M}(\X)$ and stiffness matrix  $ {\mathcal{K}}(\X)$ are rewritten as 
$
\bar{\mathcal{M}}(\X) = \rho(\X) \bar{I}_0 
$, and $
\doubleunderline{\mathcal{K}}(\X) = \doubleunderline{\mathcal{D}}(\cdot) 
$, 
where $\doubleunderline{\mathcal{D}}(\cdot)= D \left[(1-\nu)(\cdot) + \nu \mbox{tr}(\cdot)1_2 \right]$ is the internal bending moment constitutive tensor for isotropic plates. With the above consider the following energy variables and co-energy variables 
\begin{equation}
 p_{w}(\X,t) \! = \!  \rho(\X)\bar{I}_0 \dot{w}(\X,t)  
, \quad
\q_{\theta}(\X,t) 
 \! =  \! 
\mbox{Grad$\,\circ\,$grad}(w(\X,t))
, \quad
e_{p_w}(\X,t)  
 =  \dot{w}(\X,t) 
, \quad
\underline{e_{{\q}_\theta}}(\X,t) = 
\doubleunderline{\mathcal{D}}(\q_\theta(\X,t))
\end{equation}
$ $\\[-5mm]
where
$ $\\[-5mm]
$$
\begin{array}{rrrr}
\q_{\theta}(\X,t) = \begin{bmatrix}
\q_1(\X,t) & \frac{1}{2}\q_3(\X,t) \\[0mm] \frac{1}{2}\q_3(\X,t) & \q_2(\X,t) 
\end{bmatrix} 
,& \quad
\underline{e_{{\q}_\theta}}(\X,t)  = \begin{bmatrix}
e_{\q_1}(\X,t) & e_{\q_3}(\X,t) \\[0mm] e_{\q_3}(\X,t) & e_{\q_2}(\X,t) 
\end{bmatrix} 
\end{array}
$$
Then, the model in \eqref{eq:KL_plate} written using tensor notation is given by
\begin{equation}
\begin{array}{r}
\ub{\begin{bmatrix}
\dot{p}_{w} \\[-1mm] \dot{\q}_{\theta} 
\end{bmatrix}}{\dot{x}} =   \ub{\left[ \!\begin{array}{c c}
 0  &-\mbox{div$\,\circ\,$Div}  \\[-0.5mm]
\mbox{Grad$\,\circ\,$grad} &  0  
\end{array} \! \right]}{\mathcal{J}=-\mathcal{J}^*} \ub{\begin{bmatrix}
{e}_{p_w} \\[-1mm] \underline{e_{\q_{\theta}}} 
\end{bmatrix}}{\var_x H} 
\end{array}
\label{eq:KL_plate_tensor}
\end{equation}
with Hamiltonian functional $H[p_w,{\q_\theta}]$ given by 
\begin{equation}
H[p_w,{\q_\theta}] = \mfrac{1}{2} \dint_{\Omega}   \frac{p_w ^2}{\rho \bar{I}_0}  + \doubleunderline{\mathcal{D}}(\q_{\theta}) \!:\!\q_{\theta} \, d\X
\label{eq:H_KL_plate_tensor}
\end{equation}
\noindent The boundary variables are obtained from the energy balance which is given by
\begin{equation}
\begin{array}{rl}
\p_t H =  \langle \var_x H,\, \dot{x} \rangle_{in}^{\Omega}  = & 
 \dint_{\Omega}   \underline{e_{\q_\theta}}\!:\mbox{Grad$\,\circ\,$grad}(e_{p_w}) - e_{p_w} \mbox{div$\,\circ\,$Div}(\underline{e_{\q_\theta}}) \,  d\X
\end{array}
\end{equation}
so applying integration by parts (see \cite[Remark 6]{brugnoli2019port2}), then we obtain
\begin{equation}
\begin{array}{r}
\p_t H =  \dint_{\p\Omega}   (\hat{n} \otimes \mbox{grad}(e_{p_w}))\!:\underline{e_{\q_\theta}} - \hat{n}\!\cdot \mbox{Div}(\underline{e_{\q_\theta}})\,e_{p_w} d\mathtt{s} = \dint_{\p\Omega} y_\p^\top u_\p \,d\mathtt{s}
\end{array}
\end{equation}
where $ \hat{n} = [\hat{n}_1 \;\; \hat{n}_2]^\top $ and $\otimes$ denotes the dyadic product. The above expression is completely  analogous to  \eqref{eq:KL_plate_BC}. For more details see \cite{brugnoli2019port2}.


\bibliographystyle{elsarticle-num}
\bibliography{autosam}           

\end{document}